\documentclass[a4paper,11pt,reqno]{amsart}
\usepackage[centering,includeheadfoot,margin=3cm]{geometry}
\usepackage{graphics,enumitem,epsfig,textcomp}
\usepackage{amsfonts,amsmath,amssymb,euscript,color,mathrsfs}
\usepackage[utf8]{inputenc}	

\usepackage[caption=false,subrefformat=parens,labelformat=parens]{subfig}

\usepackage{url}

\usepackage[symbol]{footmisc}

\usepackage{hyperref}

\numberwithin{equation}{section}

\newtheorem{theorem}{Theorem}[section]

\newtheorem{lemma}{Lemma}[section]
\newtheorem{proposition}{Proposition}[section]
\newtheorem{corollary}{Corollary}[section]

\newtheorem{remark}{Remark}[section]

\renewcommand{\d}{\,\mathrm{d}}
\newcommand{\N}{\mathbb{N}}
\newcommand{\R}{\mathbb{R}}
\newcommand{\Norm}[1]{\left\|#1\right\|}
\newcommand{\Na}[2]{\Norm{#2}_{#1}}
\newcommand{\Ha}[1]{H^1_0(\Omega;#1)}
\newcommand{\E}{\mathcal{E}}
\newcommand{\F}{\mathcal{F}}
\newcommand{\G}{\mathcal{G}}
\newcommand{\J}{\mathcal{J}}
\newcommand{\K}{\mathcal{K}}
\newcommand{\Rext}{\overline{\R}}

\DeclareMathOperator*{\argmin}{arg\,min}

\begin{document}
%%%%%%%%%%%%%%%%

%%%%%%%%%%%%%%%%%%%%%%%%%%%%%%%%%
%% Title
%%%%%%%%%%%%%%%%%%%%%%%%%%%%%%%%%
%\centerline{\Large\textbf{Rigorous connected graphon limit for a model}}
%\centerline{\Large\textbf{of biological transportation networks}}
%\vskip 7mm

\centerline{\Large\textbf{Variational relaxation and gradient flows}}
\centerline{\Large\textbf{for degenerate elliptic conductivity problems}}
\vskip 7mm

%%%%%%%%%%%%%%%%%%%%%%%%%%%%%%%%%
%% Authors
%%%%%%%%%%%%%%%%%%%%%%%%%%%%%%%%%
\centerline{
	{\large Jan Haskovec}\footnote{Mathematical and Computer Sciences and Engineering Division,
		King Abdullah University of Science and Technology,
		Thuwal 23955-6900, Kingdom of Saudi Arabia; 
		{\it jan.haskovec@kaust.edu.sa}}\quad and\quad
	{\large Peter Markowich}\footnote{Faculty of Mathematics, University of Vienna,
	     Oskar-Morgenstern-Platz 1, 1090 Wien, Austria;
		{\it peter.markowich@univie.ac.at}}\quad and\quad
	{\large Stefano Zampini}\footnote{Mathematical and Computer Sciences and Engineering Division,
		King Abdullah University of Science and Technology,
		Thuwal 23955-6900, Kingdom of Saudi Arabia; 
		{\it stefano.zampini@kaust.edu.sa}}	}
\vskip 10mm

\noindent{\bf Abstract.}
We develop a convex-analytic $L^2$-framework for gradient flows of pumping-energy
functionals coupled with elliptic state equations and an evolving
nonnegative conductivity $a\geq 0$. The main difficulty is that the conductivity may be
merely square-integrable, or degenerate,
and formal differentiation need not produce an $L^2$-subgradient.
After imposing the strictly positive lower cutoff
$r+a\geq r>0$, we define the pumping-energy functional by a variational formula,
prove convexity and lower semicontinuity, and show that its $L^2$-subdifferential
is nonempty if and only if the pressure gradient belongs to
$L^4$. We also obtain a unique maximal-monotone $L^2$-gradient flow
satisfying an energy identity. 
Moreover, we prove Mosco convergence of the pumping-energy functionals
as the cutoff vanishes.
An inverse-moment conductivity constraint provides a weak alternative to a uniform positive
lower bound. For nonlinear laws $a=\mu(b)$, primal and flux
representations distinguish two complementary regimes. If $\mu$ is concave,
the pumping-energy functional remains convex, and we characterize its
$L^2$-subdifferential and the induced gradient flow.
If $1/\mu$ is concave, the pumping-energy functional is concave,
and therefore its negative is convex. 
We construct its pointwise gradient flow on a uniformly elliptic box.
Admissible stationary points for power laws yield $p$-Laplacian equations with $p>2$ and
$1 < p < 2$, respectively.
\vskip 7mm

%\allowdisplaybreaks

%
%\subjclass{Primary: 35Q35, 35Q81; Secondary: 35L65, 35K55, 76W05.}

\vskip 2mm

%%35Q92 ? PDEs in connection with biology and other natural sciences
%35B27, % ? Homogenization in context of PDEs
%49J45, % ? Methods involving semicontinuity and convergence; relaxation
%49Q22,  % - Optimal transportation
%92C42  % -  Systems biology, networks

\vspace{7mm}

%\maketitle \centerline{\date}

\tableofcontents

%%%%%%%%%%%%%%%%%%%%%%%%%%%%%%%%%
%% Introduction
%%%%%%%%%%%%%%%%%%%%%%%%%%%%%%%%%
\section{Introduction}\label{sec:Intro}

\subsection{Model, motivation, and objectives}

Adaptive transportation networks must reconcile two competing demands: they
should move material with little dissipation, but the conducting structure is
itself costly to build and maintain. In the basic modeling framework,
a sufficiently regular, uniformly positive scalar conductivity $a=a(x)$ on a bounded Lipschitz domain $\Omega\subset\R^d$
determines a pressure $u=u[a]$ from the Dirichlet problem
\begin{equation}  \label{eq:E0ell}
   \begin{aligned}
      -\operatorname{div}(a\nabla u) &= f &&\text{in }\Omega,\\
      u&=0 &&\text{on }\partial\Omega,
   \end{aligned}
\end{equation}
where $f\in L^2(\Omega)$ is the prescribed source--sink distribution. With
the Darcy convention, $-a\nabla u$ is the flux and $a|\nabla u|^2$ is the
local dissipated power. In this classical setting, the associated energy is
\begin{equation}   \label{def:E0}
   \E(a):=\int_\Omega a|\nabla u[a]|^2\,\d x
      +\int_\Omega g(a)\,\d x.
\end{equation}
The first term is the pumping-energy contribution; it represents the total
power dissipated by the stationary Darcy flow.
The second term models the metabolic maintenance cost of the transportation structure.
Convex power costs $g(s)=s^\gamma/\gamma$,
$\gamma\geq1$, are a
principal example. The competition between these terms drives the
conductivity toward efficient configurations and makes its $L^2$-gradient
flow a natural relaxation mechanism for finding them.

The goal of this paper is to address the following three questions.
First, how should this pumping-energy functional be defined when the coefficient is unbounded or may vanish,
so that the classical $H^1_0$ formulation of \eqref{eq:E0ell} is no longer available?
Second, when does the relaxed energy generate a well-posed $L^2$-gradient flow, and when
can that abstract flow be identified with a pointwise evolution law? Third,
what changes when the physical conductivity is a nonlinear function
$a=\mu(b)$ of the evolving variable? The last question is not a mere formal
reparametrization: it changes both the convexity of the pumping-energy
functional and the metric seen by the evolution. Our purpose is to answer these questions in a
single variational framework.

\subsection{The variational viewpoint and the source of degeneracy}

For a bounded coefficient satisfying $a\geq r>0$, problem
\eqref{eq:E0ell} is the Euler--Lagrange equation for the functional
\begin{equation}   \label{def:Ja}
   \J_a(\varphi):=2\int_\Omega f\varphi\,\d x
      -\int_\Omega a|\nabla\varphi|^2\,\d x.
\end{equation}
If $u$ is the solution of \eqref{eq:E0ell}, testing by $u$ gives
$\int_\Omega a|\nabla u|^2\,\d x=\int_\Omega fu\,\d x$ and hence
\begin{equation}\label{eq:intro-pumping-energy}
   \int_\Omega a|\nabla u|^2\,\d x
   =\sup_{\varphi\in H^1_0(\Omega)}\J_a(\varphi).
\end{equation}
The decisive feature of \eqref{eq:intro-pumping-energy} is that the right-hand
side is a supremum of affine, order-decreasing functions of $a$. Replacing
$H^1_0(\Omega)$ by $C_0^\infty(\Omega)$ therefore gives a
convex lower semicontinuous extended functional on $L^2(\Omega)$ even when
the classical solution in \eqref{eq:E0ell} is unavailable.
This construction is the central relaxation device used throughout the paper.

Several difficulties are hidden by the uniformly elliptic, bounded setting.
For $a\in L^2(\Omega)$, the weighted quadratic form need not be finite on all
of $H^1_0(\Omega)$; if $a$ vanishes, ordinary coercivity is also lost. Its
natural energy space consequently depends on $a$, which complicates both
compactness and differentiation with respect to the coefficient. Moreover,
the formal derivative of the pumping-energy functional is
$-|\nabla u[a]|^2$, an $L^2$ subgradient only when
$\nabla u[a]\in L^4(\Omega)$. The variational
formulation separates what follows from convexity alone from what needs this
additional elliptic regularity.

For a nonlinear conductivity $a=\mu(b)$ there is a second issue: composition
of the pumping-energy functional with $\mu$ does not preserve convexity in
general. Even in the special case in which the metabolic term is transformed
as well, so that
the full energy is $\E\circ\mu$, its Euclidean $L^2$-flow is not simply the
flow of $\E$ written in another coordinate. Formally, the Euclidean $b$-flow
has velocity
$b_t=-\mu'(b)\partial_a\E(\mu(b))$, whereas transforming
$a_t=-\partial_a\E(a)$ gives
$b_t=-\partial_a\E(\mu(b))/\mu'(b)$ wherever $\mu'(b)\ne0$.
The transformed flow is instead the gradient flow of $\E\circ\mu$ for the
pullback metric with tensor $G(b)=\mu'(b)^2$. This distinction motivates our
direct study of the energy in the variable $b$.

\subsection{Relation to earlier work}

The pumping--metabolic competition considered here is rooted in discrete
models of adaptive biological transportation networks. Hu and Cai formulated
the optimization principle and a conductivity adaptation driven by pressure
differences
\cite{HuCaiPRL,HuCaiCMS}. Continuum network equations and their analytical
properties were developed in \cite{HMP15,HKM:CMS:2019}; a rigorous
mesh-refinement limit on rectangular grids was obtained in
\cite{HKM:CommPDE:2019}, and tensor-valued extensions were studied in
\cite{HMP:CMS:2022}.
A closely related scalar functional was studied in
\cite{Zampini}, where it arises as a continuum limit of sparse (planar) networks and
its relaxation dynamics are related to a superquadratic $p$-Laplacian.
An extension of the problem which accounts for robustness of the transportation network
was introduced in \cite{HV:2024}.
Dense-graph limits lead to nonlocal graphon energies \cite{AH:2026, AH:arxiv}. 
Conductivity--pressure couplings also occur in
continuum Physarum dynamics aimed at Monge--Kantorovich equilibria
\cite{FaccaCardinPutti}; the evolution and dissipation mechanism there are
different from the Hilbert-space subgradient flow considered here.

The present paper develops an analytical relaxation theory of the local scalar
functional beyond the positive-background, linearly parametrized setting:
we admit genuinely degenerate coefficients in the elliptic equation,
introduce nonlinear conductivity maps, and identify mechanisms leading to both the
superquadratic and subquadratic $p$-Laplacian regimes.

There is also a close static connection with optimal-design problems and
optimal material or mass distribution, where relaxed or measure-valued
formulations are fundamental
\cite{KohnStrangI,BouchitteButtazzo,ButtazzoGelliLucic}. Our emphasis is
different: the ambient space is $L^2(\Omega)$, the principal object is the
subgradient evolution, and the coefficient-dependent elliptic problem is
incorporated through a common-core dual formula. Weighted Sobolev spaces and
the approximation of degenerate weights provide the natural elliptic
background for this construction \cite{Zhikov}.

\subsection{Main results and analytical methods}

Here we summarize the contributions of this paper in the order in which they are proved.
%The assumptions mentioned below are part of the statements; in particular,
%we do not claim an unconditional pointwise flow theory for every nonlinear mobility.

\begin{enumerate}[label=\textup{(\roman*)},leftmargin=*,itemsep=0.5em]
\item
For the strong cutoff $r>0$, we replace $a$ in \eqref{eq:E0ell} by $r+a$.
For every $a\in L^2_+(\Omega)$, equation \eqref{eq:E0ell} has a unique weighted energy
solution, and the common-core representation makes the resulting
pumping-energy functional $\F_r$ proper, convex, and lower semicontinuous on
$L^2(\Omega)$. We prove the dichotomy
\[
   \partial\F_r(a)=
   \begin{cases}
      -|\nabla u_r[a]|^2+N_+(a),&\nabla u_r[a]\in L^4(\Omega),\\
      \emptyset,&\nabla u_r[a]\notin L^4(\Omega).
   \end{cases}
\]
Here $N_+(a)$ denotes the normal cone to the nonnegative cone $L^2_+(\Omega)$
at $a$. After adding $\Gamma(a)=\int_\Omega g(a)\,\d x$, with
$g:[0,+\infty)\to[0,+\infty)$ convex and continuous at zero,
maximal-monotone operator theory gives a unique global $L^2$-gradient flow
for every nonnegative $L^2$ initial datum. Finite-energy data give a strong
solution and exact energy identity. If, in addition,
$g\in C^1([0,+\infty))$ and the two densities $|\nabla u_r[a]|^2$ and
$g'(a)$ belong separately to $L^2(\Omega)$, the abstract inclusion agrees
with the projected pointwise law.

\item
Convex conjugacy eliminates the conductivity from the stationary problem
and produces a pressure Hamiltonian. % involving $g_+^*(|\nabla u|^2)$.
Under the stated measurability, integrability, and saddle-point attainment
conditions, this yields the Fenchel relation between pressure and
conductivity. Linear costs produce a gradient constraint, while power costs
produce regularized $p$-Laplacian equations with $p>2$. Whenever the set of
minimizers is nonempty, the regularized gradient-flow trajectory converges
weakly in $L^2$ to a minimizer; uniqueness of the minimizer identifies the
limit.

\item
Next, we remove the strong cutoff at the level of the pumping-energy functional.
In the conductivity variable $b=r+a$, the lower-bound restricted functionals
Mosco-converge as $r\to 0^+$ to the common-core relaxation of
the pumping energy. This statement concerns the pumping-energy functionals themselves and is
not used to assert convergence of their gradient flows. For $d>1$, we also
introduce the weak cutoff $\int_\Omega a^{-d/2}\,\d x\leq K$. We prove that this
set is closed and convex and that it supplies a weighted
Sobolev--Poincar\'e inequality, hence existence and uniqueness of the
degenerate energy solution.

\item
For nonlinear mobilities, a formal second-variation calculation first
reveals the relevant curvature condition in the smooth uniformly elliptic
setting. If $\mu$ is concave, order monotonicity of the relaxed pumping-energy
functional then gives a rigorous convexity argument without differentiating the
pressure. In the degenerate case $\mu(0)=0$, we work on the inverse-moment
admissible set, establish lower semicontinuity, and compute directional
derivatives along comparable directions. Moreover, we characterize the full $L^2$ subdifferential and
the correponding pointwise gradient flow.
%under explicit density, chain-rule, integrability, Slater, tangent-cone, and subdifferential sum-rule assumptions.
For $0<\sigma\leq1<\lambda$, the choices $\mu(s)=s^\sigma/\sigma$ and $g(s)=s^\lambda/\lambda$ yield, at an admissible
stationary point, the $p$-Laplacian with $p=2\lambda/(\lambda-\sigma)>2$.

\item
The complementary regime is detected by the flux representation of the
pumping-energy functional. Under suitable assumptions,
concavity of $1/\mu$ implies concavity of the pumping-energy functional $I_\mu$
on the admissible set and hence convexity of  the extended
sign-reversed energy $-\frac12 I_\mu+\Gamma$.
A a direct flux certificate proves global minimality of supplied admissible stationary pairs.
For $\lambda>\kappa>0$, under the further hypotheses of Section~\ref{sec:ReciprocalConcave},
the inverse-power example $\mu(s)=s^{-\kappa}/\kappa$ leads formally at an admissible stationary point
to the $p$-Laplacian with $p=2\lambda/(\lambda+\kappa)\in(1,2)$.
A separate uniformly elliptic box restriction recovers an abstract $L^2$-subgradient flow, with the constraint
contribution retained whenever a pointwise identification is justified.
\end{enumerate}

The common tools behind these results are weighted energy completions, dual
representations of the pumping-energy functional, convex integral functionals,
normal-cone and
subdifferential calculus, the Br\'ezis--Komura theory of maximal-monotone
evolutions, Fenchel duality, and Mosco convergence. In the reciprocal
curvature regime, the additional tool is a flux-minimization
representation, which turns the coefficient dependence into an integral
involving $1/\mu(b)$.

\subsection{Organization and notation}

Section~\ref{sec:StrongCutoff} treats the strictly positive regularization,
its $L^2$-gradient flow, stationary states, and the degenerate Mosco limit.
Section~\ref{sec:WeakCutoff} develops the inverse-moment weak cutoff.
Section~\ref{sec:NonlinearMobility} derives the formal curvature tests for a
general nonlinear mobility. Section~\ref{sec:ConcaveMobility} studies concave
mobilities and the superquadratic power-law examples, while
Section~\ref{sec:ReciprocalConcave} treats concave reciprocal mobilities and
the subquadratic examples.

We write $\Rext :=\R\cup\{+\infty\}$. For $q\in[1,+\infty]$, we set
\[
   L^q_+(\Omega):=\{v\in L^q(\Omega):v\geq0\ \text{a.e. in }\Omega\}.
\]

%%%%%%%%%%%%%%%%%%%%%%%%%%%%%%%%%
%% Strong cutoff
%%%%%%%%%%%%%%%%%%%%%%%%%%%%%%%%%
\section{The regularized problem---strong cutoff}\label{sec:StrongCutoff}
We begin with a strictly positive background conductivity. This is the
setting in which the variational strategy can be developed without loss of
coercivity, while still allowing the variable part of the coefficient to be
unbounded. The results of this section provide the reference problem for the
degenerate and nonlinear theories that follow. We first construct the
weighted energy solution, then establish the convexity and lower
semicontinuity of the pumping-energy functional, identify its $L^2$
subdifferential, and finally generate the corresponding gradient flow.

For $r>0$, consider the regularized version of \eqref{eq:E0ell},
\begin{equation}  \label{eq:ur}
\begin{aligned}
 -\operatorname{div}\left((r+a)\nabla u\right) = f \qquad \text{in } & \Omega,\\
   u = 0 \qquad \text{on } & \partial\Omega,
 \end{aligned}
\end{equation}
with given $a\in L^2_+(\Omega)$ and regularization parameter $r>0$.

\subsection{The weighted energy solution}

Since the coefficient $r+a$ in \eqref{eq:ur} is, in general, not bounded from above,
the natural state space for the solution is the completion for the weighted
Dirichlet norm. % rather than all of $H^1_0(\Omega)$.
In that space the pressure is characterized as follows.

For fixed $r>0$ and $a\in L^2_+(\Omega)$, equip $C_0^\infty(\Omega)$ with the
scalar product
\[
   (\varphi,\psi)_{r+a}:=
   \int_\Omega(r+a)\nabla\varphi\cdot\nabla\psi\,\d x
\]
and denote by $\Ha{r+a}$ its completion with respect to the induced norm
$\Na{r+a}{\varphi}:=(\varphi,\varphi)_{r+a}^{1/2}$. %Since $r+a\geq r>0$,
Obviously, this space embeds continuously into $H^1_0(\Omega)$.

\begin{lemma}\label{lem:ur}
Let $r>0$ and $f\in L^2(\Omega)$. Then, for every $a\in L^2_+(\Omega)$,
the Dirichlet problem \eqref{eq:ur} has a unique energy solution
$u=u_r[a]\in\Ha{r+a}$, %\hookrightarrow H^1_0(\Omega)$, 
characterized by
\begin{equation}\label{eq:urweak-statement}
   \int_\Omega (r+a)\nabla u_r[a]\cdot\nabla\varphi\,\d x
   =\int_\Omega f\varphi\,\d x
   \qquad\text{for every }\varphi\in\Ha{r+a}.
\end{equation}
Moreover, there exists a constant $C=C(\Omega, r)$, independent of $a$, such that
\begin{equation}    \label{est:ur}
   \Norm{\nabla u_r[a]}_{L^2(\Omega)} \leq C \Norm{f}_{L^2(\Omega)}\qquad
   \mbox{and}\qquad
   \Norm{\sqrt{a} \, \nabla u_r[a]}_{L^2(\Omega)} \leq C \Norm{f}_{L^2(\Omega)}.
\end{equation}
\end{lemma}

\begin{proof}
For every $\varphi\in\Ha{r+a}$, the Poincar\'e inequality and the lower bound
$r+a\geq r$ give
\begin{align*}
   \left| \int_\Omega f \varphi \,\d x \right|
      &\leq \Norm{f}_{L^2(\Omega)} \Norm{\varphi}_{L^2(\Omega)} \\
      &\leq C_\Omega \Norm{f}_{L^2(\Omega)} \Norm{\nabla\varphi}_{L^2(\Omega)} \\
      &\leq \frac{C_\Omega}{\sqrt{r}}\Norm{f}_{L^2(\Omega)}\Na{r+a}{\varphi}.
\end{align*}
Here $C_\Omega$ is the Poincar\'e constant associated with the domain $\Omega$.
Consequently, the map $\varphi\mapsto\int_\Omega f\varphi\,\d x$ is a continuous linear functional
on $\Ha{r+a}$. The Riesz representation theorem then provides a unique
$u=u_r[a]\in\Ha{r+a}$ such that
\begin{equation}   \label{eq:urweak}
    (u,\varphi)_{r+a} =  \int_\Omega f \varphi \,\d x \qquad\mbox{for all } \varphi\in \Ha{r+a}.
\end{equation}
Thus $u=u_r[a]$ is the energy solution of the Dirichlet problem \eqref{eq:ur} in the sense of
\eqref{eq:urweak-statement}. In particular, the equation holds distributionally because
$C_0^\infty(\Omega)\subset\Ha{r+a}$.

Moreover, choosing $\varphi:=u$ in \eqref{eq:urweak} gives
\[
    \Na{r+a}{u}^2  = \int_\Omega f u \,\d x \leq \frac{C_\Omega}{\sqrt{r}}  \Norm{f}_{L^2(\Omega)}  \Na{r+a}{u},
\]
so that
\[
   \Na{r+a}{u}  \leq \frac{C_\Omega}{\sqrt{r}}  \Norm{f}_{L^2(\Omega)},
\]
from which \eqref{est:ur} directly follows.
\end{proof}

\subsection{The pumping-energy functional and its subdifferential}

With the pressure well defined, we now eliminate the elliptic constraint and
view the pumping energy as a functional of the conductivity alone. Define
$\F_r:L^2(\Omega)\to\Rext$ by
\begin{equation}   \label{def:Fr}
   \F_r(a) :=
\begin{cases}
   \displaystyle\int_{\Omega} (r+a)\, |\nabla u_r[a]|^2\, \d x,
& \text{if } a\ge 0 \text{ a.e. in } \Omega,\\[3mm]
+\infty, & \text{otherwise,}
\end{cases}
\end{equation}
where $u_r[a] \in \Ha{r+a}$ is the unique energy solution of \eqref{eq:ur}
provided by Lemma \ref{lem:ur}.

The extension by $+\infty$ incorporates the nonnegativity constraint directly
into the functional.
We now derive a variational representation of \eqref{def:Fr}.

\begin{lemma}\label{lem:Frrepr}
For any $a\in L^2_+(\Omega)$ we have
\begin{equation}   \label{eq:Frrepr}
     \F_r(a)   = \sup_{\varphi\in C_0^\infty(\Omega)}   \left(2\int_\Omega f\varphi\,\d x   -\int_\Omega(r+a)|\nabla\varphi|^2\,\d x\right).
\end{equation}
\end{lemma}

\begin{proof}
For $a\in L^2_+(\Omega)$, the weak equation \eqref{eq:ur} is the Euler equation of the strictly concave functional
\[
   F_a(\varphi) :=  2\int_\Omega f\varphi\,\d x  -  \int_\Omega(r+a)|\nabla\varphi|^2\,\d x.
\]
We therefore have
\begin{equation} \label{eq:FrFa}
    \F_r(a)  = \sup_{\varphi\in \Ha{r+a}}  F_a(\varphi) \geq \sup_{\varphi\in C_0^\infty(\Omega)}  F_a(\varphi).
\end{equation}
Now, take $\varphi\in \Ha{r+a}$. Since $\Ha{r+a}$ is the completion of $C_0^\infty(\Omega)$ with respect to
the norm $\Na{r+a}{\cdot}$, there exists a sequence $(\varphi_n)_{n\in\N} \subset C_0^\infty(\Omega)$
converging to $\varphi$ in $\Ha{r+a}$. The continuity of the norm gives
\[
   \int_\Omega (r+a) |\nabla\varphi_n|^2 \d x \to  \int_\Omega (r+a) |\nabla\varphi|^2 \d x \qquad\mbox{as } n\to\infty.
\]
Since the functional $\varphi \mapsto \int_\Omega f\varphi\,\d x$ is continuous, we readily have
\[
   F_a(\varphi) = \lim_{n\to\infty} F_a(\varphi_n) \leq \sup_{\psi\in C_0^\infty(\Omega)} F_a(\psi),
\]
and thus
\[
   \sup_{\varphi\in \Ha{r+a}}  F_a(\varphi) \leq \sup_{\varphi\in C_0^\infty(\Omega)}  F_a(\varphi).
\]
Combined with \eqref{eq:FrFa} this yields \eqref{eq:Frrepr}.
\end{proof}

We stress that although the space $\Ha{r+a}$ depends on $a$,
the test functions $\varphi\in C_0^\infty(\Omega)$ in the representation \eqref{eq:Frrepr} do not.
This allows us to directly establish convexity and lower semicontinuity
of $\F_r$ in the ambient $L^2$ topology.

\begin{lemma}\label{lem:FrconvexLSC}
The functional $\F_r: L^2(\Omega) \to \Rext$ defined by \eqref{def:Fr} is convex and lower semicontinuous
on $L^2(\Omega)$.
\end{lemma}

\begin{proof}
For each fixed $\varphi\in C_0^\infty(\Omega)$, the expression in parentheses in \eqref{eq:Frrepr}
is continuous and affine as a function of $a\in L^2_+(\Omega)$. Hence its supremum is
convex and lower semicontinuous. Adding the indicator function of the closed convex cone $L^2_+(\Omega)$
preserves convexity and lower semicontinuity and yields the extended functional \eqref{def:Fr}.
\end{proof}

The same variational identity also has a useful dual interpretation. For
fixed $r>0$ and $a\in L^2_+(\Omega)$, define
$\mathcal H_{r+a}:L^2(\Omega)\to\Rext$ by
\[
   \mathcal{H}_{r+a} (\psi):=
   \begin{cases}
     \displaystyle \frac12 \int_\Omega(r+a)|\nabla\psi|^2\,\d x,
        & \text{if }\psi\in\Ha{r+a},\\[3mm]
     +\infty, & \text{otherwise.}
   \end{cases}
\]
Then $\frac12\F_r(a)$ is the $L^2$-Legendre dual of
$\mathcal{H}_{r+a}$ evaluated at $f$. Thus the representation \eqref{eq:Frrepr}
and the Legendre-dual formulation describe the same
elimination of the state variable.

Convexity and lower semicontinuity are sufficient for an abstract evolution
theory. To recover a pointwise evolution law, however, we must determine when
the pumping-energy functional has an $L^2$ subgradient. The normal-cone
contribution from the nonnegativity constraint is described, for
$a\in L^2_+(\Omega)$, by
\[
   N_+(a):=\left\{w\in L^2(\Omega):
      w\leq0\ \text{a.e. on }\{a=0\},\quad
      w=0\ \text{a.e. on }\{a>0\}\right\}.
\]

\begin{lemma}\label{lem:SubDiffFr}
The subdifferential of $\F_r: L^2(\Omega) \to \Rext$ at $a\in L^2_+(\Omega)$ is given by
\[
   \partial \F_r(a)=- |\nabla u_r[a]|^2+N_+(a)
\]
if $\nabla u \in L^4(\Omega)$, where $u=u_r[a]$ is the unique solution of \eqref{eq:ur}.
If $\nabla u \notin L^4(\Omega)$, then $\partial \F_r(a) = \emptyset$.
\end{lemma}

\begin{proof}
We first compute the directional derivative along feasible directions
generated by bounded nonnegative competitors. We then identify the limit of
the corresponding state solutions and finally localize the subgradient
inequality.

Let us fix $a_0\in L^2_+(\Omega)$ and for $q\in [1,+\infty]$ define the set of admissible directions
\[
    S^q(a_0) = %\bigcup_{\lambda>0} \lambda\left( L^2_+(\Omega) - \{a_0\} \right),
       \bigcup_{\lambda>0} \left\{ \lambda(a-a_0);\, a\in L^q_+(\Omega) \right\}.
\]
%Then $S^2(a_0)$ is the domain of the directional derivative
Consider the directional derivative
\[
   D\F_r(a_0)[h] =  \lim_{t\to 0+} \frac{ \F_r(a_0 + th) - \F_r(a_0)}{t}. 
\]
Note that since $\F_r$ is convex by Lemma \ref{lem:FrconvexLSC}, $D\F_r(a_0)[h]$
exists in $\R \cup \{\pm\infty\}$ for all $h\in L^2(\Omega)$, is a convex sublinear functional
and satisfies $D\F_r(a_0)[0] = 0$.

Let us now choose a sequence $(t_n)_{n\in\N}$ with $t_n\to0+$ as
$n\to+\infty$. Fix $h\in S^\infty(a_0)$, set $u_0:=u_r[a_0]$, and denote by
$u_n:=u_r[a_0+t_nh]\in\Ha{r+a_0+t_nh}$ the energy solution of
\eqref{eq:ur} corresponding to the conductivity $a_0+t_nh$.
We then have
\begin{align*}
   \frac{ \F_r(a_0 + t_n h) - \F_r(a_0)}{t_n}
      &= \int_\Omega h |\nabla u_n|^2 \d x
       + \int_\Omega (r+a_0)
         \frac{|\nabla u_n|^2 - |\nabla u_0|^2}{t_n} \d x \\
      &= \int_\Omega h |\nabla u_n|^2 \d x
       + \int_\Omega (r+a_0) (\nabla u_n + \nabla u_0) \cdot
         \frac{\nabla u_n - \nabla u_0}{t_n} \d x.
\end{align*}
By the definition of $S^\infty(a_0)$, we may write
$h=\lambda_0(y_0-a_0)$ with $\lambda_0>0$ and
$y_0\in L^\infty_+(\Omega)$. For sufficiently large $n$,
\[
   a_0+t_nh=(1-\lambda_0t_n)a_0+\lambda_0t_ny_0.
\]
%This identity, together with the boundedness of $y_0$ and the lower bound
%$r>0$, makes the two weighted energy norms comparable. Hence $u_n+u_0$ is an admissible cross-test.
Subtracting the weak equations for $u_n$ and $u_0$,
dividing by $t_n$, and testing by $u_n+u_0$, we obtain
\[
   \int_\Omega h \nabla u_n\cdot \nabla (u_n+u_0) \d x + \int_\Omega (r+a_0) \nabla(u_n + u_0) \cdot \frac{\nabla(u_n-u_0)}{t_n} \d x = 0.
\]
Consequently, we have
\[
   \frac{ \F_r(a_0 + t_n h) - \F_r(a_0)}{t_n} =  - \int_\Omega h \nabla u_n\cdot\nabla u_0 \,\d x.
\]
Estimate \eqref{est:ur} from Lemma~\ref{lem:ur}, together with the
Poincar\'e inequality, yields uniform boundedness of $(u_n)$ in
$H^1_0(\Omega)$.
Therefore, every subsequence of $u_n$ has a subsequence, not relabelled, such that
$u_n \rightharpoonup \tilde u$ weakly in $H^1_0(\Omega)$ as $n\to+\infty$.
To identify the limit, we use $\varphi\in C_0^\infty(\Omega)$ as a test function in \eqref{eq:ur} for $u_n$,
\[
   t_n \int_\Omega h\nabla u_n \cdot \nabla \varphi \,\d x + \int_\Omega (r+a_0) \nabla u_n \cdot \nabla \varphi \, \d x = \int_\Omega f \varphi \,\d x,
\]
and pass to the limit as $n\to+\infty$,
\[
   \int_\Omega (r+a_0) \nabla \tilde u \cdot \nabla \varphi \, \d x = \int_\Omega f \varphi \,\d x.
\]
The same convex-combination identity gives
$a_0\leq(1-\lambda_0t_n)^{-1}(a_0+t_nh)$ for all sufficiently large $n$.
Hence $(u_n)$ is bounded in $\Ha{r+a_0}$, and the weak limit belongs to this
energy space.
By uniqueness in $\Ha{r+a_0}$, we conclude that $\tilde u = u_0$.
Since every weakly convergent subsequence has the same limit, the whole
sequence converges weakly to $u_0$ in $H^1_0(\Omega)$.

Returning to the representation $h=\lambda_0(y_0-a_0)$, we find
\[
   \frac{ \F_r(a_0 + t_n h) - \F_r(a_0)}{t_n} = - \lambda_0 \int_\Omega (y_0 - a_0) \nabla u_n \cdot \nabla u_0 \d x.
\]
Using $u_n$ as a test function in \eqref{eq:ur} for $u_0$, we have
\begin{align*}
   \lim_{n\to+\infty} \int_\Omega (r+a_0) \nabla u_0 \cdot \nabla u_n \, \d x
     &= \lim_{n\to+\infty} \int_\Omega f u_n \,\d x \\
     &= \int_\Omega f u_0 \,\d x \\
     &= \int_\Omega (r+a_0)|\nabla u_0|^2\,\d x.
\end{align*}
Consequently, for every $h\in S^\infty(a_0)$,
\begin{align*}
   \lim_{n\to+\infty} \frac{ \F_r(a_0 + t_n h) - \F_r(a_0)}{t_n}
     &= - \lambda_0 \int_\Omega y_0 |\nabla u_0|^2 \d x
        + \lambda_0 \int_\Omega a_0 |\nabla u_0|^2 \d x \\
     &= -\int_\Omega h |\nabla u_0|^2 \d x.
\end{align*}
The directions generated by bounded competitors are dense in the contingent cone
\[
   T_{L^2_+(\Omega)}(a_0) = \left\{ h\in L^2(\Omega);\, h(x)\geq 0 \text{ almost everywhere on } \{a_0 = 0\} \right\}.
\]
The functional
\[
    h \mapsto - \int_{\Omega} h\, |\nabla u_0|^2 \, \d x
\]
with $u_0 = u_r[a_0]$ has a continuous extension from
$S^\infty(a_0)$ to $T_{L^2_+(\Omega)}(a_0)$ if and only if
$|\nabla u_0|^2 \in L^2(\Omega)$, i.e., if and only if $\nabla u_0 \in L^4(\Omega)$.
If $y\in\partial\F_r(a_0)$, the subgradient inequality tested against
directions generated by bounded competitors gives
\[
   \int_\Omega yh\,\d x\leq D\F_r(a_0)[h]
   \qquad\text{for all }h\in S^\infty(a_0).
\]
Set $z:=y+|\nabla u_0|^2$. For every $c\in L^\infty_+(\Omega)$, the preceding
inequality applied to $h=c-a_0$ gives
\[
   \int_\Omega z(c-a_0)\,\d x\leq0.
\]
Here $z\in L^1(\Omega)$ and $za_0\in L^1(\Omega)$, because
$y,a_0\in L^2(\Omega)$ and
$a_0|\nabla u_0|^2\in L^1(\Omega)$. Choosing arbitrarily large bounded
competitors supported where $z>0$ shows that $z\leq0$ almost everywhere. The
choice $c=0$ then gives $\int_\Omega za_0\,\d x\geq0$; since $z\leq0$ and
$a_0\geq0$, it follows that $za_0=0$ almost everywhere.
Therefore,
\[
   y = -|\nabla u_0|^2
   \qquad \text{on } \{a_0>0\},
\]
and
\[
   y \leq -|\nabla u_0|^2
   \qquad \text{on } \{a_0=0\}.
\]
Equivalently, $y+|\nabla u_0|^2\in N_+(a_0)$. This makes the role of the
$L^4$ regularity transparent.
Indeed, if $|\nabla u_0|^2\notin L^2(\Omega)$, this condition cannot be
satisfied by any $y\in L^2(\Omega)$. If $|\nabla u_0|^2\in L^2(\Omega)$, the
converse inclusion follows from the variational representation
\eqref{eq:Frrepr}, approximating $u_0$ in $W^{1,4}_0(\Omega)$. We conclude
\[
   \partial \F_r(a_0)=- |\nabla u_r[a_0]|^2+N_+(a_0).
\]
\end{proof}

We next incorporate the local metabolic cost. Let
$g:[0,+\infty)\to\R$ be convex and continuous at zero, and define
$\Gamma:L^2(\Omega)\to\Rext$ by
\begin{equation}\label{def:Gamma}
   \Gamma(a):=
   \begin{cases}
      \displaystyle\int_\Omega g(a(x))\,\d x,
         & a\in L^2_+(\Omega),\\[3mm]
      +\infty, & \text{otherwise.}
   \end{cases}
\end{equation}

%The elementary lower semicontinuity statement needed below is recorded for completeness.

\begin{lemma}\label{lem:GammaLSC}
The functional $\Gamma$ defined in \eqref{def:Gamma} is proper, convex and
lower semicontinuous on $L^2(\Omega)$.
\end{lemma}

\begin{proof}
Since $\Gamma(0)=|\Omega|g(0)<+\infty$, the functional is proper.
Convexity follows directly from the pointwise convexity of $g$. To prove lower
semicontinuity, let $a_n\to a$ in $L^2(\Omega)$. Since
$L^2_+(\Omega)$ is closed, it is enough to consider a sequence contained in
$L^2_+(\Omega)$. Passing first to a subsequence along which the relevant
$\liminf$ is attained and then, if necessary, to a further subsequence, we may
assume that $a_n\to a$ almost everywhere. Convexity gives continuity of $g$ on
$(0,+\infty)$ and, together with the assumed continuity at zero, yields
$g\in C([0,+\infty))$.
Therefore,
%$g$ is lower semicontinuous on $[0,+\infty)$, so that
\[
   \lim_{n\to+\infty}g(a_n(x))=g(a(x))
   \qquad\mbox{for almost all }x\in\Omega.
\]
First let us assume that $g\geq 0$ on $[0,+\infty)$.
Then by Fatou's lemma we have
\[
   \int_\Omega g(a)\d x \leq \liminf_{n\to+\infty} \int_\Omega g(a_n)\d x.
\]
If $g$ is not nonnegative, choose an affine supporting function
$\ell(s)=c_1s+c_0$ and apply the preceding argument to $g-\ell\geq0$.
The map $a\mapsto\int_\Omega\ell(a)\,\d x$ is continuous on $L^2(\Omega)$,
and the result follows.
\end{proof}

Note that if $g(s)\leq C_0+C_1s^p$ for some $p\geq1$, then
$L^2_+(\Omega)\cap L^p(\Omega)\subseteq\operatorname{dom}(\Gamma)$.
For the following explicit subdifferential formula, assume further that
$g\in C^1([0,+\infty))$. At every $a\in\operatorname{dom}(\Gamma)$ for which
$\nabla u_r[a]\in L^4(\Omega)$ and $g'(a)\in L^2(\Omega)$, the directional
derivative of the integral term is $\int_\Omega g'(a)h\,\d x$, and adding it to
the derivative computed in Lemma~\ref{lem:SubDiffFr} and repeating the same
localization argument gives
\begin{equation}   \label{eq:subdiff_exp}
   \partial(\F_r+\Gamma)(a)
   =- |\nabla u_r[a]|^2+g'(a)+N_+(a).
\end{equation}
No characterization is asserted here when either of these two $L^2$ regularity
conditions fails.
Thus the abstract generation of the flow and its pointwise identification
must be kept distinct.
%%%%%%%%%%%%%%%%%%%%%%%%%%%%%%%%%%%%%%%%%%%%%%%%%%%%%%%%%%%%
\subsection{Gradient flow generated by the regularized energy}
\label{sec:GradientFlowEr}
%%%%%%%%%%%%%%%%%%%%%%%%%%%%%%%%%%%%%%%%%%%%%%%%%%%%%%%%%%%%
The preceding convexity and subdifferential results place the evolution on
the fixed Hilbert space $L^2(\Omega)$. We set
\[
   \E_r(a):=\F_r(a)+\Gamma(a),
\]
where $\F_r$ and $\Gamma$ are defined in \eqref{def:Fr} and
\eqref{def:Gamma}, respectively. For every
$a\in\operatorname{dom}(\partial\E_r)$, we denote by
\begin{equation}\label{def:circ}
   \partial^\circ\E_r(a)
   :=\argmin_{\xi\in\partial\E_r(a)}\|\xi\|_{L^2(\Omega)}
\end{equation}
the element of minimal $L^2$ norm in the subdifferential.

If $\nabla u_r[a]\in L^4(\Omega)$ and $g'(a)\in L^2(\Omega)$, formula
\eqref{eq:subdiff_exp} identifies this minimal section explicitly. Setting
$h_r(a):=-|\nabla u_r[a]|^2+g'(a)$, we obtain
\begin{equation}\label{calc:circ}
   \partial^\circ\E_r(a)
   =
   \begin{cases}
   h_r(a),
      &\text{a.e. on }\{a>0\}\cup\bigl(\{a=0\}\cap\{h_r(a)<0\}\bigr),\\[0.6em]
   0,
      &\text{a.e. on }\{a=0\}\cap\{h_r(a)\geq0\}.
   \end{cases}
\end{equation}
The gradient flow can nevertheless be defined in an abstract way without the
additional regularity assumptions. It is the Cauchy problem
\begin{equation}\label{eq:GFabstract}
   \frac{\d a}{\d t}(t)+\partial\E_r(a(t))\ni0
   \qquad\mbox{for }t>0,
\end{equation}
subject to $a(0)=a^0$. The next theorem gives the well-posedness of this
evolution; the pointwise representation of its velocity will be recovered
subsequently under the assumptions needed for \eqref{calc:circ}.

\begin{theorem}[Existence and uniqueness of the regularized gradient flow]
\label{thm:GFexistence}
Let $r>0$, $f\in L^2(\Omega)$, and let $g:[0,+\infty)\to[0,+\infty)$ be
convex and continuous at zero.
Then for every initial datum $a^0 \in L^2_+(\Omega)$
there exists a unique function $a\in C([0,+\infty);L^2_+(\Omega))$
solving \eqref{eq:GFabstract} in the sense of maximal monotone operators.
%Moreover, $a(t)\in L^2_+(\Omega)$ for all $t\geq 0$.

If, in addition, $a^0\in\mathrm{dom}(\E_r)$, then
$a\in H^1(0,T;L^2(\Omega))$ for every $T>0$,
$a(t)\in\mathrm{dom}(\partial\E_r)$ for almost all $t >0$, and
\begin{equation}  \label{eq:circ}
   \frac{\d a}{\d t}(t)
      +\partial^\circ\E_r(a(t)) =0
   \qquad\text{for almost all }t>0,
\end{equation}
where $\partial^\circ\E_r(a)$ is defined in \eqref{def:circ}.
Moreover, we have the energy identity
\begin{equation}   \label{eq:energyidentity}
   \E_r(a(t))
      +\int_0^t \left\|\frac{\d a}{\d s}(s)\right\|_{L^2(\Omega)}^2  \,\d s
   = \E_r(a^0)
\end{equation}
for all $0 \leq t<+\infty$.
%Equivalently,
%\[  \E_r(a(t))  +\int_s^t \left\|\partial^\circ\E_r(a(\tau))\right\|_{L^2(\Omega)}^2 \,\d\tau  =  \E_r(a(s)).  \]
%Finally, the semigroup generated by $-\partial\E_r$ is contractive:
%if $a_1,a_2$ are two solutions with initial data $a_1^0,a_2^0\in L^2_+(\Omega)$,
%then
%\(   \label{eq:L2contraction}
%   \|a_1(t)-a_2(t)\|_{L^2(\Omega)} \leq \|a_1^0-a_2^0\|_{L^2(\Omega)} \qquad\text{for all }t\geq0.
%\)
\end{theorem}

\begin{proof}
By Lemmas~\ref{lem:FrconvexLSC} and \ref{lem:GammaLSC},
$\E_r$ is convex and lower semicontinuous on $L^2(\Omega)$. Moreover,
$\E_r(0)<+\infty$, so it is proper.

We first identify the closure of its effective domain. Given
$a\in L^2_+(\Omega)$, set $a_k:=\min\{a,k\}$. Then
$a_k\to a$ in $L^2(\Omega)$, while $a_k\in L^\infty_+(\Omega)$,
$\F_r(a_k)<+\infty$ by Lemma~\ref{lem:ur}, and
$\Gamma(a_k)<+\infty$ because $g$ is finite on bounded intervals. Hence
\[
   \overline{\operatorname{dom}(\E_r)}^{\,L^2}
   =L^2_+(\Omega).
\]
For a proper convex lower semicontinuous functional on a Hilbert space,
\[
   \overline{\operatorname{dom}(\partial\E_r)}^{\,L^2}
   =\overline{\operatorname{dom}(\E_r)}^{\,L^2}.
\]
Rockafellar's theorem \cite{Rockafellar} implies that $\partial\E_r$ is
maximal monotone. The Br\'{e}zis--Komura theorem therefore gives a unique
semigroup solution for every $a^0\in L^2_+(\Omega)$; see
\cite[Theorem 11.7]{LecturesOT} and \cite{Brezis,Komura}. Since the semigroup
acts on $\overline{\operatorname{dom}(\partial\E_r)}=L^2_+(\Omega)$, the
solution remains nonnegative.

If $a^0\in\operatorname{dom}(\E_r)$, the standard regularity theory for
maximal-monotone evolutions shows that the solution is strong on every finite
interval and satisfies \eqref{eq:circ}.
%\[  \frac{\d a}{\d t}(t) = \partial^\circ\E_r(a(t)) \qquad\text{for almost all }t>0. \]
%and the element selected by the evolution is the minimal section $\partial^\circ\E_r(a(t))$.
Finally, the energy identity \eqref{eq:energyidentity} follows from the chain
rule for convex lower semicontinuous functionals,
\[
   \frac{\d}{\d t}\E_r(a(t))
      =
   \left\langle \partial^\circ\E_r(a(t)),
      \frac{\d a}{\d t}(t)
   \right\rangle_{L^2}
%      =  -\left\|\partial^\circ\E_r(a(t))\right\|_{L^2}^2.
         =  -\left\| \frac{\d a}{\d t}(t) \right\|_{L^2}^2.
\]
\end{proof}

Theorem~\ref{thm:GFexistence} is entirely variational and does not require a
pointwise formula for the subgradient. If the pressure and metabolic terms
have the additional spatial integrability appearing in
\eqref{eq:subdiff_exp}, the abstract evolution takes the following projected
form.

\begin{corollary}[Projected form of the regularized flow]
\label{cor:ProjectedGF}
Let the assumptions of Theorem \ref{thm:GFexistence} hold.
Assume in addition that $g\in C^1([0,+\infty))$.
Let $a^0\in\mathrm{dom}(\E_r)$, and let $a\in C([0,+\infty);L^2_+(\Omega))$ be the unique solution of
\eqref{eq:GFabstract}.
Suppose that, for almost all $t>0$,
\[
   g'(a(t))\in L^2(\Omega), \qquad \nabla u_r[a(t)]\in L^4(\Omega).
\]
Set $u(t):=u_r[a(t)]$.
Then, for almost all $t>0$ we have
\begin{equation}  \label{eq:GFidentified}
   \partial_t a
   =
   \begin{cases}
      |\nabla u|^2-g'(a),
         & \text{a.e. on }\{a>0\},\\[1mm]
      \bigl(|\nabla u|^2-g'(0)\bigr)_+,
         & \text{a.e. on }\{a=0\}.
   \end{cases}
\end{equation}
\end{corollary}

\begin{proof}
Set $h_r(a):=-|\nabla u_r[a]|^2+g'(a)$. By \eqref{eq:circ} and
\eqref{calc:circ},
\[
   \partial_t a=-h_r(a)=|\nabla u|^2-g'(a)
   \qquad\text{a.e. on }\{a>0\},
\]
whereas on $\{a=0\}$,
\[
   \partial_t a=\bigl(-h_r(a)\bigr)_+
      =\bigl(|\nabla u|^2-g'(0)\bigr)_+.
\]
This is \eqref{eq:GFidentified}.
\end{proof}

\subsection{Bounds induced by the metabolic cost}
The energy identity \eqref{eq:energyidentity} immediately yields global bounds whenever $g$ is
coercive. Assume that $g$ is nonnegative and that there exist constants
$C_0>0$, $C_1\geq0$, and $p\geq1$ such that
\begin{equation}   \label{eq:gcoercive}
   g(s)\geq C_0 s^p-C_1   \qquad\text{for all }s\geq0.
\end{equation}
Then, for every solution with $a^0\in\mathrm{dom}(\E_r)$, the energy identity \eqref{eq:energyidentity} gives
\[
   \E_r(a(t))\leq \E_r(a^0)   \qquad\text{for all }t\geq0.
\]
Since $\F_r\geq0$, we obtain
\[
   \int_\Omega g(a(t))\,\d x     \leq \E_r(a^0)   \qquad\text{for all }t\geq0.
\]
Combined with \eqref{eq:gcoercive}, this gives a uniform a priori bound on $\|a(t)\|_{L^p(\Omega)}$, $t\geq 0$.
In the particular case $g(s) = s^\gamma/\gamma$, we have
\[
   \Gamma(a(t))
      =
   \frac1\gamma \|a(t)\|_{L^\gamma(\Omega)}^\gamma,
\]
and therefore
\[
   \|a(t)\|_{L^\gamma(\Omega)}^\gamma  \leq   \gamma \E_r(a^0)  \qquad\text{for all }t\geq0.
\]
Since $g'(a)=a^{\gamma-1}$, we have $g'(a(t))\in L^2(\Omega)$ for all $t\geq 0$ whenever $1\leq\gamma\leq 2$.

%A sufficient, although not optimal, set of conditions ensuring $\nabla u_r[a]\in L^4(\Omega)$ is the following: assume that
%$\Omega$ is of class $C^{1,1}$, $d\leq4$, $f\in L^2(\Omega)$, and that $a(t)$ is bounded in $W^{1,\infty}(\Omega)$ on the time interval under
%consideration. Since $r+a(t)\geq r>0$, standard elliptic regularity then gives
%\[  u_r[a(t)]\in H^2(\Omega),  \]
%and therefore
%\[  \nabla u_r[a(t)]\in L^4(\Omega)  \]

%Gateaux differentiability...

\subsection{Conditional persistence of positivity}
The preceding estimates follow directly from the abstract gradient-flow
theory. A pointwise lower bound requires, in addition, the identified
evolution \eqref{eq:GFidentified} combined with a pointwise comparison
argument. Whenever \eqref{eq:GFidentified} is valid, the nonnegativity of
$|\nabla u|^2$ gives $\partial_t a\geq-g'(a)$. For almost every
$x\in\Omega$, let
$b=b(x,t)$ solve the scalar equation
\[
\begin{cases}
   \partial_t b=-g'(b),&b>0,\\
   \partial_t b=0,&b=0,\\
   b(x,0)=a^0(x).
\end{cases}
\]
Convexity of $g$ implies that $g'$ is nondecreasing. Thus, whenever comparison
is justified for the chosen representatives, the scalar equation is
order-preserving and gives
\[
   b(x,t)\leq a(x,t)
   \qquad\text{for almost every }x\in\Omega.
\]
Assume moreover that $g'(s)>0$ for $s>0$.
As long as $b(x,t)>0$, separation of variables gives
\[
   \int_{b(x,t)}^{a^0(x)}\frac{\d s}{g'(s)}=t.
\]
Consequently, $b(x,t)$ reaches zero in finite time if and only if
\[
   \int_0^{a^0(x)}\frac{\d s}{g'(s)}<+\infty.
\]
In particular, if $a^0\geq\alpha>0$ almost everywhere and
\[
   \int_0^\alpha\frac{\d s}{g'(s)}=+\infty,
\]
then $a(\cdot,t)>0$ almost everywhere for every $t>0$. This conclusion remains
conditional on the regularity and representative-level comparison required
above.

%%%%%%%%%%%%%%%%%%%%%%%%%%%%%%%%%%%%%
\subsection{Steady states of the gradient flow}\label{subsec:SSGF}

\begin{lemma}[Radial energy approximation]
\label{lem:radial-energy-approximation}
Let $\Omega\subset\R^d$ be a bounded Lipschitz domain
and let $\Phi:\R^d\to(-\infty,+\infty]$ be convex,
lower semicontinuous, radial and nondecreasing
with respect to $|\xi|$, with $\Phi(0)<+\infty$.
Then, for every $u\in H_0^1(\Omega)$ satisfying
\[
   \int_\Omega\Phi(\nabla u)\,\d x<+\infty,
\]
there exists a sequence $\varphi_n\in C_0^\infty(\Omega)$
such that
\begin{equation}\label{eq:radial-approximation}
   \varphi_n\to u\quad\text{in }H_0^1(\Omega),
   \qquad
   \int_\Omega\Phi(\nabla\varphi_n)\,\d x
   \to\int_\Omega\Phi(\nabla u)\,\d x.
\end{equation}
Moreover, the sequence can be chosen such that
\begin{equation}\label{eq:radial-approximation-L1}
   \nabla\varphi_n\to\nabla u\quad\text{a.e. in }\Omega,
   \qquad
   \Phi(\nabla\varphi_n)\to\Phi(\nabla u)
   \quad\text{in }L^1(\Omega).
\end{equation}
%No growth or doubling condition on $\Phi$ is required.
\end{lemma}

\begin{proof}
Since $\Phi$ is radial and nondecreasing,
$\Phi(\xi)\geq\Phi(0)$ for all $\xi\in\R^d$.
We may therefore replace $\Phi$ by
\[
   V(\xi):=\Phi(\xi)-\Phi(0),
\]
which is nonnegative, convex, lower semicontinuous
and satisfies $V(0)=0$.
We carry out the proof in three steps.
%The main difficulty is that ordinary approximation in $H_0^1(\Omega)$ does not necessarily preserve the integral of $V$,
%especially when $V$ has superquadratic growth or takes the value $+\infty$.
%We circumvent this by first moving the support of $u$ into the interior of $\Omega$ while ensuring that its gradient does not increase.

\smallskip
\noindent
\emph{Step 1. Inward deformation.}
By the geometric construction of
Ern and Guermond~\cite{ErnGuermond2016},
there exist smooth diffeomorphisms
$T_\delta:\R^d\to\R^d$, for sufficiently small
$\delta>0$, such that
\[
   T_\delta(\Omega)+B(0,c\delta)\subset\Omega,
   \qquad
   \|DT_\delta-I\|_{L^\infty(\R^d)}\leq C\delta,
\]
where $c,C>0$ are independent of $\delta$.
In particular, $T_\delta$ and its inverse
converge to the identity, together with their
first derivatives, as $\delta\to0$.

Set
\[
   M_\delta:=
   \max\left\{1,
   \|DT_\delta^{-T}\|_{L^\infty(\Omega)}\right\},
   \qquad
   \theta_\delta:=M_\delta^{-1}.
\]
Then $0<\theta_\delta\leq1$,
$\theta_\delta\to1$, and
\[
   \theta_\delta
   \|DT_\delta^{-T}\|_{L^\infty(\Omega)}\leq1.
\]
Define
\[
   u_\delta(y):=
   \begin{cases}
   \theta_\delta u(T_\delta^{-1}(y)),
      &y\in T_\delta(\Omega),\\
   0,&y\notin T_\delta(\Omega).
   \end{cases}
\]
Since $u\in H_0^1(\Omega)$, its zero extension
belongs to $H^1(\R^d)$. Hence $u_\delta\in H_0^1(\Omega)$,
and its support is contained in
$\overline{T_\delta(\Omega)}$, at distance
at least $c\delta$ from $\partial\Omega$.
Moreover, the near-identity properties imply
\begin{equation}\label{eq:radial-deformation-convergence}
   u_\delta\to u\quad\text{in }H_0^1(\Omega).
\end{equation}

For almost every $x\in\Omega$, the chain rule gives
\[
   \nabla u_\delta(T_\delta(x))
   =\theta_\delta DT_\delta(x)^{-T}\nabla u(x).
\]
Consequently,
\[
   |\nabla u_\delta(T_\delta(x))|
   \leq|\nabla u(x)|.
\]
By radial monotonicity of $V$, this implies
\[
   V(\nabla u_\delta(T_\delta(x)))
   \leq V(\nabla u(x)).
\]
Since $\det DT_\delta=1+\mathcal O(\delta)$
uniformly, a change of variables yields
\begin{equation}\label{eq:radial-deformation-energy}
   \int_\Omega V(\nabla u_\delta)\,\d x
   \leq(1+C\delta)\int_\Omega V(\nabla u)\,\d x.
\end{equation}

\smallskip
\noindent
\emph{Step 2. Mollification.}
Extend $u_\delta$ by zero to $\R^d$.
Let $\rho_\varepsilon$ be a standard nonnegative
mollifier and choose $0<\varepsilon<c\delta/2$.
Then
\[
   \varphi_{\delta,\varepsilon}
   :=\rho_\varepsilon*u_\delta
   \in C_0^\infty(\Omega).
\]
For fixed $\delta$, standard mollification gives
\[
   \varphi_{\delta,\varepsilon}\to u_\delta
   \quad\text{in }H^1(\Omega)
   \qquad\text{as }\varepsilon\to0.
\]
Moreover,
\[
   \nabla\varphi_{\delta,\varepsilon}
   =\rho_\varepsilon*\nabla u_\delta.
\]
By Jensen's inequality and convexity of $V$,
\[
   V(\nabla\varphi_{\delta,\varepsilon})
   \leq\rho_\varepsilon*V(\nabla u_\delta).
\]
Because the zero extension of $u_\delta$
has zero gradient outside $\Omega$ and $V(0)=0$,
integration gives
\[
   \int_\Omega
      V(\nabla\varphi_{\delta,\varepsilon})\,\d x
   \leq\int_\Omega V(\nabla u_\delta)\,\d x.
\]
Combining this with
\eqref{eq:radial-deformation-energy}, we obtain
\begin{equation}\label{eq:radial-mollified-energy}
   \int_\Omega
      V(\nabla\varphi_{\delta,\varepsilon})\,\d x
   \leq(1+C\delta)\int_\Omega V(\nabla u)\,\d x.
\end{equation}

Choose $\delta_n\to0$ and then
$\varepsilon_n<c\delta_n/2$ sufficiently small
that
\[
   \|\varphi_{\delta_n,\varepsilon_n}
         -u_{\delta_n}\|_{H^1(\Omega)}
   \leq\frac1n.
\]
Setting
$\varphi_n:=\varphi_{\delta_n,\varepsilon_n}$,
we conclude from
\eqref{eq:radial-deformation-convergence} that
$\varphi_n\to u$ in $H_0^1(\Omega)$, while
\eqref{eq:radial-mollified-energy} implies
\begin{equation}\label{eq:radial-energy-limsup}
   \limsup_{n\to\infty}
   \int_\Omega V(\nabla\varphi_n)\,\d x
   \leq\int_\Omega V(\nabla u)\,\d x.
\end{equation}

\smallskip
\noindent
\emph{Step 3. Convergence of the energies.}
After passing to a subsequence, we may assume
\[
   \nabla\varphi_n(x)\to\nabla u(x)
   \qquad\text{for a.e. }x\in\Omega.
\]
By lower semicontinuity of $V$,
\[
   V(\nabla u)
   \leq\liminf_{n\to\infty}
   V(\nabla\varphi_n)
   \qquad\text{a.e. in }\Omega.
\]
Fatou's lemma therefore yields
\[
   \int_\Omega V(\nabla u)\,\d x
   \leq\liminf_{n\to\infty}
      \int_\Omega V(\nabla\varphi_n)\,\d x.
\]
Together with \eqref{eq:radial-energy-limsup},
this proves convergence of the integrals.

Finally, since
\[
   \bigl(V(\nabla u)-V(\nabla\varphi_n)\bigr)_+
   \to0\qquad\text{a.e. in }\Omega
\]
and this quantity is bounded by
$V(\nabla u)\in L^1(\Omega)$, dominated
convergence gives
\[
   \int_\Omega
   \bigl(V(\nabla u)-V(\nabla\varphi_n)\bigr)_+
   \,\d x\to0.
\]
The convergence of the integrals then implies
that the positive parts with the opposite sign
also converge to zero in $L^1(\Omega)$.
Hence
\[
   \|V(\nabla\varphi_n)-V(\nabla u)\|_{L^1(\Omega)}
   \to0.
\]
Adding back the constant $\Phi(0)$ proves
\eqref{eq:radial-approximation} and
\eqref{eq:radial-approximation-L1}.
\end{proof}

\begin{remark}
Lemma~\ref{lem:radial-energy-approximation}
applies to spatially homogeneous radial integrands,
including those imposing pointwise gradient constraints.
It does not assert density of smooth functions in
arbitrary weighted Sobolev spaces with energies
of the form $\int_\Omega a(x)|\nabla u|^2\,\d x$.
In particular, it does not identify the smooth-core
weighted space with the maximal finite-energy space
for a general unbounded coefficient $a$.
\end{remark}

Since $\E_r$ is convex, a conductivity is a steady state of the abstract
gradient flow if and only if it is a global minimizer of $\E_r$.

Let $a_0$ be a steady state, so that $0\in\partial\E_r(a_0)$, and set
$u_0:=u_r[a_0]$ the solution of \eqref{eq:ur}. If $\nabla u_0\in L^4(\Omega)$ and
$g'(a_0)\in L^2(\Omega)$, then
\eqref{eq:subdiff_exp} gives the pointwise conditions
\begin{equation}  \label{eq:pointwiseStat}
   |\nabla u_0|^2=g'(a_0) \quad \text{almost everywhere in }\{a_0>0\},
\end{equation}
and
\[
   |\nabla u_0|^2 \leq g'(0) \quad \text{almost everywhere in }\{a_0=0\}.
\]
To formulate the same condition without requiring $|\nabla u_0|^2$ and
$g'(a_0)$ to belong separately to $L^2(\Omega)$, we absorb the nonnegativity
constraint into the metabolic integrand and use Fenchel duality.

We now denote the $\R^+$--$\R^+$ Legendre transform of $g$,
\[
   g_+^*(s) = \sup_{v\ge 0}\bigl(sv-g(v)\bigr), \qquad s\ge 0.
\]
The qualification condition here is that $g$ is convex and continuous on
$[0,\infty)$. Its role becomes clear through the saddle functional
\[
L(\varphi,a)
:=
\int_\Omega
\left(
\frac r2|\nabla\varphi|^2
+
\frac12\bigl(a|\nabla\varphi|^2-g(a)\bigr)
\right)\d x
-
\int_\Omega f\varphi\,\d x.
\]
Note that $L$ is strictly convex in the first argument $\varphi$ and concave in the second argument $a$.

\begin{proposition}
Let $r>0$ and let $g: [0,\infty) \to [0,\infty)$ be convex and continuous at zero.
Then
\[
   \inf_{a\in L^2_+(\Omega)} \E_r(a) \geq -2 \inf_{\varphi\in H^1_0(\Omega)} H(\varphi),
\]
where
\[
   H(\varphi) = \int_\Omega \left( \frac{r}{2} |\nabla\varphi|^2 + \frac12 g_+^*\left(|\nabla\varphi|^2 \right) - f\varphi \right) \,\d x.
\]
\end{proposition}

\begin{proof}
For every $a\in L^2_+(\Omega)$ with
$\E_r(a)<+\infty$ and every
$\varphi\in C_0^\infty(\Omega)$,
the common-core representation of the pumping
energy and the Fenchel--Young inequality give
\begin{align*}
   \E_r(a)
   &\geq
   2\int_\Omega f\varphi\,\d x
   -\int_\Omega(r+a)|\nabla\varphi|^2\,\d x
   +\int_\Omega g(a)\,\d x\\
   &\geq
   2\int_\Omega f\varphi\,\d x
   -\int_\Omega r|\nabla\varphi|^2\,\d x
   -\int_\Omega g_+^*(|\nabla\varphi|^2)\,\d x\\
   &=-2H(\varphi).
\end{align*}
The inequality is trivial when $\E_r(a)=+\infty$.

To extend it to arbitrary $\varphi\in H_0^1(\Omega)$
with $H(\varphi)<+\infty$, consider
\[
   \Phi_r(\xi):=
   \frac r2|\xi|^2+\frac12g_+^*(|\xi|^2).
\]
Since $g$ is finite, convex, nonnegative and
continuous at zero, $\Phi_r$ is convex,
lower semicontinuous, radial and nondecreasing
in $|\xi|$, with $\Phi_r(0)<+\infty$.
By Lemma~\ref{lem:radial-energy-approximation},
there exist $\varphi_n\in C_0^\infty(\Omega)$
such that
\[
   \varphi_n\to\varphi\quad\text{in }H_0^1(\Omega),
   \qquad
   \int_\Omega\Phi_r(\nabla\varphi_n)\,\d x
   \to\int_\Omega\Phi_r(\nabla\varphi)\,\d x.
\]
Since $f\in L^2(\Omega)$, this implies
$H(\varphi_n)\to H(\varphi)$.
Passing to the limit yields
$\E_r(a)\geq-2H(\varphi)$.
Taking the infimum over $a$ and $\varphi$
proves the assertion.
\end{proof}

Assume now that $(u_0,a_0)$ is a saddle point of $L$, with $a_0\in L^2_+(\Omega)$ and $u_0 \in H_0^1(r +a_0)$.
Stationarity in the pressure variable gives
\begin{equation}\label{eq:saddle-pressure}
   \int_\Omega(r+a_0)\nabla u_0\cdot\nabla v\,\d x
   =\int_\Omega fv\,\d x
   \qquad \mbox{for all } v\in\Ha{r+a_0}.
\end{equation}
The maximization in the conductivity variable is equivalent to the
Fenchel relation
\begin{equation}\label{eq:saddle-conductivity}
   |\nabla u_0|^2\in\partial g_+(a_0)
   \qquad\Longleftrightarrow\qquad
   a_0\in\partial g_+^{*}(|\nabla u_0|^2)
   \quad\text{a.e. in }\Omega.
\end{equation}
If $g\in C^1([0,+\infty))$, this relation becomes
the pointwise stationarity condition \eqref{eq:pointwiseStat}.
Moreover, since $g_+^*$ is convex and nondecreasing on $[0,\infty)$, the map
$w\mapsto g_+^*(w^2)$ is convex. This shows the convexity of $H$.
The growth of $g$ at infinity determines the effective domain of $g_+^*$ and
therefore the nature of the pressure problem: it may rule out nontrivial
minima, impose a gradient constraint, or produce a superlinear
gradient penalty. More precisely:
\begin{itemize}
\item
If $\lim_{s\to\infty}g(s) <+\infty$,  then we have
\[
   g_+^*(v)=+\infty \qquad \text{for all } v>0,
\]
so the Hamiltonian is finite only at $\varphi=0$. Consequently, no stationary
state exists when $f\not\equiv0$.
\item
If $\lim_{s\to\infty}g(s)=+\infty$ and $\lim_{s\to\infty}\frac{g(s)}s=\alpha\in(0,+\infty)$,
then
\[
   [0,\alpha)\subseteq\operatorname{dom}g_+^*\subseteq[0,\alpha],
   \qquad g_+^*(v)=+\infty\quad\text{for }v>\alpha.
\]
The endpoint $\alpha$ belongs to the effective domain if and only if
$\sup_{s\geq0}(\alpha s-g(s))<+\infty$. Thus the Hamiltonian itself encodes a
strict or non-strict gradient constraint, depending on this endpoint behavior.
The gradient-constraint problem has at most one minimizing pressure because
of the strictly convex Dirichlet term. Existence of a corresponding stationary
conductivity additionally requires an $L^2$ Lagrange multiplier realizing the
pointwise supremum.
\item
If $\lim_{s\to\infty}g(s)=+\infty$ and
$\lim_{s\to\infty}g(s)/s=+\infty$, i.e., if $g$ grows superlinearly, then
$\operatorname{dom}g_+^*=[0,+\infty)$.
Since $H$ is coercive, proper, lower semicontinuous, and strictly convex, it
has a unique minimizing pressure $u_0 \in H_0^1(\Omega)$.
This pressure lifts to a stationary conductivity if there exists a pointwise maximizer
(in the $\sup$ that defines the Legendre transform of $g_+$)
$a_0\in L^2_+(\Omega)$ of finite energy such that $(u_0,a_0)$ satisfies the
pressure equation \eqref{eq:saddle-pressure}. Uniqueness of the resulting
conductivity follows, for instance, if $g$ is strictly convex.
\end{itemize}

The linear and power-law metabolic costs illustrate, respectively, the
gradient-constraint and superlinear regimes.
\begin{enumerate}
\item
First, we choose $g(s):=\lambda s$ with $\lambda>0$.
Then
\[
g_+^*(v)
=
\sup_{s\ge0} \, (vs-\lambda s)
=
\begin{cases}
+\infty, &\mbox{for } v>\lambda,\\
0, & \mbox{for } 0 \le v\le \lambda.
\end{cases}
\]
The steady-state problem is
\[
   u_0 \in \argmin_{\substack{\varphi\in H^1_0(\Omega)\\ \|\nabla\varphi\|_{L^\infty(\Omega)}^2\leq\lambda}}
   \left( \frac r2\int_\Omega |\nabla\varphi|^2\,\d x - \int_\Omega f\varphi\,\d x \right).
\]
\item
For the choice $g(s):=\frac\lambda\gamma s^\gamma$ with $\lambda>0$ and $\gamma>1$, we have
\[
   g_+^*(v) = \sup_{s\ge0} \left( sv-\frac\lambda\gamma s^\gamma \right)
     = c\,v^{\gamma/(\gamma-1)},
\]
with $c:=\frac{\gamma-1}{\gamma\,\lambda^{1/(\gamma-1)}}$.
The steady-state problem then reads
\[
   u_0  = \argmin_{\varphi\in W^{1, p}_0(\Omega)}
   \int_\Omega \left(  \frac r2|\nabla\varphi|^2 +
   \frac c2|\nabla\varphi|^p -   f\varphi \right)\,\d x,
\]
with $p=\frac{2\gamma}{\gamma-1}\in(2,\infty)$.
Note that the second term under the integral is of $p$-Laplacian type.
\end{enumerate}

%%%%%%%%%%%%%%%%%%%%%%%%%%%%%%%%%%%%%%%%%%%%%%%%%%%%%%%%%%%%

\subsection{Minimizers and long-time behavior}
\label{subsec:ELCharacterizationMinimizers}

We now characterize stationary minimizers of $\E_r$ and
discuss the long-time behavior of its gradient flow.
By Fermat's principle, $a_0\in L^2_+(\Omega)$
is a global minimizer of $\E_r$ if and only if
\[
   0\in\partial\E_r(a_0),
\]
or, equivalently, $\partial^\circ\E_r(a_0)=0$,
where $\partial^\circ\E_r(a_0)$ denotes the
minimal section of the subdifferential.
The following result provides a sufficient condition
for minimality in terms of the stationary elliptic problem.
%without requiring strong duality or saddle-point attainment.

\begin{proposition}[Stationary minimizers]
\label{prop:stationary-minimizers}
Let $r>0$ and let $g:[0,\infty)\to[0,\infty)$
be finite, convex and continuous at zero.
Suppose that $a_0\in L^2_+(\Omega)$ and that
$u_0\in\Ha{r+a_0}$ is the corresponding
energy solution of
\begin{equation}\label{eq:stationary-pressure}
   -\operatorname{div}\bigl((r+a_0)\nabla u_0\bigr)=f
   \qquad\text{in }\Omega,
\end{equation}
with homogeneous Dirichlet boundary conditions.
Assume moreover that
\begin{equation}\label{eq:stationary-conductivity}
   |\nabla u_0|^2\in\partial g_+(a_0)
   \qquad\text{a.e. in }\Omega,
\end{equation}
where $g_+$ denotes the extension of $g$ by
$+\infty$ to negative arguments.
Then $a_0$ is a global minimizer of $\E_r$
on $L^2(\Omega)$. If $g$ is strictly convex,
the minimizer is unique.
\end{proposition}

\begin{proof}
Set $s_0:=|\nabla u_0|^2$.
By \eqref{eq:stationary-conductivity} and
Fenchel's equality,
\begin{equation}\label{eq:stationary-Fenchel}
   g_+^*(s_0)=a_0s_0-g(a_0)
   \qquad\text{a.e. in }\Omega.
\end{equation}
Since $g\geq0$ and $s_0\geq0$, we have
$g_+^*(s_0)\geq-g(0)$, and consequently
\[
   0\leq g(a_0)
   \leq a_0s_0+g(0).
\]
The weighted energy identity implies
\[
   \int_\Omega(r+a_0)s_0\,\d x
   =\int_\Omega fu_0\,\d x<+\infty.
\]
Thus $g(a_0)\in L^1(\Omega)$, and
\eqref{eq:stationary-Fenchel} shows that
$g_+^*(s_0)\in L^1(\Omega)$.
In particular, the Hamiltonian $H$
introduced in the preceding subsection
is finite at $u_0$.

By the Hamiltonian inequality established there,
whose extension from smooth test functions to
$H_0^1(\Omega)$ follows from
Lemma~\ref{lem:radial-energy-approximation},
we have for every $a\in L^2_+(\Omega)$
\[
   \E_r(a)\geq-2H(u_0).
\]
Using \eqref{eq:stationary-Fenchel} and the
weighted energy identity again, we obtain
\begin{align*}
   -2H(u_0)
   &=2\int_\Omega fu_0\,\d x
     -r\int_\Omega s_0\,\d x
     -\int_\Omega g_+^*(s_0)\,\d x\\
   &=2\int_\Omega fu_0\,\d x
     -\int_\Omega(r+a_0)s_0\,\d x
     +\int_\Omega g(a_0)\,\d x\\
   &=\int_\Omega(r+a_0)|\nabla u_0|^2\,\d x
     +\int_\Omega g(a_0)\,\d x\\
   &=\E_r(a_0).
\end{align*}
Hence $\E_r(a)\geq\E_r(a_0)$ for every
$a\in L^2_+(\Omega)$, proving global minimality.
If $g$ is strictly convex, then $\E_r$ is strictly
convex on its effective domain, and the minimizer
is unique.
\end{proof}

The condition \eqref{eq:stationary-conductivity}
is equivalent to
\[
   a_0\in\partial g_+^*(|\nabla u_0|^2)
   \qquad\text{a.e. in }\Omega.
\]
If $g_+^*$ is differentiable at
$|\nabla u_0|^2$, this reduces to
\[
   a_0=(g_+^*)'(|\nabla u_0|^2).
\]
%The subdifferential formulation is more general, as it also covers nonsmooth conjugates and
%gradient constraints arising, for instance, from linear metabolic costs.

Conversely, a minimizer of $\E_r$ gives rise
to a stationary pair of the above form whenever
the corresponding pressure dual problem is attained
without a duality gap and the optimizing conductivity
belongs to $L^2_+(\Omega)$. % with the required core-energy properties.
%The proposition itself does not require these additional assumptions.
%In particular, neither $\nabla u_0\in L^4(\Omega)$ nor $g'(a_0)\in L^2(\Omega)$ is needed for its sufficiency conclusion.

For the example of the power law $g(s)=\frac1\gamma s^\gamma$ with $\gamma>1$,
we calculated
\[
   g_+^*(v) = C_\gamma v^{\frac{\gamma}{\gamma-1}}
\]
for some constant $C_\gamma>0$. Hence
\[
   \left(g_+^*\right)'(v) = C_\gamma \frac{\gamma}{\gamma-1} v^{\frac1{\gamma-1}},
\]
and, after absorbing the preceding factor into $C_\gamma$,
\[
   a_0 = \left(g_+^*\right)'\left(|\nabla u_0|^2\right) = C_\gamma |\nabla u_0|^{\frac{2}{\gamma-1}}.
\]
The natural Hamiltonian estimate gives $a_0\in L^\gamma(\Omega)$. It therefore
guarantees $a_0\in L^2(\Omega)$ when $\gamma\geq2$, while
$\nabla u_0\in L^4(\Omega)$ and $g'(a_0)\in L^2(\Omega)$ follow when
$1<\gamma\leq2$.
%Hence all three regularity properties follow simultaneously from the natural estimates only for $\gamma=2$.
%They are not needed for the cancellation-aware variational characterization above.
The stationary certificate does not require that $\nabla u_0 \in L^4(\Omega)$ or $g'(a_0)\in L^2(\Omega)$.
However, the condition $a_0 \in L^2(\Omega)$ must be checked separately when $1 < \gamma < 2$.

%\paragraph{Convergence of the flow.}
Finally, we make an observation about the long-time behavior of the gradient flow.
If $\argmin\E_r\neq\emptyset$, then there exists $a_\infty\in\argmin\E_r$ such that
\[
   a(t)\rightharpoonup a_\infty
   \qquad\text{weakly in }L^2(\Omega)\quad\text{as }t\to+\infty.
\]
Indeed, for every $\bar a\in\argmin\E_r$, the subgradient inequality and the
evolution equation give
\[
   \frac12\frac{\d}{\d t}\|a(t)-\bar a\|_{L^2(\Omega)}^2
   \leq-\bigl(\E_r(a(t))-\E_r(\bar a)\bigr).
\]
Thus the distance to every minimizer has a limit, the energy converges to its
minimum, and every weak cluster point is a minimizer. Opial's lemma \cite{Opial, OnOpial} then gives
convergence of the whole trajectory. If the minimizer is unique, then it is $a_\infty$.
%is that unique minimizer.

%%%%%%%%%%%%%%%%%%%%%%%%%%%%%%%%%
%% Degenerate Mosco limit
%%%%%%%%%%%%%%%%%%%%%%%%%%%%%%%%%
\subsection{Passage to the degenerate problem}\label{subsec:DegenerateLimit}
All results above use the uniform ellipticity supplied by $r>0$. We now
identify the variational problem obtained in the limit $r\to 0^+$. It is more
convenient to work with the total conductivity $b:=r+a$, rather than with $a$ itself.
This places all cutoff problems on the common
space $L^2(\Omega)$ and makes their admissible sets nested. For a fixed
$a\in L^2_+(\Omega)$, we have $b\in L^2_r(\Omega)$, where
\[
   L^2_{r}(\Omega) :=
   \left\{   v\in L^2(\Omega):\ v(x)\geq r  \text{ for a.e. }x\in\Omega  \right\},
\]
and $r+L^2_+(\Omega)=L^2_{r}(\Omega)$.
We define
\[
   \G_r(b) :=
   \begin{cases}
      \displaystyle
      \int_\Omega b\,|\nabla u[b]|^2\,\d x,
         & \text{if } b\geq r \text{ a.e. in }\Omega,\\[3mm]
      +\infty,
         & \text{otherwise},
   \end{cases}
\]
where $u=u[b]\in\Ha{b}$ %\hookrightarrow H^1_0(\Omega)$
is the unique $H_0^1(\Omega, b)$-energy solution of
\[
   -\operatorname{div}(b\nabla u)=f
      \qquad\text{in }\Omega,
\]
subject to the homogeneous Dirichlet boundary condition $u=0$ on $\partial\Omega$.
Thus $\G_r:L^2(\Omega)\to\Rext$.
By construction,
\[
   \G_r(r+a)=\F_r(a)
   \qquad\text{for every }a\in L^2_+(\Omega).
\]
Thus $\G_r$ is the regularized pumping-energy functional written in
terms of the total conductivity; the parameter $r$ now enters only through the
constraint $b\geq r$. In particular, $\G_{r_1}$ extends $\G_{r_2}$ whenever
$0<r_1\leq r_2$.

For $b\in L^2_{r}(\Omega)$ the representation \eqref{eq:Frrepr} gives
\[
   \G_r(b)
   =\sup_{\varphi\in C_0^\infty(\Omega)}
   \left(2\int_\Omega f\varphi\,\d x
   -\int_\Omega b|\nabla\varphi|^2\,\d x\right).
\]
It is also useful to introduce the quadratic functional
\[
   \mathcal H_b(\varphi)
   :=
   \begin{cases}
      \displaystyle
      \frac12\int_\Omega b|\nabla\varphi|^2\,\d x,
         & \text{if } \varphi\in\Ha{b},\\[3mm]
      +\infty,
         & \text{otherwise}.
   \end{cases}
\]
on $L^2(\Omega)$. With this notation,
\[
   \G_r(b)=2\mathcal H_b^*(f),
\]
where the Legendre transform $\mathcal H_b^*$ is taken with respect to the
$L^2$-$L^2$ duality.

Let $(r_n)_{n\in\N}$ be a decreasing sequence with vanishing limit as $n\to\infty$.
The corresponding functionals are monotone,
\[
   \G_{r_{n+1}}\leq \G_{r_n}
   \qquad\text{on }L^2(\Omega).
\]
Indeed, if $b\geq r_n$ almost everywhere, then $b\geq r_k$ for every
$k\geq n$, and the variational representation, which depends only on $b$,
gives $\G_{r_k}(b)=\G_{r_n}(b)$. If $b\not\geq r_n$, then
$\G_{r_n}(b)=+\infty$, and the inequality is automatic.

The limit is best expressed in the Mosco topology because this is the
variational convergence compatible with maximal-monotone subdifferentials.
Here we identify only the pumping-energy functionals; no convergence of the
associated gradient flows is asserted.

\begin{proposition}[Degenerate Mosco limit]\label{prop:degenerate-Mosco}
Let $r_n\to 0^+$. Then $\G_{r_n}$ converges in the sense of Mosco on
$L^2(\Omega)$ to the relaxed pumping-energy functional
\begin{equation}\label{def:G0-relaxed}
   \G_0(b):=
   \begin{cases}
   \displaystyle\sup_{\varphi\in C_0^\infty(\Omega)}
   \left(2\int_\Omega f\varphi\,\d x
   -\int_\Omega b|\nabla\varphi|^2\,\d x\right),&b\in L^2_+(\Omega),\\[3mm]
   +\infty,&\text{otherwise.}
   \end{cases}
\end{equation}
\end{proposition}

\begin{proof}
The cone $L^2_+(\Omega)$ is weakly closed. Moreover, for every fixed
$\varphi\in C_0^\infty(\Omega)$, the mapping
\[
   b\mapsto
   2\int_\Omega f\varphi\,\d x
   -\int_\Omega b|\nabla\varphi|^2\,\d x
\]
is weakly continuous on $L^2(\Omega)$. Hence $\G_0$, being the supremum of
these affine functionals on the weakly closed cone, is weakly lower
semicontinuous.

Let $b_n\rightharpoonup b$ in $L^2(\Omega)$. Since
$\G_{r_n}(b_n)\geq\G_0(b_n)$, we obtain
\[
   \liminf_{n\to\infty}\G_{r_n}(b_n)
   \geq\liminf_{n\to\infty}\G_0(b_n)
   \geq\G_0(b).
\]
It remains to construct a recovery sequence when $\G_0(b)<+\infty$, in which
case $b\in L^2_+(\Omega)$. We take $b_n:=b+r_n$. Then
$b_n\to b$ strongly in $L^2(\Omega)$, $b_n\geq r_n$, and
\[
   \G_{r_n}(b_n)=\G_0(b+r_n).
\]
For each fixed $\varphi\in C_0^\infty(\Omega)$, the corresponding expression
in the defining supremum increases to
\[
   2\int_\Omega f\varphi\,\d x
   -\int_\Omega b|\nabla\varphi|^2\,\d x.
\]
Since $b+r_n \to b$ strongly in $L^2(\Omega)$, lower semicontinuity and order reversal give
\[
   \G_0(b) \leq \liminf_{n\to\infty} \G_0(b+r_n) \leq \limsup_{n\to\infty} \G_0(b+r_n) \leq \G_0(b).
\]
Consequently, $\G_0(b+r_n)\to \G_0(b)$ as $n\to\infty$. If $\G_0(b)=+\infty$, the Mosco
recovery inequality is automatic; for instance, one may take the constant
sequence $b_n=b$. This proves the claimed Mosco convergence
\cite{Maso93,mosco1969convergence,mosco1994composite}.
\end{proof}

The functional $\G_0$ is well defined on the entire nonnegative cone, but for
a general degenerate conductivity it is only a relaxed pumping-energy
functional:
the associated elliptic equation need not possess an energy solution. The
next section restores such a solution theory.

%%%%%%%%%%%%%%%%%%%%%%%%%%%%%%%%%
%% Linear mobility
%%%%%%%%%%%%%%%%%%%%%%%%%%%%%%%%%
\section{Linear mobility---weak cutoff}\label{sec:WeakCutoff}
Throughout this section we consider the (possibly degenerate) elliptic equation
\begin{equation}   \label{eq:linear}
   \begin{aligned}
      -\operatorname{div} (a\nabla u) &= f \qquad\mbox{in } \Omega,\\
      u &= 0 \qquad\mbox{on } \partial\Omega.
   \end{aligned}
\end{equation}
%\subsection{The inverse-moment cutoff and the energy solution}
For $K>0$ we define the set
\begin{equation}   \label{def:KK}
   \K_K := \left\{a\in L^2_+(\Omega);\; \int_\Omega a(x)^{-d/2} \,\d x \leq K \right\},
\end{equation}
where we recall that $\Omega\subset\R^d$ is bounded with Lipschitz boundary.
We refer to this inverse-moment condition as a weak cutoff. It implies that
$a>0$ almost everywhere, but does not require $a$ to be bounded away from
zero. We first show that $\K_K$ has the geometry required for
convex analysis and provides enough coercivity to solve the corresponding
weighted elliptic problem.

\begin{lemma}\label{lem:KKcc}
For any $K>0$, the set $\K_K$ defined in \eqref{def:KK} is convex and closed in $L^2(\Omega)$.
\end{lemma}

\begin{proof}
Extend $s\mapsto s^{-d/2}$ by the value $+\infty$ at $s=0$. The resulting
function is convex and lower semicontinuous on $[0,+\infty)$. Thus, for any
$a_0$, $a_1\in \K_K$ and $t \in (0,1)$,
\[
   [ta_1+(1-t)a_0]^{-d/2} \leq t a_1^{-d/2} + (1-t) a_0^{-d/2}.
\]
Integration over $\Omega$ gives
\[
   \int_\Omega [ta_1+(1-t)a_0]^{-d/2} \d x  \le tK+(1-t)K = K,
\]
so that $ta_1+(1-t)a_0 \in \K_K$. Therefore, $\K_K$ is convex.

To prove closedness, let $a_n\in\K_K$ converge to $a$ in $L^2(\Omega)$.
After passing to a subsequence, we may assume that
\[
   a_n(x)\to a(x)\qquad\text{for almost every }x\in\Omega.
\]
In particular, $a\geq0$ almost everywhere. Fatou's lemma then gives
\[
   \int_\Omega a(x)^{-d/2} \,\d x
    \leq \liminf_{n\to+\infty} \int_\Omega a_n(x)^{-d/2} \,\d x
   \leq K.
\]
Hence $a\in\K_K$, and the set is closed.
\end{proof}

From now on we assume that $d\geq 2$. The choice of the exponent $-d/2$ in
\eqref{def:KK} is dictated by the following weighted Sobolev--Poincar\'e
estimate: it allows the inverse-moment bound to compensate for the loss of
uniform ellipticity.

\begin{lemma}\label{lem:embdL2}
Let $a$ be locally integrable on $\Omega$ and $a^{-1/2} \in L^{d}_+(\Omega)$.
Let $C_\Omega>0$ be the constant in the Sobolev--Poincar\'{e} inequality on $\Omega$.
Then
\[
   \Norm{u}_{L^2(\Omega)} \leq %C \left( \int_\Omega a |\nabla u|^2 \d x \right)^{1/2}
      %\Norm{\nabla u}_{L^\frac{2d}{d+2}(\Omega)} \leq 
      C_\Omega \Norm{a^{-1/2} }_{L^{d}(\Omega)} \left( \int_\Omega a |\nabla u|^2 \d x \right)^\frac12
   \qquad\mbox{for all } u\in C_0^\infty(\Omega).
\]
\end{lemma}

\begin{proof}
Let $u\in C_0^\infty(\Omega)$ and set $q:=2d/(d+2)$. The
Sobolev--Poincar\'e inequality gives
\[
   \Norm{u}_{L^2(\Omega)}\leq C_\Omega\Norm{\nabla u}_{L^q(\Omega)}.
\]
Since $1/q=1/d+1/2$, we write
$\nabla u=a^{-1/2}(\sqrt a\,\nabla u)$ and apply H\"older's inequality to
obtain
\[
   \Norm{\nabla u}_{L^q(\Omega)}
   \leq \left( \int_\Omega a(x)^{-d/2} \,\d x \right)^{1/d}
      \left( \int_\Omega a |\nabla u|^2 \d x \right)^{1/2}.
\]
Combining the two estimates proves the claim.
\end{proof}

For $a\in\K_K$, we use the weighted bilinear form
\[
   (u,v)_a := % \int_\Omega u v\,\d x + 
      \int_\Omega a \nabla u\cdot\nabla v \,\d x,
\]
and the induced norm $\Na{a}{u}:=(u,u)_a^{1/2}$. Lemma~\ref{lem:embdL2}
shows that $\Na{a}{u}=0$ implies $u=0$ almost everywhere, so this is indeed
a norm on $C_0^\infty(\Omega)$. We define $\Ha{a}$ as its completion.
Then the estimate of Lemma~\ref{lem:embdL2} extends by density. Moreover,
\[
   \Norm{a^{-1/2}}_{L^d(\Omega)}
   =\left(\int_\Omega a^{-d/2}\,\d x\right)^{1/d}\leq K^{1/d},
\]
and therefore
\[
   \Norm{u}_{L^2(\Omega)}
   \leq C_\Omega K^{1/d}\Na{a}{u}
   \qquad\text{for every }u\in\Ha{a}.
\]
Thus $\Ha{a}\hookrightarrow L^2(\Omega)$ continuously, and its elements have
well-defined $L^2$ representatives.
The gradient estimate used in the proof of Lemma~\ref{lem:embdL2} also
extends by density. With $q:=2d/(d+2)$, it yields
\[
   \Norm{\nabla u}_{L^q(\Omega)}
   \leq K^{1/d}\Na{a}{u}
   \qquad\text{for every }u\in\Ha{a}.
\]
Consequently, $\Ha{a}$ embeds continuously into $W^{1,q}_0(\Omega)$. In
particular, the variational equation \eqref{eq:uvarphi} below has the usual distributional
interpretation.

\begin{lemma} \label{lem:Riesz}
For every $a\in\K_K$ and $f\in L^2(\Omega)$, there exists a unique weak
solution $u\in\Ha{a}$ of the Dirichlet problem \eqref{eq:linear},
characterized by
\begin{equation}  \label{eq:uvarphi}
   (u,\varphi)_a=\int_\Omega f\varphi\,\d x
   \qquad\mbox{for every }\varphi\in\Ha{a}.
\end{equation}
In particular, $u$ solves \eqref{eq:linear} in the sense of distributions.
\end{lemma}

\begin{proof}
For every $\varphi\in\Ha{a}$, the preceding embedding estimate gives
\[
   \left|\int_\Omega f\varphi\,\d x\right|
   \leq\Norm{f}_{L^2(\Omega)}\Norm{\varphi}_{L^2(\Omega)}
   \leq C_\Omega K^{1/d}\Norm{f}_{L^2(\Omega)}\Na{a}{\varphi}.
\]
Hence $\varphi\mapsto\int_\Omega f\varphi\,\d x$ is a continuous linear
functional on $\Ha{a}$. The existence and uniqueness of $u$ now follow from
the Riesz representation theorem.
\end{proof}

\subsection{Relaxed pumping energy and dominated coefficients}

Lemma~\ref{lem:Riesz} solves the state equation for each fixed
$a\in\K_K$. However, since the energy space $\Ha{a}$ itself depends on $a$,
it is preferable to describe variations of the coefficient through a fixed
test space. The required functional has already appeared as the Mosco limit
$\G_0$ in \eqref{def:G0-relaxed}; from now on we denote it
$\G:L^2(\Omega)\to\Rext$,
\begin{equation}  \label{def:GLin}
   \G(a) :=
   \begin{cases}
   \displaystyle\sup_{u\in C_0^\infty(\Omega)}
   \left(2\int_\Omega f u\,\d x-\int_\Omega a|\nabla u|^2\,\d x\right),
      &\text{if }a\in L^2_+(\Omega),\\[3mm]
   +\infty,&\text{otherwise.}
   \end{cases}
\end{equation}
On the weak-cutoff set $\K_K$, this relaxed definition agrees with the energy of
the uniquely determined state. Indeed, density of $C_0^\infty(\Omega)$ in $\Ha{a}$,
continuity of the $H_0^1(a)$-norm and of the scalar product map $u \mapsto \int_\Omega f u \d x$, $H_0 ^1(a) \mapsto \R$, 
and Lemma~\ref{lem:Riesz} give
\[
   \G(a)
   =\sup_{u\in\Ha{a}}\left(2\int_\Omega fu\,\d x-\int_\Omega a|\nabla u|^2\,\d x\right)
   = \int_\Omega a |\nabla u[a]|^2 \d x,
\]
where $u[a]\in\Ha{a}$ is the unique weak solution of \eqref{eq:linear}.
Using analogous arguments as in the proof of Lemma \ref{lem:Frrepr}, we conclude that for any $a\in \K_K$,
\[
   \sup_{u\in\Ha{a}}\left(2\int_\Omega fu\,\d x-\int_\Omega a|\nabla u|^2\,\d x\right)
   =
   \sup_{u\in C_0^\infty(\Omega)}   \left(2\int_\Omega f u\,\d x-\int_\Omega a|\nabla u|^2\,\d x\right). 
\]
The advantage of \eqref{def:GLin} is that the common test space makes the
convexity and lower semicontinuity transparent.

Besides its convexity and lower semicontinuity, the functional
$\G$ admits a flux representation and can be characterized
as the strong-$L^2$ relaxation of the classical pumping
energy. We establish these properties below.

\begin{lemma}\label{lem:Gconv1}
The functional $\G$ defined by \eqref{def:GLin} is convex and lower
semicontinuous on $L^2(\Omega)$.
\end{lemma}

\begin{proof}
For each fixed $u\in C_0^\infty(\Omega)$, the expression in parentheses in
\eqref{def:GLin} is a continuous affine functional of $a\in L^2(\Omega)$.
Its supremum is therefore convex and lower semicontinuous. Adding the indicator
of the closed convex cone $L^2_+(\Omega)$ preserves both properties.
\end{proof}

\begin{proposition}[Flux representation]\label{prop:flux-representation}
Let $a\in L^2_+(\Omega)$ and define
\[
   \mathcal Q_f:=
   \left\{j\in L^{4/3}(\Omega;\R^d):
   -\operatorname{div}j=f\text{ in }\mathcal D'(\Omega)\right\}.
\]
Then the relaxed pumping energy $\G$ defined in \eqref{def:GLin}
admits the representation
\begin{equation}\label{eq:G-flux-representation}
   \G(a)=\inf_{j\in\mathcal Q_f}
   \int_\Omega\frac{|j|^2}{a}\,\d x,
\end{equation}
where $|j|^2/a$ is interpreted as $+\infty$ on
$\{a=0,\ j\neq0\}$ and as zero on $\{a=0,\ j=0\}$.
If $\G(a)<+\infty$, the infimum is attained by a
unique flux $j_a\in\mathcal Q_f$.
\end{proposition}

\begin{proof}
Since $a\in L^2(\Omega)$, the functional
\[
   \Psi(\sigma):=\int_\Omega a|\sigma|^2\,\d x,
   \qquad \sigma\in L^4(\Omega;\R^d),
\]
is finite, convex and continuous on $L^4(\Omega;\R^d)$.
Indeed, by H\"older's inequality,
\[
   \Psi(\sigma)\leq
   \|a\|_{L^2(\Omega)}\|\sigma\|_{L^4(\Omega)}^2.
\]
Moreover, $C_0^\infty(\Omega)$ is dense in
$W^{1,4}_0(\Omega)$, and both terms in the
variational definition \eqref{def:GLin} are continuous
on this space. Consequently,
\[
   \G(a)=\sup_{u\in W^{1,4}_0(\Omega)}
   \left(2\int_\Omega fu\,\d x-\Psi(\nabla u)\right).
\]

We apply the Fenchel--Rockafellar duality theorem
to the gradient operator
\[
   \nabla:W^{1,4}_0(\Omega)
   \longrightarrow L^4(\Omega;\R^d).
\]
The continuity of $\Psi$ at the origin ensures
strong duality and attainment of the dual minimum
whenever the primal value is finite.
The convex conjugate of $\Psi$ is
\[
   \Psi^*(\tau)=
   \int_\Omega\frac{|\tau|^2}{4a}\,\d x,
   \qquad \tau\in L^{4/3}(\Omega;\R^d),
\]
with the same convention on the zero set of $a$.
The dual constraint reads
\[
   \int_\Omega\tau\cdot\nabla u\,\d x
   =2\int_\Omega fu\,\d x
   \qquad\text{for all }u\in W^{1,4}_0(\Omega),
\]
or equivalently,
$-\operatorname{div}\tau=2f$ in $\mathcal D'(\Omega)$.
Setting $j=\tau/2$ yields
\[
   \G(a)=\inf_{j\in\mathcal Q_f}
   \int_\Omega\frac{|j|^2}{a}\,\d x.
\]

If $\G(a)<+\infty$, the dual minimum is attained.
Every minimizing flux vanishes almost everywhere on
$\{a=0\}$, and the integrand $|j|^2/a$ is strictly
convex in $j$ on $\{a>0\}$. Hence the minimizing
flux $j_a$ is unique.
\end{proof}

We note that the quadratic structure of the flux functional
\eqref{eq:G-flux-representation} also appears in the
Benamou--Brenier formulation of optimal transport and its
extensions to nonlinear control systems; see
\cite{Elamvazhuthi2025}.
The difference is that our flux satisfies a stationary
divergence constraint, rather than a time-dependent
continuity equation.

\begin{remark}
The flux representation \eqref{eq:G-flux-representation}
holds for arbitrary nonnegative coefficients in
$L^2(\Omega)$, including degenerate coefficients
for which a weighted energy solution need not exist.
Thus the minimizing flux provides a canonical
variational state independently of the existence
of a pressure solution.

If $a\in\K_K$, the weighted energy solution $u[a]$
exists and the minimizing flux is given by
\[
   j_a=a\nabla u[a].
\]
Indeed, this flux belongs to $\mathcal Q_f$ and
satisfies
\[
   \int_\Omega\frac{|j_a|^2}{a}\,\d x
   =\int_\Omega a|\nabla u[a]|^2\,\d x
   =\G(a).
\]
%Notice that $j_a$ has the opposite sign to the Darcy flux $-a\nabla u[a]$.
%No normal-flux boundary condition is imposed, consistently with the homogeneous Dirichlet condition for the pressure.
\end{remark}

The next result shows that the definition \eqref{def:GLin}
is not merely a convenient extension of the pumping energy to degenerate conductivities.
In fact, it is precisely the lower-semicontinuous relaxation
of the classical pumping energy defined for bounded,
uniformly positive coefficients.

\begin{proposition}[Relaxation of the pumping energy]
\label{prop:G-relaxation}
Let $\Omega\subset\R^d$ be a bounded Lipschitz domain
and $f\in L^2(\Omega)$. Define
$\mathcal C:L^2(\Omega)\to\Rext$ by
\[
   \mathcal C(a):=
   \begin{cases}
   \displaystyle\int_\Omega a|\nabla u[a]|^2\,\d x,
   &\text{if }a\in L^\infty(\Omega)
   \text{ and }\operatorname*{ess\,inf}_\Omega a>0,\\[2mm]
   +\infty,&\text{otherwise},
   \end{cases}
\]
where $u[a]\in H_0^1(\Omega)$ is the unique weak
solution of
\[
   -\operatorname{div}(a\nabla u[a])=f
   \qquad\text{in }\Omega.
\]
Then the functional $\G$ defined in \eqref{def:GLin}
is the lower-semicontinuous envelope of $\mathcal C$
with respect to the strong topology of $L^2(\Omega)$.
In particular, for every sequence $a_n\to a$ in
$L^2(\Omega)$,
\begin{equation}\label{eq:G-relaxation-liminf}
   \G(a)\leq\liminf_{n\to\infty}\mathcal C(a_n).
\end{equation}
Conversely, for every $a\in L^2_+(\Omega)$ with
$\G(a)<+\infty$, there exists a sequence
$a_n\in L^\infty(\Omega)$ such that
\[
   \operatorname*{ess\,inf}_\Omega a_n>0,
   \qquad a_n\to a\quad\text{in }L^2(\Omega),
\]
and
\begin{equation}\label{eq:G-relaxation-recovery}
   \lim_{n\to\infty}\mathcal C(a_n)=\G(a).
\end{equation}
If $a\geq r_0>0$ almost everywhere on $\Omega$, the recovery sequence can
be chosen such that $a_n\geq r_0$ almost everywhere on $\Omega$.
\end{proposition}

\begin{proof}
We divide the proof into four steps.

\smallskip
\noindent
\emph{Step 1. Lower bound.}
For every bounded, uniformly positive coefficient $a$,
the energy solution belongs to $H_0^1(\Omega)$ and
the smooth-core representation gives $\mathcal C(a)=\G(a)$.
Consequently, $\G\leq\mathcal C$ on $L^2(\Omega)$.
Since $\G$ is lower semicontinuous by
Lemma~\ref{lem:Gconv1}, for every $a_n\to a$ in $L^2$
we have
\[
   \G(a)\leq\liminf_{n\to\infty}\G(a_n)
   \leq\liminf_{n\to\infty}\mathcal C(a_n).
\]
It remains to construct a recovery sequence.

\smallskip
\noindent
\emph{Step 2. Approximation of the conductivity and flux.}
Fix $a\in L^2_+(\Omega)$ with $\G(a)<+\infty$.
By Proposition~\ref{prop:flux-representation},
there exists $j\in\mathcal Q_f$ such that
\begin{equation}\label{eq:relaxation-optimal-flux}
   -\operatorname{div}j=f,\qquad
   \int_\Omega\frac{|j|^2}{a}\,\d x=\G(a).
\end{equation}
Our aim is to approximate the pair $(a,j)$ while
preserving the divergence constraint and controlling
the flux energy.

To avoid boundary effects during mollification, we first
extend the pair to a slightly larger domain.
By the geometric construction of~\cite{ErnGuermond2016},
for sufficiently small $\delta>0$, there exist
smooth diffeomorphisms $T_\delta$ of $\R^d$,
with inverses $F_\delta$, such that
\[
   \|DT_\delta-I\|_{L^\infty}
   +\|DF_\delta-I\|_{L^\infty}\leq C\delta,
\]
and
\[
   \Omega_\delta:=F_\delta(\Omega)
   \supset\Omega+B(0,c\delta)
\]
for some $c>0$ independent of $\delta$.
In particular, their Jacobian determinants
converge uniformly to one.

For $x\in\Omega$, set $y=F_\delta(x)$ and
$D_\delta(x)=\det DF_\delta(x)$. Define
on $\Omega_\delta$
\[
   a^\delta(y):=\frac{a(x)}{D_\delta(x)},\qquad
   j^\delta(y):=\frac{DF_\delta(x)j(x)}{D_\delta(x)},
   \qquad
   f^\delta(y):=\frac{f(x)}{D_\delta(x)},
\]
and extend these functions by zero outside
$\Omega_\delta$.

The transformation of the flux is chosen to preserve
its divergence. Indeed, for every
$\varphi\in C_0^\infty(\Omega_\delta)$,
the change of variables $y=F_\delta(x)$ gives
\begin{align*}
   \int_{\Omega_\delta}j^\delta\cdot\nabla\varphi\,\d y
   =\int_\Omega
     j\cdot\nabla(\varphi\circ F_\delta)\,\d x
   =\int_\Omega f(\varphi\circ F_\delta)\,\d x
   =\int_{\Omega_\delta}f^\delta\varphi\,\d y.
\end{align*}
Thus $-\operatorname{div}j^\delta=f^\delta$
in $\mathcal D'(\Omega_\delta)$.

Moreover, the change of variables and
$\|DF_\delta-I\|_{L^\infty}\leq C\delta$ yield
\begin{align}
   \int_{\Omega_\delta}
   \frac{|j^\delta|^2}{a^\delta}\,\d y
   &=\int_\Omega
     \frac{|DF_\delta j|^2}{a}\,\d x
   \notag\\
   &\leq(1+C\delta)
     \int_\Omega\frac{|j|^2}{a}\,\d x
   =(1+C\delta)\G(a).
   \label{eq:relaxation-transport-estimate}
\end{align}
Here and below $C$ denotes a constant independent
of $\delta$, whose value may change from line to line.

Let $\rho_\varepsilon$ be a standard nonnegative
mollifier and choose
$0<\varepsilon_\delta<c\delta/2$, with
$\varepsilon_\delta\to0$ sufficiently fast.
Define on $\Omega$
\[
   A_\delta:=\rho_{\varepsilon_\delta}*a^\delta,
   \qquad
   q_\delta:=\rho_{\varepsilon_\delta}*j^\delta,
   \qquad
   s_\delta:=\rho_{\varepsilon_\delta}*f^\delta.
\]
The convolutions are well defined throughout $\Omega$,
since the mollification balls are contained in
$\Omega_\delta$. Hence
\begin{equation}\label{eq:relaxation-mollified-div}
   -\operatorname{div}q_\delta=s_\delta
   \qquad\text{in }\mathcal D'(\Omega).
\end{equation}
The functions $A_\delta$ and $q_\delta$ are smooth
and bounded, $A_\delta\geq0$, and the near-identity
property of the transformations, together with
standard mollification, gives
\begin{equation}\label{eq:relaxation-mollified-convergence}
   A_\delta\to a,\qquad
   s_\delta\to f
   \qquad\text{strongly in }L^2(\Omega).
\end{equation}

The crucial observation is that the function
\[
   (s,z)\longmapsto\frac{|z|^2}{s},
   \qquad s\geq0,\quad z\in\R^d,
\]
with the convention of Proposition~
\ref{prop:flux-representation}, is jointly convex
and positively one-homogeneous.
Jensen's inequality therefore gives
\[
   \frac{|q_\delta|^2}{A_\delta}
   \leq\rho_{\varepsilon_\delta}*
   \left(\frac{|j^\delta|^2}{a^\delta}\right).
\]
Integrating over $\Omega$ and using
\eqref{eq:relaxation-transport-estimate}, we obtain
\begin{equation}\label{eq:relaxation-mollified-energy}
   \int_\Omega\frac{|q_\delta|^2}{A_\delta}\,\d x
   \leq(1+C\delta)\G(a).
\end{equation}

\smallskip
\noindent
\emph{Step 3. Correction of the source and uniform ellipticity.}
%The mollified flux satisfies $-\operatorname{div}q_\delta=s_\delta$, whereas we need the original source $f$.
Let $z_\delta\in H_0^1(\Omega)$ solve
\[
   -\Delta z_\delta=f-s_\delta
   \qquad\text{in }\Omega,
\]
and set $r_\delta:=\nabla z_\delta$.
By the standard energy estimate and
\eqref{eq:relaxation-mollified-convergence},
\begin{equation}\label{eq:relaxation-correction-estimate}
   \|r_\delta\|_{L^2(\Omega)}
   \leq C_\Omega\|f-s_\delta\|_{L^2(\Omega)}
   \longrightarrow0.
\end{equation}
The corrected flux
\[
   J_\delta:=q_\delta+r_\delta
\]
therefore satisfies
\[
   -\operatorname{div}J_\delta=f
   \qquad\text{in }\mathcal D'(\Omega),
\]
so that $J_\delta\in\mathcal Q_f$.

To ensure strict positivity of the conductivity,
set
\[
   t_\delta:=\delta+\|r_\delta\|_{L^2(\Omega)},
   \qquad
   a_\delta:=A_\delta+t_\delta.
\]
Then $a_\delta\in L^\infty(\Omega)$,
$a_\delta\geq t_\delta>0$, and
$a_\delta\to a$ strongly in $L^2(\Omega)$.

For any $A,t>0$ and vectors $q,r\in\R^d$,
the Cauchy--Schwarz inequality gives
\[
   \frac{|q+r|^2}{A+t}
   \leq\frac{|q|^2}{A}+\frac{|r|^2}{t}.
\]
The inequality extends to $A=0$ with the
usual convention. Applying it to
$J_\delta=q_\delta+r_\delta$ and
$a_\delta=A_\delta+t_\delta$, we deduce
\begin{align*}
   \int_\Omega\frac{|J_\delta|^2}{a_\delta}\,\d x
   &\leq
   \int_\Omega\frac{|q_\delta|^2}{A_\delta}\,\d x
   +\frac{\|r_\delta\|_{L^2(\Omega)}^2}{t_\delta}\\
   &\leq
   (1+C\delta)\G(a)+\|r_\delta\|_{L^2(\Omega)}.
\end{align*}
By the flux representation and
$\mathcal C(a_\delta)=\G(a_\delta)$,
\[
   \mathcal C(a_\delta)
   \leq\int_\Omega\frac{|J_\delta|^2}{a_\delta}\,\d x
   \leq(1+C\delta)\G(a)
      +\|r_\delta\|_{L^2(\Omega)}.
\]
Consequently,
\[
   \limsup_{\delta\to0}\mathcal C(a_\delta)
   \leq\G(a).
\]
Together with the lower bound from Step 1,
this proves the recovery property
\eqref{eq:G-relaxation-recovery}.

\smallskip
\noindent
\emph{Step 4. Preservation of a positive lower bound.}
Finally, suppose that $a\geq r_0>0$ a.e. on $\Omega$.
Replace the preceding recovery sequence by
\[
   \widetilde a_\delta:=\max\{a_\delta,r_0\}.
\]
The truncation is Lipschitz, hence
$\widetilde a_\delta\to a$ strongly in $L^2(\Omega)$.
Moreover, $\widetilde a_\delta$ remains bounded
and uniformly positive. By the order-reversing
property of the pumping energy,
\[
   \mathcal C(\widetilde a_\delta)
   =\G(\widetilde a_\delta)
   \leq\G(a_\delta)=\mathcal C(a_\delta).
\]
Thus the modified sequence also satisfies
\eqref{eq:G-relaxation-recovery}, and the proof
is complete.
\end{proof}

\begin{remark}[Variational meaning of the relaxation]
Proposition~\ref{prop:G-relaxation} shows that the
smooth-core functional $\G$ is the canonical
strong-$L^2$ relaxation of the classical
Dirichlet pumping energy. In particular, it is
the largest lower-semicontinuous functional
bounded above by $\mathcal C$.

This also explains why the common smooth test
space in \eqref{def:GLin} is essential.
For unbounded or degenerate coefficients, the
space of functions having finite weighted
Dirichlet integral may be strictly larger than
the closure of $C_0^\infty(\Omega)$ in the
weighted energy norm. Taking the supremum over
this larger space can therefore yield a
different functional, which need not coincide
with the relaxation of $\mathcal C$.

The recovery construction is adapted to the fixed
source $f$ and does not imply convergence of the
associated weighted Dirichlet forms along arbitrary
sequences of coefficients. In particular, it does
not assert convergence of the corresponding pressure
solutions.
\end{remark}

The weak cutoff enters the variational problem through the restriction
\[
   \G_K:=\G+\iota_{\K_K},
\]
where $\iota_{\K_K}$ is the indicator of $\K_K$. A sufficiently large
positive constant belongs to $\K_K$ and has finite $\G$ energy, so $\G_K$ is
proper. Lemmas~\ref{lem:KKcc} and~\ref{lem:Gconv1} show that it is also convex
and lower semicontinuous on $L^2(\Omega)$.

The common-core representation \eqref{def:GLin} also preserves the natural order: increasing
the conductivity can only decrease the value of the pumping-energy functional.
This property will be central when the conductivity is composed with a nonlinear
mobility.

\begin{lemma}\label{lem:EtildeMonotonicity}
Let $a_1,a_2\in L^2_+(\Omega)$ satisfy
\[
   a_1(x)\leq a_2(x) \qquad\text{for almost all } x\in\Omega.
\]
Then
\[
   \G(a_1) \geq \G(a_2).
\]
\end{lemma}

\begin{proof}
For every $u\in C_0^\infty(\Omega)$, the pointwise order $a_1\leq a_2$ gives
\[
   2\int_\Omega fu\,\d x-\int_\Omega a_1|\nabla u|^2\,\d x
   \geq
   2\int_\Omega fu\,\d x-\int_\Omega a_2|\nabla u|^2\,\d x.
\]
Taking the supremum over $C_0^\infty(\Omega)$ proves the claim.
\end{proof}

The auxiliary result concerns admissible variations around a fixed
coefficient. Estimates between coefficient-dependent energy spaces are
simplest when the comparison coefficient is bounded above by a multiple of
the reference coefficient. The following lemma shows that restricting first
to such coefficients entails no loss at the level of $L^2$ density.

%%%%%%%%%%%%%%%%%%%%%%%%%%%%%%%%%%%%%%%%%%%%%%%%%%%%%%%%%%%%
% New Lemma 3.6
\begin{lemma} %[Density of dominated coefficients]
\label{lem:BmuK-dominated-density}
Let $\mu\in C([0,+\infty))$ satisfy $\mu(0)=0$, and assume that $\mu$ is
strictly increasing on $[0,+\infty)$ and of at most linear growth at infinity.
For $K>0$ define the set
\begin{equation}\label{def:BmuK-density}
   \mathcal B_{\mu,K}
   :=\left\{b\in L^2_+(\Omega):\mu(b)\in\K_K\right\}
   =\left\{b\in L^2_+(\Omega): \mu(b)\in L^2(\Omega),\,
      \int_\Omega \mu(b)^{-d/2}\,\d x\leq K\right\}.
\end{equation}
For $b_0\in\mathcal B_{\mu,K}$ denote
\begin{equation}\label{def:Amu-density}
   A_\mu(b_0)
   :=\bigcup_{n=1}^\infty
   \left\{b\in\mathcal B_{\mu,K}:
      b(x)\leq n b_0(x)\ \text{for almost every }x\in\Omega\right\}.
\end{equation}
Then $A_\mu(b_0)$ is dense in $\mathcal B_{\mu,K}$ with respect to the
$L^2(\Omega)$ topology.
\end{lemma}

\begin{proof}
Note that the (growth) assumptions on $\mu$ imply that $\mu(b) \in L^2 _+(\Omega)$ if $b \in L^2 _+ (\Omega)$.
We first create slack in the inverse-moment constraint. Define
\[
   \widetilde{\mathcal B}_{\mu,K}
   :=\left\{b\in\mathcal B_{\mu,K}:
      \int_\Omega\mu(b)^{-d/2}\,\d x<K\right\}.
\]
Given $b_1\in\mathcal B_{\mu,K}$ and $\varepsilon>0$, let
\[
   b_1^\varepsilon:=(1+\varepsilon)b_1.
\]
Then
$\mu(b_1^\varepsilon)>\mu(b_1)$ on the set $\{b_1>0\}$. Moreover, the
integrability of $\mu(b_1)^{-d/2}$, together with $\mu(0)=0$, implies that
$\{b_1=0\}$ has Lebesgue measure zero. Consequently,
\[
   \int_\Omega\mu(b_1^\varepsilon)^{-d/2}\,\d x
   =\int_{\{b_1>0\}}\mu(b_1^\varepsilon)^{-d/2}\,\d x
   <\int_\Omega\mu(b_1)^{-d/2}\,\d x
   \leq K.
\]
Since $b_1^\varepsilon\to b_1$ in $L^2(\Omega)$ as
$\varepsilon\to 0^+$, it follows that
$\widetilde{\mathcal B}_{\mu,K}$ is dense in
$\mathcal B_{\mu,K}$.

Now fix $b\in\widetilde{\mathcal B}_{\mu,K}$ and define
\[
   b_n:=\min\{b,n b_0\},
   \qquad n\in\mathbb N.
\]
Since $b_0>0$ almost everywhere in $\Omega$, we have
$b_n(x)\to b(x)$ as $n\to\infty$ almost everywhere in $\Omega$. Moreover,
$0\leq b_n\leq b$, and hence the dominated convergence theorem gives
$b_n\to b$ in $L^2(\Omega)$.
Set
\[
   B_n:=\{x\in\Omega:b(x)\leq n b_0(x)\}.
\]
Then
\[
   b_n=b\,\mathbf 1_{B_n}+n b_0\,\mathbf 1_{\Omega\setminus B_n},
\]
and therefore, using the monotonicity of $\mu$,
\begin{align*}
   \int_\Omega\mu(b_n)^{-d/2}\,\d x
   &=\int_{B_n}\mu(b)^{-d/2}\,\d x
     +\int_{\Omega\setminus B_n}\mu(n b_0)^{-d/2}\,\d x\\
   &\leq\int_{B_n}\mu(b)^{-d/2}\,\d x
     +\int_{\Omega\setminus B_n}\mu(b_0)^{-d/2}\,\d x.
\end{align*}
The first integral is bounded above by
\[
   \int_\Omega\mu(b)^{-d/2}\,\d x<K.
\]
Furthermore, $\mathbf 1_{\Omega\setminus B_n}\to0$ almost everywhere, and
$\mu(b_0)^{-d/2}\in L^1(\Omega)$. Hence
\[
   \int_{\Omega\setminus B_n}\mu(b_0)^{-d/2}\,\d x\to 0
   \qquad\text{as }n\to\infty.
\]
It follows that, for all sufficiently large $n$,
\[
   \int_\Omega\mu(b_n)^{-d/2}\,\d x<K.
\]
Thus $b_n\in\mathcal B_{\mu,K}$ and $b_n\leq n b_0$ almost everywhere,
so $b_n\in A_\mu(b_0)$. Since $b_n\to b$ in $L^2(\Omega)$ and
$\widetilde{\mathcal B}_{\mu,K}$ is dense in
$\mathcal B_{\mu,K}$, the proof is complete.
\end{proof}

We have therefore retained, under the weak cutoff, both coefficient-wise
elliptic solvability and the convex, order-reversing structure of the relaxed
pumping-energy functional. Lemma~\ref{lem:BmuK-dominated-density} provides
a dense family of coefficients dominated by a multiple of a fixed reference
coefficient.
%An analogous density property after nonlinear reparametrization is not automatic and will therefore be imposed explicitly where it is needed.

%%%%%%%%%%%%%%%%%%%%%%%%%%%%%%%%%%%%%%%%%%%%%%%%%%%%%%%%%%%%
\subsection{Gradient flow under the weak cutoff}
\label{subsec:WeakCutoffGradientFlow}
%%%%%%%%%%%%%%%%%%%%%%%%%%%%%%%%%%%%%%%%%%%%%%%%%%%%%%%%%%%%

We now derive the $L^2$-gradient flow associated with the weak cutoff.
We recall the functional $\Gamma$ from \eqref{def:Gamma} and set
\begin{equation}\label{def:EKweak}
   \E_K(a):=\G(a)+\Gamma(a)+\iota_{\K_K}(a)
   =\G_K(a)+\Gamma(a),
   \qquad a\in L^2(\Omega).
\end{equation}
Here $\iota_{\K_K}$ denotes the indicator of the closed convex set $\K_K$.
The abstract flow can be constructed without identifying the subdifferential
pointwise. This distinction is important because the weak cutoff guarantees
solvability of the elliptic state equation, but it does not imply the
$L^4$-regularity of the pressure gradient needed for a pointwise $L^2$
subgradient.

\begin{theorem}[Existence and uniqueness under the weak cutoff]
\label{thm:GFweak-existence}
Let $d\geq2$, $K>0$, $f\in L^2(\Omega)$, and let
$g:[0,+\infty)\to[0,+\infty)$ be convex and continuous at zero. Then
$\E_K$ is proper, convex and lower semicontinuous on $L^2(\Omega)$, and
\begin{equation}\label{eq:closure-dom-EKweak}
   \overline{\operatorname{dom}(\E_K)}^{\,L^2(\Omega)}=\K_K.
\end{equation}
Consequently, for every $a^0\in\K_K$ there exists a unique
$a\in C([0,+\infty);\K_K)$ satisfying
\begin{equation}\label{eq:GFweak-abstract}
   \frac{\d a}{\d t}(t)+\partial\E_K(a(t))\ni0,
   \qquad t>0,
   \qquad a(0)=a^0,
\end{equation}
in the sense of maximal-monotone operators.

If $a^0\in\operatorname{dom}(\E_K)$, then
$a\in H^1(0,T;L^2(\Omega))$ for every $T>0$ and
\begin{equation}\label{eq:GFweak-minimal}
   \frac{\d a}{\d t}(t)+\partial^\circ\E_K(a(t))=0
   \qquad\text{for almost every }t>0,
\end{equation}
where $\partial^\circ\E_K(a)$ denotes the element of minimal $L^2$ norm in
$\partial\E_K(a)$. Moreover,
\begin{equation}\label{eq:GFweak-energy}
   \E_K(a(t))
   +\int_0^t\left\|\frac{\d a}{\d s}(s)\right\|_{L^2(\Omega)}^2\,\d s
   =\E_K(a^0)
   \qquad(t\geq0).
\end{equation}
\end{theorem}

\begin{proof}
The convexity and lower semicontinuity of $\G$, $\Gamma$, and
$\iota_{\K_K}$ follow from Lemmas~\ref{lem:Gconv1}, \ref{lem:GammaLSC}, and
\ref{lem:KKcc}, respectively. %Their sum therefore has the same properties.
To prove properness, choose a constant $c>0$ sufficiently large that
$|\Omega|c^{-d/2}\leq K$. Then $c\in\K_K$,
$\Gamma(c)=|\Omega|g(c)<+\infty$, and Lemma~\ref{lem:Riesz} gives
$\G(c)<+\infty$.

It remains to verify \eqref{eq:closure-dom-EKweak}. Since
$\operatorname{dom}(\E_K)\subseteq\K_K$ and $\K_K$ is closed, it suffices
to approximate an arbitrary $a\in\K_K$ by finite-energy elements of
$\K_K$. Set $p:=d/2$ and
\[
   \widetilde a_n:=\min\{a,n\},
   \qquad
   I_n:=\int_\Omega \widetilde a_n^{-p}\,\d x,
   \qquad
   c_n:=\max\left\{1,\left(\frac{I_n}{K}\right)^{1/p}\right\}.
\]
Since
$\widetilde a_n^{-p}=\max\{a^{-p},n^{-p}\}$, dominated convergence gives
$I_n\to\int_\Omega a^{-p}\,\d x\leq K$, and hence $c_n\to1$. The bounded
functions $a_n:=c_n\widetilde a_n$ satisfy
\[
   \int_\Omega a_n^{-p}\,\d x=c_n^{-p}I_n\leq K,
\]
so $a_n\in\K_K$. Moreover, $a_n\to a$ in $L^2(\Omega)$. Since $a_n$ is
bounded, the continuity of $g$ on bounded intervals gives
$\Gamma(a_n)<+\infty$, while Lemma~\ref{lem:Riesz} gives
$\G(a_n)<+\infty$. Thus $a_n\in\operatorname{dom}(\E_K)$ and
\eqref{eq:closure-dom-EKweak} follows.

Since $\E_K$ is proper, convex and lower semicontinuous, its $L^2$
subdifferential is maximal monotone. Moreover,
\[
   \overline{\operatorname{dom}(\partial\E_K)}^{\,L^2(\Omega)}
   =\overline{\operatorname{dom}(\E_K)}^{\,L^2(\Omega)}
   =\K_K.
\]
The Br\'{e}zis--Komura \cite{Brezis,Komura} theorem therefore yields the unique semigroup solution
of \eqref{eq:GFweak-abstract}. For finite-energy initial data, the standard
regularity theory for subgradient evolutions gives
\eqref{eq:GFweak-minimal}. The convex chain rule then yields
\[
   \frac{\d}{\d t}\E_K(a(t))
   =\left\langle\partial^\circ\E_K(a(t)),
      \frac{\d a}{\d t}(t)\right\rangle_{L^2}
   =-\left\|\frac{\d a}{\d t}(t)\right\|_{L^2(\Omega)}^2,
\]
and integration proves \eqref{eq:GFweak-energy}.
\end{proof}

We next identify the subdifferential of the pumping-energy part. For
$a\in\K_K$, let $N_{\K_K}(a)$ denote the $L^2$ normal cone to $\K_K$ at
$a$.

\begin{lemma}[Subdifferential of the constrained pumping energy]
\label{lem:subdiff-GK}
Let $a\in\K_K$, let $u[a]\in\Ha{a}$ be the solution supplied by
Lemma~\ref{lem:Riesz}, and assume that $\nabla u[a]\in L^4(\Omega)$. Then
\begin{equation}\label{eq:subdiff-GK}
   \partial\G_K(a)=-|\nabla u[a]|^2+N_{\K_K}(a).
\end{equation}
\end{lemma}

\begin{proof}
We first compute the directional derivative along the comparable
coefficients of Lemma~\ref{lem:BmuK-dominated-density}. Let $b\in\K_K$
satisfy $b\leq La$ almost everywhere for some $L>0$, and set
$a_t:=(1-t)a+tb$ for $0<t<1$. Then
\[
   (1-t)a\leq a_t\leq(1-t+tL)a,
\]
so the spaces $\Ha{a_t}$ and $\Ha{a}$ coincide with equivalent norms for
$t$ in a sufficiently small interval. Writing $u_t:=u[a_t]$ and
$u:=u[a]$, we may therefore cross-test their weak equations. This gives
\begin{equation}\label{eq:GK-difference-quotient}
   \frac{\G(a_t)-\G(a)}{t}
   =-\int_\Omega(b-a)\nabla u_t\cdot\nabla u\,\d x.
\end{equation}
The energy estimate from Lemma~\ref{lem:Riesz} and the preceding
comparability show that $(u_t)$ is bounded in $\Ha{a}$. Passing to a weakly
convergent subsequence, testing first against $C_0^\infty(\Omega)$, and using
uniqueness of the energy solution, we obtain
\[
   u_t\rightharpoonup u
   \qquad\text{weakly in }\Ha{a}
   \quad\text{as }t\to 0^+.
\]
Since $(b-a)/a\in L^\infty(\Omega)$, we may pass to the limit in
\eqref{eq:GK-difference-quotient} and find
\begin{equation}\label{eq:GK-directional-comparable}
   \lim_{t\downarrow0}\frac{\G(a+t(b-a))-\G(a)}{t}
   =-\int_\Omega |\nabla u[a]|^2(b-a)\,\d x.
\end{equation}

Let $\xi\in\partial\G_K(a)$. Applying the subgradient inequality to
$a+t(b-a)$ and then using \eqref{eq:GK-directional-comparable}, we obtain
\[
   \int_\Omega\bigl(\xi+|\nabla u[a]|^2\bigr)(b-a)\,\d x\leq0
\]
for every comparable $b$. Lemma~\ref{lem:BmuK-dominated-density} and the
$L^2$ regularity of $|\nabla u[a]|^2$ extend this inequality to every
$b\in\K_K$. Hence
\[
   \xi+|\nabla u[a]|^2\in N_{\K_K}(a).
\]

Conversely, $u[a]\in W^{1,q}_0(\Omega)$ with $q=2d/(d+2)$, while the
additional assumption $\nabla u[a]\in L^4(\Omega)$ implies
$u[a]\in W^{1,4}_0(\Omega)$. We may therefore approximate $u[a]$ by
$C_0^\infty(\Omega)$-functions in $W^{1,4}_0(\Omega)$.
Using $u[a]$ in the common-core representation \eqref{def:GLin}, we obtain
for every $b\in L^2_+(\Omega)$
\[
   \G(b)
   \geq2\int_\Omega fu[a],\d x
      -\int_\Omega b|\nabla u[a]|^2\,\d x
   =\G(a)-\int_\Omega|\nabla u[a]|^2(b-a)\,\d x.
\]
Thus $-|\nabla u[a]|^2\in\partial\G(a)$. Adding any element of
$N_{\K_K}(a)$ gives a subgradient of $\G_K$, proving
\eqref{eq:subdiff-GK}.
\end{proof}

The preceding formula gives a particularly simple evolution when
there is no metabolic cost.

\begin{corollary}[Pumping-energy flow]
\label{cor:GFweak-pumping}
Let $g\equiv0$ and let $a$ be the solution supplied by
Theorem~\ref{thm:GFweak-existence}. Assume that
\[
   \nabla u[a(t)]\in L^4(\Omega)
   \qquad\text{for almost every }t>0.
\]
Then
\begin{equation}\label{eq:GFweak-pumping}
   \partial_t a=|\nabla u[a]|^2
   \qquad\text{almost everywhere in }\Omega\times(0,+\infty).
\end{equation}
In particular, $a(\cdot,t)$ is nondecreasing in time almost everywhere,
and
\begin{equation}\label{eq:inverse-moment-dissipation}
   \int_\Omega a(x,t)^{-d/2}\,\d x
   \leq
   \int_\Omega a(x,s)^{-d/2}\,\d x
   \qquad(0\leq s\leq t).
\end{equation}
Moreover,
\begin{equation}\label{eq:GFweak-pumping-energy}
   \G(a(t))
   +\int_0^t\int_\Omega|\nabla u[a(s)]|^4\,\d x\,\d s
   =\G(a^0).
\end{equation}
\end{corollary}

\begin{proof}
The set $\K_K$ is upward closed: if $b\in L^2(\Omega)$ and $b\geq a$
almost everywhere, then $b\in\K_K$. Consequently, every
$\eta\in N_{\K_K}(a)$ satisfies $\eta\leq0$ almost everywhere. Since
$|\nabla u[a]|^2\geq0$, formula \eqref{eq:subdiff-GK} implies that the
element of minimal norm in $\partial\G_K(a)$ is
$-|\nabla u[a]|^2$. Equation \eqref{eq:GFweak-pumping} now follows from
\eqref{eq:GFweak-minimal}. The monotonicity of $a$ and
\eqref{eq:inverse-moment-dissipation} follow immediately. Finally,
substituting \eqref{eq:GFweak-pumping} into
\eqref{eq:GFweak-energy} gives \eqref{eq:GFweak-pumping-energy}.
\end{proof}

Now we consider the case when the metabolic cost is present.
Define
\begin{equation}\label{def:Phi-weak}
   \Phi(a):=\int_\Omega a^{-d/2}\,\d x,
   \qquad
   z_a:=a^{-d/2-1}.
\end{equation}
The next lemma records the additional regularity needed to represent the
normal cone by a scalar multiplier.

\begin{lemma}[Normal cone to the weak-cutoff set]
\label{lem:normal-KK}
Let $a\in\K_K$ and assume that $z_a\in L^2(\Omega)$. Then
\begin{equation}\label{eq:normal-KK}
   N_{\K_K}(a)
   =
   \begin{cases}
      \{0\},&\Phi(a)<K,\\[1mm]
      \{-\Lambda z_a:\ \Lambda\geq0\},&\Phi(a)=K.
   \end{cases}
\end{equation}
\end{lemma}

\begin{proof}
We first identify the tangent cone. If $\Phi(a)<K$, localize and truncate an
arbitrary $h\in L^2(\Omega)$ by
\[
   h_n:=h\,\mathbf 1_{\{1/n\leq a\leq n,\ |h|\leq n\}}.
\]
Then $h_n\to h$ in $L^2(\Omega)$, and for each fixed $n$ one has
$a+t h_n>0$ and $\Phi(a+t h_n)<K$ for all sufficiently small $t>0$.
Hence $h_n\in T_{\K_K}(a)$ and therefore
$T_{\K_K}(a)=L^2(\Omega)$.

Suppose next that $\Phi(a)=K$. Convexity of $s\mapsto s^{-d/2}$ gives, for
every $b\in\K_K$,
\[
   \Phi(b)\geq\Phi(a)-\frac d2\int_\Omega z_a(b-a)\,\d x.
\]
Since $\Phi(b)\leq\Phi(a)=K$, every tangent direction $h$ satisfies
$\int_\Omega z_a h\,\d x\geq0$. Conversely, if this integral is strictly
positive, the preceding localized truncations satisfy
$\int_\Omega z_a h_n\,\d x>0$ for all sufficiently large $n$. For each such
$n$,
\[
   \left.\frac{\d}{\d t}\Phi(a+t h_n)\right|_{t=0+}
   =-\frac d2\int_\Omega z_a h_n\,\d x<0,
\]
and hence $a+t h_n\in\K_K$ for all sufficiently small $t>0$. Thus
$h\in T_{\K_K}(a)$. If $\int_\Omega z_a h\,\d x=0$, apply the preceding
argument to $h+\varepsilon a$ and let $\varepsilon\to 0^+$, noting that
$\int_\Omega z_a a\,\d x=\Phi(a)=K$. Consequently,
\[
   T_{\K_K}(a)
   =\left\{h\in L^2(\Omega):
      \int_\Omega z_a h\,\d x\geq0\right\}
   \qquad\text{when }\Phi(a)=K.
\]
Taking the polar cone gives \eqref{eq:normal-KK}; the factor $d/2$ has been
absorbed into $\Lambda$.
\end{proof}

We can now identify the full pumping--metabolic flow. The following
assumptions are stronger than the weak cutoff itself and are stated
explicitly for this reason.

\begin{corollary}[Pointwise form of the weak-cutoff flow]
\label{cor:GFweak-identified}
Let the assumptions of Theorem~\ref{thm:GFweak-existence} hold, assume in
addition that $g\in C^1([0,+\infty))$, and let $a$ be the solution of
\eqref{eq:GFweak-abstract} with $a^0\in\operatorname{dom}(\E_K)$. Suppose
that, for almost every $t>0$,
\begin{equation}\label{eq:GFweak-regularity}
   \nabla u[a(t)]\in L^4(\Omega),
   \qquad
   g'(a(t))\in L^2(\Omega),
   \qquad
   a(t)^{-d/2-1}\in L^2(\Omega),
\end{equation}
and that the subdifferential sum rule
\begin{equation}\label{eq:GFweak-sum-rule}
   \partial\E_K(a(t))
   =\partial\G_K(a(t))+\partial\Gamma(a(t))
\end{equation}
is valid. Then, with $u(t):=u[a(t)]$ and
$z_a(t):=a(t)^{-d/2-1}$,
\begin{equation}\label{eq:GFweak-identified}
   \partial_t a
   =|\nabla u|^2-g'(a)+\Lambda(t)z_a(t)
   \qquad\text{almost everywhere in }\Omega\times(0,+\infty),
\end{equation}
where $\Lambda(t)\geq0$ satisfies
\begin{equation}\label{eq:GFweak-complementarity}
   \Lambda(t)\bigl(\Phi(a(t))-K\bigr)=0.
\end{equation}
More precisely, the minimal-section selection in
\eqref{eq:GFweak-minimal} gives
\begin{equation}\label{eq:GFweak-multiplier}
   \Lambda(t)
   =
   \begin{cases}
      0,&\Phi(a(t))<K,\\[2mm]
      \displaystyle
      \frac{\left[
         \int_\Omega
         \bigl(g'(a(t))-|\nabla u(t)|^2\bigr)z_a(t)\,\d x
      \right]_+}
      {\displaystyle\int_\Omega z_a(t)^2\,\d x},
      &\Phi(a(t))=K.
   \end{cases}
\end{equation}
\end{corollary}

\begin{proof}
Since $a(t)>0$ almost everywhere, the assumptions on $g$ and
$g'(a(t))$ give
\[
   \partial\Gamma(a(t))=\{g'(a(t))\}.
\]
Combining the sum rule \eqref{eq:GFweak-sum-rule},
Lemma~\ref{lem:subdiff-GK}, and Lemma~\ref{lem:normal-KK}, we obtain
\[
   \partial\E_K(a)
   =
   \left\{
      g'(a)-|\nabla u[a]|^2-\Lambda z_a:
      \Lambda\geq0,\quad \Lambda(\Phi(a)-K)=0
   \right\}.
\]
At an active constraint, the element of smallest $L^2$ norm is obtained by
minimizing
\[
   \left\|g'(a)-|\nabla u[a]|^2-\Lambda z_a\right\|_{L^2(\Omega)}^2
   \qquad\text{over }\Lambda\geq0,
\]
which gives \eqref{eq:GFweak-multiplier}. Equations
\eqref{eq:GFweak-identified} and \eqref{eq:GFweak-complementarity} now follow
from \eqref{eq:GFweak-minimal}.
\end{proof}

\begin{remark}
The regularity assumptions in \eqref{eq:GFweak-regularity} are not
consequences of the weak cutoff. In particular,
$a^{-d/2}\in L^1(\Omega)$ does not imply
$a^{-d/2-1}\in L^2(\Omega)$. Likewise, the sum rule
\eqref{eq:GFweak-sum-rule} is a separate qualification because
$\K_K$ has empty interior in $L^2(\Omega)$. A convenient sufficient
condition for this sum rule is that $g'$ have at most linear growth. Indeed,
one may then extend $g$ affinely to $(-\infty,0)$; the resulting integral
functional is continuously differentiable on $L^2(\Omega)$, while its values
on $\K_K$ agree with those of $\Gamma$.

For $g\equiv0$, formula \eqref{eq:GFweak-multiplier} gives
$\Lambda(t)=0$, consistently with Corollary~\ref{cor:GFweak-pumping}. Thus
the inverse-moment constraint is preserved by the pumping-energy flow but
does not exert an active force: the velocity $|\nabla u[a]|^2$ increases the
conductivity and therefore moves the solution further into the admissible
set.
\end{remark}

%%%%%%%%%%%%%%%%%%%%%%%%%%%%%%%%%%%%%%%
\section{Nonlinear mobility}\label{sec:NonlinearMobility}
The preceding sections treated the physical conductivity $a$ itself as the
variable of the pumping-energy functional. We now introduce a constitutive law
$a=\mu(b)$, where $b$ is the quantity that evolves. Although the relaxed
pumping-energy functional is convex in $a$, its composition with $\mu$ need
not retain the
same curvature in $b$. The purpose of this section is to identify, through a
formal second-variation calculation, two structural conditions on $\mu$:
concavity of $\mu$, which leads to formal convexity of the pumping-energy
functional, and
concavity of $1/\mu$, which provides a sufficient criterion for its formal
concavity. These conditions motivate
the two nonlinear regimes studied in Sections~\ref{sec:ConcaveMobility}
and~\ref{sec:ReciprocalConcave}.

Let $\mu:[0,\infty)\to[0,\infty)$ satisfy $\mu(s)>0$ for $s>0$ and be twice
continuously differentiable on $(0,\infty)$.
For $b>0$ almost everywhere on $\Omega$, let us consider the functional
\begin{equation}\label{eq:J-def}
  J_\mu(b) := \frac12 \int_\Omega \mu(b)\, |\nabla u[\mu(b)]|^2 \,\d x,
\end{equation}
where $u=u[\mu(b)]$ is the solution of the Dirichlet problem
\begin{equation}\label{eq:pde}
   \begin{aligned}
      -\operatorname{div} \left( \mu(b) \nabla u \right) &= f \qquad\mbox{in } \Omega,\\
      u &= 0 \qquad\mbox{on } \partial\Omega,
   \end{aligned}
\end{equation}
whenever it is uniquely defined.

Under the uniformly elliptic assumptions imposed below, the energy
representation from Section~\ref{sec:WeakCutoff} gives
\[
   J_\mu(b)=\frac12\G(\mu(b)).
\]
Thus the calculation concerns the nonlinear composition of the
relaxed pumping-energy functional introduced above; the factor $1/2$ is only a
convenient normalization.

To justify the calculation in a simple auxiliary setting, take
$b_0,b_1\in L^\infty(\Omega)$ with
$\mathop{\rm ess\,inf}_\Omega b_0>0$. For all sufficiently small $|t|$, the
functions $b_0+t b_1$ then take their values in a compact subset of
$(0,+\infty)$. Consequently, there are constants $0<c<C<+\infty$ such that
\[
   c\leq\mu(b_0+t b_1)\leq C
   \qquad\text{almost everywhere in }\Omega
\]
for all such $t$, and the coefficient depends twice differentiably on $t$ in
$L^\infty(\Omega)$. Standard parameter dependence for uniformly elliptic
problems then gives twice differentiable dependence of the solution in
$H^1_0(\Omega)$. The calculation below is therefore rigorous under these
temporary assumptions, but it is used only to identify the relevant
curvature conditions. The subsequent sections derive the corresponding
variational conclusions under substantially weaker hypotheses.

Fix $b_0$ and a perturbation direction $b_1$, and set
\begin{equation*}
  u_j:=\left.\frac{\d^j}{\d t^j}u[\mu(b_0+t b_1)]\right|_{t=0},
  \quad j=0,1,2,\qquad
  \mu_0:=\mu(b_0).
\end{equation*}
%Equivalently, for $a_t=a_0+t a_1$,
%\begin{equation*}   u[a_t]=u_0+t u_1+\frac{t^2}{2}u_2+o(t^2).   \end{equation*}
In the formulas below, we denote
\begin{equation*}
  \mu'_0:=\mu'(b_0),\qquad \mu''_0:=\mu''(b_0).
\end{equation*}
%and derivatives of $\mu$ are evaluated pointwise at $a_0(x)$.
%The derivatives of the state are determined by the linearized elliptic equations.
Differentiating \eqref{eq:pde} once gives
\begin{equation}\label{eq:first-diff-pde}
  -\operatorname{div}\bigl(\mu'_0 b_1\nabla u_0+\mu_0\nabla u_1\bigr)=0,
  \qquad u_1|_{\partial\Omega}=0.
\end{equation}
A second differentiation gives
\begin{equation}\label{eq:second-diff-pde}
  -\operatorname{div}\bigl(
    \mu''_0 b_1^2\nabla u_0
    +2\mu'_0 b_1\nabla u_1
    +\mu_0\nabla u_2
  \bigr)=0,
  \qquad u_2|_{\partial\Omega}=0.
\end{equation}
We first compute the variation of $J_\mu$. Direct differentiation of
\eqref{eq:J-def} at $b_0$ in the direction $b_1$ gives
\begin{align*}
  DJ_\mu(b_0)[b_1]
  &= \frac12\int_\Omega \mu'_0 b_1 |\nabla u_0|^2\d x
     +\int_\Omega \mu_0\nabla u_0\cdot\nabla u_1\d x.
\end{align*}
The term involving the implicit state variation can be eliminated using the
linearized equation. Indeed, testing \eqref{eq:first-diff-pde} by $u_0$ gives
\begin{equation*}
  \int_\Omega \mu_0\nabla u_0\cdot\nabla u_1\d x
  = -\int_\Omega \mu'_0 b_1|\nabla u_0|^2\d x,
\end{equation*}
and therefore
\[
  DJ_\mu(b_0)[b_1]
  = -\frac12\int_\Omega \mu'_0 b_1|\nabla u_0|^2\d x.
\]
%Thus the first variation depends only on the unperturbed state $u_0$.

The second variation of $J_\mu(b_0)$ in the direction $b_1$ reads
\begin{align*}
  D^2J_\mu(b_0)[b_1,b_1]
  &=\frac12\int_\Omega \mu''_0 b_1^2|\nabla u_0|^2\d x
    +\int_\Omega \mu'_0 b_1\nabla u_0\cdot\nabla u_1\d x\notag\\
  &\quad
    +\int_\Omega \mu'_0 b_1\nabla u_0\cdot\nabla u_1\d x
    +\int_\Omega \mu_0|\nabla u_1|^2\d x
    +\int_\Omega \mu_0\nabla u_0\cdot\nabla u_2\d x\notag\\
  &=\frac12\int_\Omega \mu''_0 b_1^2|\nabla u_0|^2\d x
    +2\int_\Omega \mu'_0 b_1\nabla u_0\cdot\nabla u_1\d x\notag\\
  &\quad
    +\int_\Omega \mu_0|\nabla u_1|^2\d x
    +\int_\Omega \mu_0\nabla u_0\cdot\nabla u_2\d x.
\end{align*}
%The linearized state equations allow us to eliminate these auxiliary terms.
Testing \eqref{eq:first-diff-pde} by $u_1$ gives
\begin{equation}\label{eq:first-diff-test-u1}
  \int_\Omega \mu_0|\nabla u_1|^2\d x
  =-\int_\Omega \mu'_0 b_1\nabla u_0\cdot\nabla u_1\d x.
\end{equation}
Substituting this identity into the preceding expression eliminates the mixed
term and gives
\[
  D^2J_\mu(b_0)[b_1,b_1]
  =\frac12\int_\Omega \mu''_0 b_1^2|\nabla u_0|^2\d x
    -\int_\Omega \mu_0|\nabla u_1|^2\d x
    +\int_\Omega \mu_0\nabla u_0\cdot\nabla u_2\d x.
\]
%It remains to eliminate the second state derivative.
Moreover, testing \eqref{eq:second-diff-pde} by $u_0$ gives
\[
  \int_\Omega \mu_0\nabla u_2\cdot\nabla u_0\d x
  =-\int_\Omega \mu''_0 b_1^2|\nabla u_0|^2\d x
   -2\int_\Omega \mu'_0 b_1\nabla u_1\cdot\nabla u_0\,\d x.
\]
Combining this relation with \eqref{eq:first-diff-test-u1} yields
\begin{equation}\label{eq:second-var-J-final}
  D^2J_\mu(b_0)[b_1,b_1]
  =-\frac12\int_\Omega \mu''_0 b_1^2|\nabla u_0|^2\d x
   +\int_\Omega \mu_0|\nabla u_1|^2\d x.
\end{equation}
%The two terms in \eqref{eq:second-var-J-final} have a transparent interpretation. The second one is the nonnegative contribution generated by
%the response of the elliptic state, whereas the sign of the first one is determined directly by the curvature of $\mu$.
We observe that if $\mu$ is concave on $(0,\infty)$, i.e., $\mu''\leq0$, then
\begin{equation*}
  D^2J_\mu(b_0)[b_1,b_1]\geq 0,
\end{equation*}
and thus $J_\mu$ is formally convex.

Formula \eqref{eq:second-var-J-final} also yields a complementary sufficient criterion for concavity.
%In this case the positive state-response term must be controlled by the term containing $\mu''$.
Applying the Cauchy--Schwarz inequality followed by Young's inequality in
\eqref{eq:first-diff-test-u1}, we obtain
\begin{align}\label{eq:young-estimate}
  \int_\Omega \mu_0|\nabla u_1|^2\d x
  &=-\int_\Omega \mu'_0 b_1\nabla u_0\cdot\nabla u_1\d x\notag\\
  &=-\int_\Omega \sqrt{\mu_0}\nabla u_1\cdot
       \frac{\mu'_0}{\sqrt{\mu_0}} b_1\nabla u_0\d x\notag\\
  &\leq \frac{1}{2}\int_\Omega \mu_0|\nabla u_1|^2\d x
       +\frac{1}{2}\int_\Omega
       \frac{(\mu'_0)^2}{\mu_0} b_1^2|\nabla u_0|^2\d x,
\end{align}
hence
\begin{equation}\label{eq:u1-estimate}
  \int_\Omega \mu_0|\nabla u_1|^2\d x
  \leq  \int_\Omega \frac{(\mu'_0)^2}{\mu_0} b_1^2|\nabla u_0|^2\d x.
\end{equation}
Substituting \eqref{eq:u1-estimate} into \eqref{eq:second-var-J-final} gives
\begin{equation}\label{eq:D2J-upper-epsilon}
  D^2J_\mu(b_0)[b_1,b_1]
  \leq  \int_\Omega \left( -\frac{\mu''_0}{2} + \frac{(\mu'_0)^2}{\mu_0} \right) b_1^2|\nabla u_0|^2\d x.
\end{equation}
Therefore, a sufficient pointwise condition for the formal concavity of
$J_\mu$ is
\begin{equation}\label{eq:concavity-condition}
  -\frac{\mu''(s)}{2}+\frac{(\mu'(s))^2}{\mu(s)}\leq 0.
\end{equation}
The identity
\[
   \left(\frac{1}{\mu}\right)''  = \frac{2(\mu')^2-\mu\mu''}{\mu^3}
\]
shows that \eqref{eq:concavity-condition} is equivalent to $(1/\mu)''\leq0$ on $(0,+\infty)$.
%Thus the reciprocal mobility $1/\mu$, rather than $\mu$ itself, is concave in this second regime.

The calculation therefore singles out two complementary cases. When $\mu$
is concave, the pumping-energy functional is formally convex and can be
combined directly
with a convex metabolic cost; this is the setting of
Section~\ref{sec:ConcaveMobility}. When $1/\mu$ is concave, the pumping-energy
functional is formally concave, so that its negative has the curvature needed for a convex
energy; this case is treated in Section~\ref{sec:ReciprocalConcave}.
The subsequent sections establish these curvature properties through primal and flux variational
representations. Lower semicontinuity is established in the concave-mobility regime. In the reciprocal-
concave regime it is recovered on the separately specified uniformly elliptic box.

%%%%%%%%%%%%%%%%%%%%%%%%%%%%%%%%%%%%%%%
\section{Concave mobility}\label{sec:ConcaveMobility}
We begin with the first of the two curvature regimes identified in
Section~\ref{sec:NonlinearMobility}. Our aim is to turn the formal convexity
observation into a variational statement and then identify the corresponding
subdifferential.

Let $\mu\in C([0,+\infty))\cap C^1((0,+\infty))$ satisfy $\mu(0) = 0$ and
$\mu(s)>0$ for $s>0$, and assume that $\mu$ is concave on $[0,+\infty)$,
\[ 
   \mu(t s_1+(1-t)s_2) \geq t\mu(s_1)+(1-t)\mu(s_2) \qquad
    \mbox{for all } t\in[0,1] \mbox{ and } s_1,s_2\geq 0.
\]
Since a positive concave function on an unbounded interval cannot eventually
decrease, $\mu$ is nondecreasing and $\mu'$ is non-increasing on
$(0,+\infty)$. Moreover, $\mu$ has at most linear growth: there are constants
$C_1,C_2>0$ such that
\[
   \mu(s)\leq C_1 + C_2 s \qquad\text{for all }s\geq 0.
\]
In particular, the limit $\lim_{s\to+\infty}\mu(s)$ exists in
$[0,+\infty]$, and the associated Nemytskii operator
$\mu:L^2_+(\Omega)\longrightarrow L^2_+(\Omega)$
is continuous on $L^2_+(\Omega)$ with respect to the $L^2(\Omega)$-topology.
To avoid pathological situations we shall for the following %slightly strengthen the assumption of concavity of $\mu$ and
impose that $\mu$ is strictly increasing on $(0,\infty)$.

\subsection{The transformed admissible set and its convexity}
\iffalse
We first consider the genuinely degenerate case $\mu(0)=0$, for which the
inverse-moment constraint from Section~\ref{sec:WeakCutoff} is needed.
The uniformly positive case $\mu(0)=r>0$ will be recovered after the
directional calculation.
Define the transformed admissible set
\begin{equation}\label{def:BmuK}
   \mathcal B_{\mu,K}:=\left\{b\in L^2_+(\Omega):\,
   \mu(b)\in\K_K\right\}
   =\left\{b\in L^2_+(\Omega):\,
   \int_\Omega\mu(b)^{-d/2}\,\d x\leq K\right\}.
\end{equation}
\fi
With the normalization used in \eqref{eq:J-def}, the relaxed pumping-energy
functional satisfies
\[
   \G(\mu(b))=2J_\mu(b)
\]
whenever the expression \eqref{eq:J-def} defining $J_\mu$ is meaningful.
We define the functional $\mathcal H:L^2(\Omega)\to\Rext$ by setting,
for $b\in L^2(\Omega)$,
\[
   \mathcal H(b) =
    \begin{cases}
       \G(\mu(b)),  & \text{if }b\in\mathcal B_{\mu,K},\\
        +\infty, & \text{otherwise},
\end{cases}
\]
with $\mathcal B_{\mu,K}$ defined in \eqref{def:BmuK-density}.
Note that if $\mu(b)\in\K_K$, we have
\begin{align*}
   \mathcal H(b)
   &= \int_\Omega \mu(b) \left|\nabla u[\mu(b)]\right|^2\d x \\
   &= \sup_{u\in\Ha{\mu(b)}}
      \left( 2\int_\Omega f u \,\d x
      - \int_\Omega \mu(b)|\nabla u|^2 \d x \right).
\end{align*}
As before, $u=u[\mu(b)]$ denotes the unique energy solution of
\begin{equation}  \label{eq:Poisson_mub}
   -\operatorname{div}\left(\mu(b)\nabla u\right)=f \qquad\text{in }\Omega,
\end{equation}
subject to the homogeneous Dirichlet boundary condition on $\partial\Omega$.

\begin{proposition}\label{prop:H-convex-lsc}
The set $\mathcal B_{\mu,K}$ is closed and convex in $L^2(\Omega)$. Moreover,
\[
   \operatorname{dom}\mathcal H=\mathcal B_{\mu,K},
\]
and $\mathcal H$ is convex and lower semicontinuous on $L^2(\Omega)$.
\end{proposition}

\begin{proof}
The function $s\mapsto\mu(s)^{-d/2}$, extended by the value $+\infty$ when
$\mu(s)=0$, is convex because $\rho\mapsto\rho^{-d/2}$ is convex and decreasing
while $\mu$ is concave. Its integral functional is lower semicontinuous: from
any $L^2$-convergent sequence, choose a subsequence realizing the lower limit
and then a further subsequence converging almost everywhere, to which Fatou's
lemma applies. This proves that $\mathcal B_{\mu,K}$ is closed and convex.

For $b_1,b_2\in\mathcal B_{\mu,K}$ and $t\in[0,1]$, concavity of $\mu$ gives
\[
   \mu(tb_1+(1-t)b_2)\geq t\mu(b_1)+(1-t)\mu(b_2).
\]
Lemma~\ref{lem:EtildeMonotonicity}, followed by the convexity of $\G$ from
Lemma~\ref{lem:Gconv1}, therefore yields
\[
\begin{aligned}
   \G\bigl(\mu(tb_1+(1-t)b_2)\bigr)
   &\leq\G\bigl(t\mu(b_1)+(1-t)\mu(b_2)\bigr)\\
   &\leq t\G(\mu(b_1))+(1-t)\G(\mu(b_2)).
\end{aligned}
\]
Thus $\mathcal H$ is convex. To prove lower semicontinuity, let
$b_n\to b$ in $L^2(\Omega)$. If
$\liminf_n\mathcal H(b_n)=+\infty$, there is nothing to prove.
Otherwise, after passing to a subsequence realizing the lower limit, we may
assume that $b_n\in\mathcal B_{\mu,K}$ for every $n$. Since
$\mathcal B_{\mu,K}$ is closed, $b\in\mathcal B_{\mu,K}$, while continuity of
the Nemytskii map gives $\mu(b_n)\to\mu(b)$ in $L^2(\Omega)$. The lower
semicontinuity of $\G$ therefore yields
\[
   \mathcal H(b)=\G(\mu(b))
   \leq\liminf_{n\to\infty}\G(\mu(b_n))
   =\liminf_{n\to\infty}\mathcal H(b_n).
\]
\end{proof}

\subsection{Directional derivative and the chain-rule qualification}

Proposition~\ref{prop:H-convex-lsc} supplies the abstract convex structure,
but it does not by itself identify the $L^2$ subgradient. We begin with
directions for which the coefficient-dependent energy spaces remain comparable.
%Throughout this subsection, $D\mathcal H(b_0)[h]$ denotes the feasible one-sided derivative as $t\downarrow0$.

\begin{proposition}[Derivative along comparable directions]
\label{prop:H-directional-comparable}
Let $b_0,b\in\mathcal B_{\mu,K}$ satisfy
\begin{equation}\label{eq:comparable-b-direction}
   0\leq b\leq Lb_0\qquad\text{a.e. in }\Omega,
\end{equation}
for some $L>0$, and let $h=\tau(b-b_0)$ with $\tau>0$. Then the feasible
right directional derivative exists and is given by
\begin{equation}\label{eq:H-directional-comparable}
   D\mathcal H(b_0)[h] := \lim_{t\to 0^+}
   \frac{\mathcal H(b_0+th)-\mathcal H(b_0)}{t}
   =-\int_\Omega
      \mu'(b_0)|\nabla u[\mu(b_0)]|^2h\,\d x.
\end{equation}
\end{proposition}

\begin{proof}
For $0\leq\tau t\leq1$,
\[
   b_0+th=(1-\tau t)b_0+\tau tb\in\mathcal B_{\mu,K}.
\]
Moreover, monotonicity and concavity of $\mu$, together with $\mu(0)=0$ and
\eqref{eq:comparable-b-direction}, give constants $c,C>0$, independent of
sufficiently small $t>0$, such that
\[
   c\,\mu(b_0)\leq\mu(b_0+th)\leq C\,\mu(b_0)
   \qquad\text{a.e. in }\Omega.
\]
Thus the corresponding weighted energy spaces are comparable, and all
cross-tests used below are admissible. For an arbitrary sequence
$t_n\to 0^+$, consider the directional difference quotient
\[
   \frac{\mathcal H(b_0+t_n h)-\mathcal H(b_0)}{t_n}
   =
   \int_\Omega
      \frac{\mu_n |\nabla u_n|^2-\mu_0|\nabla u_0|^2}{t_n}
   \,\d x .
\]
We denote
\[
   u_0:=u[\mu(b_0)],\qquad
   u_n:=u[\mu(b_0+t_nh)],\qquad
   \mu_0:=\mu(b_0),\qquad
   \mu_n:=\mu(b_0+t_nh),
\]
for $t_n\to 0^+$.
We subtract the elliptic equations
\begin{equation}   \label{eq:sec:concave:ell0}
   -\operatorname{div}(\mu_0\nabla u_0)=f,\qquad u_0|_{\partial\Omega}=0,
\end{equation}
and
\[
   -\operatorname{div}(\mu_n\nabla u_n)=f,\qquad u_n|_{\partial\Omega}=0,
\]
test the difference against $u_0+u_n$, and obtain, after a short computation,
\begin{equation}
   \frac{\mathcal H(b_0+t_n h)-\mathcal H(b_0)}{t_n}
   =
   -\int_\Omega
      \frac{\mu_n-\mu_0}{t_n}
      \nabla u_n\cdot\nabla u_0
   \,\d x .
   \label{eq:direct-quotient}
\end{equation}
Writing $h=\tau(b-b_0)$ with $b\in \mathcal B_{\mu,K}$, we have
\[
   \mu_n
   =
   \mu(b_0+t_nh)
   =
   \mu\bigl( b_0+\tau t_n(b-b_0)\bigr)
   =
   \mu\bigl( (1-\tau t_n)b_0+\tau t_n b\bigr).
\]
By concavity of $\mu$,
\[
   \mu_n
   \geq
   (1-\tau t_n)\mu(b_0)+\tau t_n\mu(b)
   \geq
   \frac12 \mu(b_0)
\]
provided $t_n\leq (2\tau)^{-1}$. From now on we assume this without further
comment.

%\paragraph{Stability of the elliptic solutions.}
Lemma~\ref{lem:Riesz}, applied with $a=\mu_n$, gives
\[
   \sup_n\int_\Omega\mu_n|\nabla u_n|^2\,\d x<+\infty,
\]
and consequently $(u_n)_{n\in\N}$ is bounded in $\Ha{\mu_0}$.
We now consider the weak formulation of
\[
   -\operatorname{div}(\mu_n\nabla u_n)=f,\qquad u_n|_{\partial\Omega}=0.
\]
For every test function $\varphi\in C_0^\infty(\Omega)$,
\[
   \int_\Omega \mu_n \nabla u_n\cdot\nabla\varphi\,\d x
   =
   \int_\Omega f\varphi\,\d x.
\]
Equivalently,
\[
   \int_\Omega \mu_0 \nabla u_n\cdot\nabla\varphi\,\d x
   =
   \int_\Omega f\varphi\,\d x
   -
   \int_\Omega (\mu_n-\mu_0)\nabla u_n\cdot\nabla\varphi\,\d x .
\]
We prove that the last integral tends to zero. Indeed, using H\"older's inequality,
\[
\begin{aligned}
   \left|
   \int_\Omega (\mu_n-\mu_0)\nabla u_n\cdot\nabla\varphi\,\d x
   \right|
   &=
   \left|
   \int_\Omega
      \left(\sqrt{\mu_n}-\frac{\mu_0}{\sqrt{\mu_n}}\right)
      \sqrt{\mu_n}\,\nabla u_n\cdot\nabla\varphi
   \,\d x
   \right|                                                        \\
   &\leq
   \left(\int_\Omega \mu_n|\nabla u_n|^2\,\d x\right)^{1/2}
   \left(
   \int_\Omega
      \left(\sqrt{\mu_n}-\frac{\mu_0}{\sqrt{\mu_n}}\right)^2
      |\nabla\varphi|^2
   \,\d x
   \right)^{1/2} .
\end{aligned}
\]
Since $\mu_n\to\mu_0$ pointwise a.e. in $\Omega$ and, for all sufficiently
large $n$,
\[
   0\leq \frac{\mu_0^2}{\mu_n}\leq 2\mu_0,
   \qquad
   0\leq\mu_n\leq\max\{\mu(b_0),\mu(b)\},
\]
where the right-hand sides are integrable, dominated convergence gives
\[
   \int_\Omega \left(\sqrt{\mu_n}-\frac{\mu_0}{\sqrt{\mu_n}}\right)^2 \,\d x \to 0.
\]
Passing to the limit in the weak formulation shows that every weak cluster
point $\bar u\in\Ha{\mu_0}$ satisfies
\[
   \int_\Omega\mu_0\nabla\bar u\cdot\nabla\varphi\,\d x
   =\int_\Omega f\varphi\,\d x
   \qquad \mbox{for all } \varphi\in C_0^\infty(\Omega).
\]
By density of $C_0^\infty(\Omega)$ in $\Ha{\mu_0}$, this identity extends to
every $\varphi\in\Ha{\mu_0}$. Uniqueness in Lemma~\ref{lem:Riesz} then gives
$\bar u=u_0$. Consequently, the whole sequence satisfies
\[
   u_n \rightharpoonup u_0
   \qquad \text{weakly in } \Ha{\mu_0}.
\]

%\paragraph{Passage to the limit in the difference quotient.}
We now write \eqref{eq:direct-quotient} as
\begin{equation}
\begin{aligned}
   \frac{\mathcal H(b_0+t_n h)-\mathcal H(b_0)}{t_n}
   &=
   -\int_\Omega
      \left(
      \frac{\mu_n-\mu_0}{t_n}-\mu'_0h
      \right)
      \nabla u_n\cdot\nabla u_0
   \,\d x                                                     \\
   &\quad
   -\int_\Omega
      \mu'_0h\,\nabla u_n\cdot\nabla u_0
   \,\d x ,
\end{aligned}
\label{eq:quotient-split}
\end{equation}
where $\mu'_0:=\mu'(b_0)$.
Note that $b_0\in\mathcal B_{\mu,K}$ implies $\mu_0\in\K_K$, so the set where
$\mu_0$ vanishes has Lebesgue measure zero. Since $\mu(0)=0$ and
$\mu(s)>0$ for $s>0$, we also have $b_0>0$ almost everywhere; hence
$\mu'(b_0)$ is finite almost everywhere.
Using the weak convergence just established, the second term in
\eqref{eq:quotient-split} satisfies
\[
   -\int_\Omega \mu'_0h\,\nabla u_n\cdot\nabla u_0\,\d x
   \,\to\,
   -\int_\Omega \mu'_0|\nabla u_0|^2 h\,\d x,
\]
provided
\[
   \frac{\mu'_0}{\mu_0}h\in L^\infty(\Omega),
\]
which guarantees
\[
   \frac{\mu'_0}{\mu_0}h\,\nabla u_0
   \in L^2(\Omega,\mu_0\d x;\mathbb R^d).
\]
Since $h=\tau(b-b_0)$, we estimate
\[
   \frac{|\mu'_0|}{\mu_0}|h|
   \leq
   \tau \frac{|\mu'_0|}{\mu_0}b
   +
   \tau \frac{|\mu'_0|}{\mu_0}b_0 .
\]
The supporting-hyperplane inequality for the concave function $\mu$ gives,
for all $b_1\geq0$,
\[
   \mu(b_1)
   \leq
   \mu(b_0)+\mu'(b_0)(b_1-b_0).
\]
Setting $b_1=0$, we find
\[
   \frac{|\mu'_0|b_0}{\mu_0}\leq 1 .
\]
Therefore, condition \eqref{eq:comparable-b-direction} guarantees the required
essential bound.

It remains to prove that the first integral on the right-hand side of
\eqref{eq:quotient-split} tends to zero. We denote
\[
   \xi_n
   :=
   \frac{\mu_n-\mu_0}{t_n}-\mu'_0h.
\]
By the mean value theorem, pointwise almost everywhere there is a value
$\eta_n(x)$ between $b_0(x)$ and $b_0(x)+t_nh(x)$ such that
\[
   \xi_n=\bigl(\mu'(\eta_n)-\mu'(b_0)\bigr)h.
\]
Hence $\xi_n\to0$ pointwise almost everywhere in $\Omega$.
Moreover,
\[
\begin{aligned}
   \left|
   \int_\Omega \xi_n\nabla u_n\cdot\nabla u_0\,\d x
   \right|
   &\leq
   \left(\int_\Omega \mu_0|\nabla u_n|^2\,\d x\right)^{1/2}
   \left(
   \int_\Omega
      \left|\frac{\xi_n}{\mu_0}\right|^2
      \mu_0|\nabla u_0|^2
   \,\d x
   \right)^{1/2}.
\end{aligned}
\]
We shall now derive conditions under which
\[
   \frac{\xi_n}{\mu_0}\in L^\infty(\Omega)
   \quad\text{uniformly in } n\in\N.
\]

First, let $x\in\Omega$ be such that $b_0(x)\leq b(x)$.
Then $h(x)\geq0$ and $\eta_n(x)\geq b_0(x)$. Using the monotonicity of
$\mu'$ on $(0,+\infty)$, we estimate
\[
   \frac{|\xi_n(x)|}{\mu_0(x)}
   \leq
   C\tau\,\frac{b(x)}{b_0(x)}\,\frac{|\mu'_0(x)|b_0(x)}{\mu_0(x)} .
\]

Second, let $x\in\Omega$ be such that $b_0(x)>b(x)$.
Then $h(x)<0$ and $\eta_n(x)\leq b_0(x)$. Moreover,
\[
   \eta_n(x)
   \geq
   b_0(x)+t_nh(x)
   =
   (1-\tau t_n)b_0(x)+\tau t_n b(x)
   \geq
   (1-\varepsilon)b_0(x)
\]
whenever $0<\varepsilon<1$ is fixed and $t_n\leq \varepsilon/\tau$.
Using again the supporting-hyperplane inequality for the concave function
$\mu$, one obtains a
pointwise bound of the form
\[
   \frac{|\xi_n(x)|}{\mu_0(x)}
   \leq
   \frac{C}{1-\varepsilon}.
\]
Consequently, by dominated convergence,
\[
   \int_\Omega
      \left|\frac{\xi_n}{\mu_0}\right|^2
      \mu_0|\nabla u_0|^2 \,\d x \,\to\, 0.
\]
We conclude that
\[
   \int_\Omega \xi_n\nabla u_n\cdot\nabla u_0\,\d x \,\to\, 0.
\]
Combining this with \eqref{eq:quotient-split}, we obtain
\[
   D\mathcal H(b_0)[h]
   =
   -\int_\Omega \mu'(b_0)|\nabla u[\mu(b_0)]|^2 h\,\d x
\]
for the comparable direction $h$. Since the sequence $t_n\to 0^+$ was
arbitrary, this proves \eqref{eq:H-directional-comparable}.
\end{proof}

Lemma \ref{lem:BmuK-dominated-density} shows that the linear functional on the right hand side of
\eqref{eq:H-directional-comparable} extends continuously from
\[
   \bigcup_{\tau>0}
   \left\{\tau(b-b_0);\, b\in A_\mu(b_0)\right\}
\]
to $T_{\mathcal B_{\mu,K}}(b_0)$ if and only if
$q_0:=\mu'(b_0)|\nabla u[\mu(b_0)]|^2\in L^2(\Omega)$.

\subsection{Subdifferentials and gradient-flow identification}

We now characterize the subdifferential of $\mathcal H$,
%without additional growth assumptions on the concave mobility $\mu$.
%The main point is to establish a global supporting inequality, which does not follow directly
%from the directional derivative formula \eqref{eq:H-directional-comparable}.
applying the radial approximation result of Lemma~\ref{lem:radial-energy-approximation}
to obtain the necessary approximation of the weighted energy solution.

\begin{proposition}[Subdifferential of $\mathcal H$]
\label{prop:H-subdifferential}
Let $d\geq2$, and let
$\mu\in C([0,\infty))\cap C^1((0,\infty))$
be strictly increasing and concave, with $\mu(0)=0$.
For $b_0\in\mathcal B_{\mu,K}$, set
\[
   u_0:=u[\mu(b_0)],
   \qquad
   q_0:=\mu'(b_0)|\nabla u_0|^2.
\]
Then
\begin{equation}\label{eq:H-subdifferential}
   \partial\mathcal H(b_0)=
   \begin{cases}
   -q_0+N_{\mathcal B_{\mu,K}}(b_0),
      &q_0\in L^2(\Omega),\\
   \emptyset,&q_0\notin L^2(\Omega).
   \end{cases}
\end{equation}
%In particular, no additional integrability assumption on $\nabla u_0$ or on the normal to the inverse-moment constraint is required.
\end{proposition}

\begin{proof}
We first establish the necessity of $q_0\in L^2(\Omega)$,
then prove the two subdifferential inclusions.

\smallskip
\noindent
\emph{Step 1. Necessity of $q_0\in L^2(\Omega)$.}
Suppose that $\xi\in\partial\mathcal H(b_0)$.
Since $b_0\in\mathcal B_{\mu,K}$,
we have $b_0>0$ almost everywhere.
Concavity of $\mu$ and $\mu(0)=0$ imply
\[
   0\leq b_0\mu'(b_0)\leq\mu(b_0).
\]
Consequently,
\[
   0\leq b_0q_0
   \leq\mu(b_0)|\nabla u_0|^2\in L^1(\Omega).
\]

For an arbitrary measurable set $D\subset\Omega$,
consider the upward comparable direction
$h=b_0\mathbf 1_D$.
Since $\mathcal B_{\mu,K}$ is upward closed,
$b_0+th\in\mathcal B_{\mu,K}$ for $t\geq0$,
and the corresponding mobilities are comparable.
The subgradient inequality and
Proposition~\ref{prop:H-directional-comparable}
therefore give
\[
   \int_D\xi b_0\,\d x
   \leq
   -\int_Dq_0b_0\,\d x.
\]
Since $D$ is arbitrary and $b_0>0$ almost everywhere,
we conclude that
\[
   \xi\leq-q_0\leq0
   \qquad\text{a.e. in }\Omega.
\]
Thus $0\leq q_0\leq|\xi|$, and hence
$q_0\in L^2(\Omega)$.

\smallskip
\noindent
\emph{Step 2. The inclusion
$\partial\mathcal H(b_0)
\subseteq-q_0+N_{\mathcal B_{\mu,K}}(b_0)$.}
Let $\xi\in\partial\mathcal H(b_0)$.
By Step 1, $q_0\in L^2(\Omega)$.
For every comparable feasible direction
$h=\tau(b-b_0)$, with
$b\in A_\mu(b_0)$ and $\tau>0$,
the subgradient inequality and
\eqref{eq:H-directional-comparable} yield
\[
   \int_\Omega(\xi+q_0)h\,\d x\leq0.
\]
By Lemma~\ref{lem:BmuK-dominated-density},
the comparable feasible directions are dense
in the tangent cone
$T_{\mathcal B_{\mu,K}}(b_0)$.
Since $\xi+q_0\in L^2(\Omega)$,
the inequality extends to every
$h\in T_{\mathcal B_{\mu,K}}(b_0)$.
Therefore
\[
   \xi+q_0\in N_{\mathcal B_{\mu,K}}(b_0),
\]
which proves the first inclusion.

\smallskip
\noindent
\emph{Step 3. The global supporting inequality.}
Assume now that $q_0\in L^2(\Omega)$.
We prove that
\begin{equation}\label{eq:H-global-support}
   \G(\mu(b))
   \geq\G(\mu(b_0))
   -\int_\Omega q_0(b-b_0)\,\d x
   \qquad\text{for every }b\in\mathcal B_{\mu,K}.
\end{equation}

Write $a_0:=\mu(b_0)$ and
$s_0:=|\nabla u_0|^2$.
By concavity of $\mu$,
\[
   \mu(b)\leq a_0+\mu'(b_0)(b-b_0).
\]
Multiplying by $s_0$ gives
\begin{equation}\label{eq:H-competitor-integrability}
   0\leq\mu(b)s_0
   \leq a_0s_0+q_0(b-b_0)
   \leq a_0s_0+q_0b.
\end{equation}
The right-hand side is integrable, since
$a_0s_0\in L^1(\Omega)$ and
$q_0,b\in L^2(\Omega)$.

In order to use $u_0$ as a test function in
the common-core representation of
$\G(\mu(b))$, we need a smooth approximation
that also converges in the weighted energy
corresponding to $\mu(b)$.

To construct it, introduce the auxiliary integrand
\begin{equation}\label{eq:H-auxiliary-integrand}
   \Phi(\zeta):=
   \sup_{t\geq0}\left\{
      \mu(t)|\zeta|^2-t^2
   \right\},
   \qquad \zeta\in\R^d.
\end{equation}
This is a convex, lower-semicontinuous, radial
function, nondecreasing in $|\zeta|$, with
$\Phi(0)=0$.

Since $\mu$ is increasing, concave and
vanishes at zero, it has at most linear growth:
\[
   0\leq\mu(t)\leq C_0+C_1t,
   \qquad t\geq0,
\]
for suitable $C_0,C_1\geq0$.
Consequently,
\[
   0\leq\Phi(\zeta)
   \leq C_0|\zeta|^2+\frac{C_1^2}{4}|\zeta|^4.
\]
In particular, $\Phi$ is finite everywhere.

By concavity of $\mu$, for every $t\geq0$,
\[
   \mu(t)s_0-t^2
   \leq a_0s_0+q_0(t-b_0)-t^2.
\]
Taking the supremum over $t\geq0$, we obtain
\begin{align*}
   0\leq\Phi(\nabla u_0)
   &\leq a_0s_0-q_0b_0
       +\sup_{t\geq0}(q_0t-t^2)\\
   &\leq a_0s_0+\frac14q_0^2.
\end{align*}
Thus
\begin{equation}\label{eq:H-auxiliary-integrability}
   \Phi(\nabla u_0)\in L^1(\Omega).
\end{equation}

Moreover, testing the supremum
\eqref{eq:H-auxiliary-integrand} at $t=1$ gives
\[
   \mu(1)|\nabla u_0|^2
   \leq\Phi(\nabla u_0)+1.
\]
Hence $\nabla u_0\in L^2(\Omega)$.
Since the weighted energy space embeds into
$W_0^{1,2d/(d+2)}(\Omega)\cap L^2(\Omega)$,
the zero-trace characterization on a Lipschitz
domain implies $u_0\in H_0^1(\Omega)$,
including the endpoint case $d=2$.

We may now apply
Lemma~\ref{lem:radial-energy-approximation}.
There exists a sequence
$\varphi_n\in C_0^\infty(\Omega)$ such that
\[
   \varphi_n\to u_0\quad\text{in }H_0^1(\Omega),
   \qquad
   \Phi(\nabla\varphi_n)\to\Phi(\nabla u_0)
   \quad\text{in }L^1(\Omega),
\]
and, after passing to a subsequence,
$\nabla\varphi_n\to\nabla u_0$ almost everywhere.

For any fixed $b\in\mathcal B_{\mu,K}$,
the definition of $\Phi$ implies
\[
   0\leq\mu(b)|\nabla\varphi_n|^2
   \leq b^2+\Phi(\nabla\varphi_n).
\]
Since $b^2\in L^1(\Omega)$ and
$\Phi(\nabla\varphi_n)$ converges in $L^1$,
the right-hand sides are uniformly integrable.
The same bound holds for $u_0$.
It follows that
$\mu(b)|\nabla(\varphi_n-u_0)|^2$
is uniformly integrable and converges
to zero almost everywhere.
Vitali's theorem therefore yields
\begin{equation}\label{eq:H-weighted-approximation}
   \int_\Omega
      \mu(b)|\nabla(\varphi_n-u_0)|^2\,\d x
   \longrightarrow0.
\end{equation}

Thus $\varphi_n$ approximates $u_0$
in the weighted energy of every fixed
competitor $b$. In particular,
we may pass to the limit in the
common-core representation to obtain
\begin{align*}
   \G(\mu(b))
   &\geq
   \lim_{n\to\infty}
   \left(
      2\int_\Omega f\varphi_n\,\d x
      -\int_\Omega\mu(b)|\nabla\varphi_n|^2\,\d x
   \right)\\
   &=2\int_\Omega fu_0\,\d x
     -\int_\Omega\mu(b)s_0\,\d x\\
   &=\G(a_0)
     -\int_\Omega(\mu(b)-a_0)s_0\,\d x\\
   &\geq\G(a_0)
     -\int_\Omega q_0(b-b_0)\,\d x.
\end{align*}
Here we used the weighted energy identity
for $u_0$ and the concavity of $\mu$.
This proves \eqref{eq:H-global-support}.

\smallskip
\noindent
\emph{Step 4. The reverse inclusion.}
The global supporting inequality
\eqref{eq:H-global-support} implies
$-q_0\in\partial\mathcal H(b_0)$.
Moreover, for every
$\eta\in N_{\mathcal B_{\mu,K}}(b_0)$
and $b\in\mathcal B_{\mu,K}$,
\[
   \int_\Omega\eta(b-b_0)\,\d x\leq0.
\]
Adding this inequality to
\eqref{eq:H-global-support} yields
\[
   \mathcal H(b)\geq\mathcal H(b_0)
   +\int_\Omega(-q_0+\eta)(b-b_0)\,\d x.
\]
Hence
\[
   -q_0+N_{\mathcal B_{\mu,K}}(b_0)
   \subseteq\partial\mathcal H(b_0).
\]
Together with Step 2, this proves
\eqref{eq:H-subdifferential}.
\end{proof}

We next identify the normal cone appearing in
\eqref{eq:H-subdifferential}. This requires an
additional integrability assumption only when
the inverse-moment constraint is active.
For $b_0\in\mathcal B_{\mu,K}$, set
\[
   z_0:=
   \frac{\mu'(b_0)}{\mu(b_0)^{d/2+1}}.
\]
If the constraint is inactive, that is,
\[
   \int_\Omega\mu(b_0)^{-d/2}\,\d x<K,
\]
then
\[
   N_{\mathcal B_{\mu,K}}(b_0)=\{0\}.
\]
If the constraint is active and
$z_0\in L^2(\Omega)$, then
\[
   N_{\mathcal B_{\mu,K}}(b_0)
   =\{-\Lambda z_0:\Lambda\geq0\}.
\]
The factor $d/2$ arising from differentiation
of the inverse-moment constraint is absorbed
into $\Lambda$.

Consequently, whenever $q_0\in L^2(\Omega)$,
the subdifferential can be written as
\begin{equation}\label{eq:H-subdifferential-multiplier}
   \partial\mathcal H(b_0)
   =
   \left\{
      -q_0-\Lambda z_0:
      \Lambda\geq0,\quad
      \Lambda\left(
         \int_\Omega\mu(b_0)^{-d/2}\,\d x-K
      \right)=0
   \right\},
\end{equation}
provided $z_0\in L^2(\Omega)$ in the active case.
In the inactive case,
$\partial\mathcal H(b_0)=\{-q_0\}$
without any assumption on $z_0$.

\medskip
We now turn to the gradient flow.
The characterization
\eqref{eq:H-subdifferential} provides
an unconditional identification of the
pure pumping flow.

\begin{corollary}%[Pure pumping flow]
\label{cor:H-pure-flow}
Assume that $\mathcal B_{\mu,K}\neq\emptyset$
and that $\mu$ satisfies the hypotheses of
Proposition~\ref{prop:H-subdifferential}.
Then every $b_{\mathrm{in}}\in\mathcal B_{\mu,K}$
generates a unique $L^2(\Omega)$-subgradient
flow of $\mathcal H$. The flow is strong on
every finite time interval and satisfies
\begin{equation}\label{eq:H-pure-flow}
   \partial_t b
   =\mu'(b)|\nabla u[\mu(b)]|^2
   \qquad\text{a.e. in }\Omega\times(0,\infty).
\end{equation}
Moreover, for every $t\geq0$,
\begin{equation}\label{eq:H-pure-flow-energy}
   \mathcal H(b(t))
   +\int_0^t\int_\Omega
      \mu'(b(s))^2|\nabla u[\mu(b(s))]|^4
      \,\d x\,\d s
   =\mathcal H(b_{\mathrm{in}}).
\end{equation}
In particular, $b(t)$ is nondecreasing
in time almost everywhere in $\Omega$, and
\[
   t\longmapsto
   \int_\Omega\mu(b(t))^{-d/2}\,\d x
\]
is nonincreasing.
\end{corollary}

\begin{proof}
By Proposition~\ref{prop:H-convex-lsc}
and the weak-cutoff energy estimate,
$\mathcal H$ is proper, convex and lower
semicontinuous on $L^2(\Omega)$,
with effective domain $\mathcal B_{\mu,K}$.
The standard Hilbert-space subgradient-flow
theory therefore gives a unique strong
solution for every initial datum in this domain,
together with the energy identity.

At almost every positive time,
$\partial\mathcal H(b(t))\neq\emptyset$.
Proposition~\ref{prop:H-subdifferential}
then gives
\[
   q(t):=\mu'(b(t))|\nabla u[\mu(b(t))]|^2
   \in L^2(\Omega)
\]
and
\[
   \partial\mathcal H(b(t))
   =-q(t)+N_{\mathcal B_{\mu,K}}(b(t)).
\]
Since $\mathcal B_{\mu,K}$ is upward closed,
every element of its normal cone is
nonpositive almost everywhere.
As $q(t)\geq0$, the minimal-norm element
of this subdifferential is therefore $-q(t)$.
Hence the gradient flow satisfies
$\partial_t b=q(t)$, proving
\eqref{eq:H-pure-flow}.

The energy identity for the subgradient flow
gives \eqref{eq:H-pure-flow-energy}.
Finally, $\partial_t b\geq0$ almost everywhere,
so $b$ is nondecreasing in time.
Since $\mu$ is increasing, the inverse-moment
integral is consequently nonincreasing.
\end{proof}

We now consider the case with nonzero metabolic cost.
Let $g$ satisfy the assumptions used in the definition of $\Gamma$ in
\eqref{def:Gamma}. Assume additionally that $g\in C^1([0,+\infty))$, and set
\[
   \mathcal E_\mu^+(b):=\mathcal H(b)+\Gamma(b).
\]
Under a valid subdifferential sum rule and provided
\[
   \mu'(b)|\nabla u[\mu(b)]|^2\in L^2(\Omega),
   \qquad g'(b)\in L^2(\Omega),
\]
the negative-subgradient evolution can be identified pointwise. When
the inverse-moment constraint is inactive, its normal multiplier
vanishes and the equation reduces to
\begin{equation*}
   \partial_t b = \mu'(b)\left|\nabla u[\mu(b)]\right|^2-g'(b).
\end{equation*}
At an active constraint the normal-cone multiplier must be retained.
%no unqualified pointwise evolution formula is asserted.

\subsection{Stationary states, duality, and power laws}

We next consider stationary states of the energy $\mathcal E_\mu^+=\mathcal H+\Gamma$.
The pointwise optimality condition for the unconstrained conductivity variable
suggests the generalized Legendre transform
\begin{equation}\label{eq:concave-generalized-conjugate}
   g_\mu^*(s):=\sup_{v\geq0}\bigl\{\mu(v)s-g(v)\bigr\},
   \qquad s\geq0.
\end{equation}
Here and throughout this subsection, $g:[0,\infty)\to[0,\infty)$ is
finite, convex and continuous at zero. We define the pressure Hamiltonian
\begin{equation}\label{eq:concave-pressure-Hamiltonian}
   H_\mu(\varphi):=\frac12\int_\Omega
      g_\mu^*(|\nabla\varphi|^2)\,\d x
      -\int_\Omega f\varphi\,\d x,
   \qquad \varphi\in H_0^1(\Omega),
\end{equation}
with the value $+\infty$ if the first integral is infinite.
The function $\xi\mapsto g_\mu^*(|\xi|^2)$ is convex,
lower semicontinuous, radial and nondecreasing in $|\xi|$;
moreover, $g_\mu^*(0)=-\inf_{v\geq0}g(v)$ is finite.

The common-core representation of $\G$ yields the weak duality inequality
\begin{equation}\label{eq:concave-stationary-weak-duality}
   \inf_{b\in\mathcal B_{\mu,K}}\mathcal E_\mu^+(b)
   \geq -2\inf_{\varphi\in H_0^1(\Omega)}H_\mu(\varphi).
\end{equation}
Indeed, for $b\in\mathcal B_{\mu,K}$ with $\Gamma(b)<+\infty$
and $\varphi\in C_0^\infty(\Omega)$,
\begin{align*}
   \mathcal E_\mu^+(b)
   &\geq 2\int_\Omega f\varphi\,\d x
      -\int_\Omega\mu(b)|\nabla\varphi|^2\,\d x
      +\int_\Omega g(b)\,\d x\\
   &\geq 2\int_\Omega f\varphi\,\d x
      -\int_\Omega g_\mu^*(|\nabla\varphi|^2)\,\d x
    =-2H_\mu(\varphi).
\end{align*}
For any $\varphi\in H_0^1(\Omega)$ of finite Hamiltonian energy,
Lemma~\ref{lem:radial-energy-approximation}, applied to
$\xi\mapsto g_\mu^*(|\xi|^2)$, gives smooth test functions
$\varphi_n\to\varphi$ in $H_0^1(\Omega)$ with
$H_\mu(\varphi_n)\to H_\mu(\varphi)$.
Thus the inequality extends to every $\varphi\in H_0^1(\Omega)$,
which proves \eqref{eq:concave-stationary-weak-duality}.

The Hamiltonian $H_\mu$ is proper, convex, lower semicontinuous and
coercive on $H_0^1(\Omega)$. In fact, choosing $v=1$ in
\eqref{eq:concave-generalized-conjugate} gives
\[
   g_\mu^*(|\nabla\varphi|^2)
   \geq \mu(1)|\nabla\varphi|^2-g(1),
\]
and the assertion follows from $\mu(1)>0$, the Poincar\'e inequality
and $f\in L^2(\Omega)$. Thus the pressure minimization problem
on the right-hand side of \eqref{eq:concave-stationary-weak-duality}
has a solution. Equality in \eqref{eq:concave-stationary-weak-duality},
however, cannot be concluded without additional information.
The next proposition provides a direct certificate for equality and global minimality.
%without assuming a general minimax theorem.

\begin{proposition}[Stationary minimizers for concave mobility]
\label{prop:concave-stationary-certificate}
%Let $\mu$ satisfy the standing assumptions of this section, and let
Let $g:[0,\infty)\to[0,\infty)$ be finite, convex and continuous at zero.
Suppose $b_0\in\mathcal B_{\mu,K}$, let $u_0:=u[\mu(b_0)]$
be its smooth-core energy solution, and assume that
\begin{equation}\label{eq:concave-stationary-calibration}
   g_\mu^*(|\nabla u_0|^2)
   =\mu(b_0)|\nabla u_0|^2-g(b_0)
   \qquad\text{a.e. in }\Omega.
\end{equation}
Then $\Gamma(b_0)<+\infty$ and $b_0$ is a global minimizer
of $\mathcal E_\mu^+$ on $L^2(\Omega)$.
If $g$ is strictly convex, this minimizer is unique.
Moreover, $u_0$ minimizes $H_\mu$ and equality holds in
\eqref{eq:concave-stationary-weak-duality}.
\end{proposition}

\begin{proof}
Put $a_0:=\mu(b_0)$ and $s_0:=|\nabla u_0|^2$.
Since $b_0\in\mathcal B_{\mu,K}$, the weighted energy identity gives
\begin{equation}\label{eq:concave-stationary-core-identity}
   \int_\Omega a_0s_0\,\d x
   =\int_\Omega fu_0\,\d x
   =\G(a_0)<+\infty.
\end{equation}
From \eqref{eq:concave-generalized-conjugate} with $v=0$
and \eqref{eq:concave-stationary-calibration}, we obtain
\[
   0\leq g(b_0)\leq a_0s_0+g(0).
\]
Hence $\Gamma(b_0)<+\infty$ and
$g_\mu^*(s_0)=a_0s_0-g(b_0)$ belongs to $L^1(\Omega)$.
Furthermore, evaluating the supremum in
\eqref{eq:concave-generalized-conjugate} at $v=1$ gives
\[
   \mu(1)|\nabla u_0|^2
   \leq g_\mu^*(s_0)+g(1),
\]
so $\nabla u_0\in L^2(\Omega)$.
The weak-cutoff embedding gives
$u_0\in W_0^{1,2d/(d+2)}(\Omega)\cap L^2(\Omega)$;
its zero-trace property, together with $\nabla u_0\in L^2$,
then implies $u_0\in H_0^1(\Omega)$, also when $d=2$.
Thus $H_\mu(u_0)<+\infty$ and $u_0$ is an admissible
test in the weak duality inequality.

For every $b\in\mathcal B_{\mu,K}$,
the weak duality inequality established above gives
\begin{align*}
   \mathcal E_\mu^+(b)
   &\geq -2H_\mu(u_0)\\
   &=2\int_\Omega fu_0\,\d x
     -\int_\Omega g_\mu^*(s_0)\,\d x\\
   &=2\int_\Omega fu_0\,\d x
     -\int_\Omega a_0s_0\,\d x
     +\int_\Omega g(b_0)\,\d x\\
   &=\G(a_0)+\Gamma(b_0)
    =\mathcal E_\mu^+(b_0),
\end{align*}
where the last equality uses
\eqref{eq:concave-stationary-core-identity}.
This proves global minimality. If $g$ is strictly convex,
then $\Gamma$ is strictly convex on its effective domain,
while $\mathcal H$ is convex. Therefore the minimizer is unique.
Finally, since
$\mathcal E_\mu^+(b_0)=-2H_\mu(u_0)$,
the weak duality inequality implies that $u_0$ minimizes
$H_\mu$ and that equality holds in
\eqref{eq:concave-stationary-weak-duality}.
\end{proof}

Condition \eqref{eq:concave-stationary-calibration} means that
$b_0(x)$ attains the pointwise supremum in
\eqref{eq:concave-generalized-conjugate} at
$s=|\nabla u_0(x)|^2$. In particular,
\begin{equation}\label{eq:concave-stationary-conjugacy}
   a_0=\mu(b_0)\in\partial g_\mu^*(|\nabla u_0|^2)
   \qquad\text{a.e. in }\Omega.
\end{equation}
Since $b_0>0$ almost everywhere, this pointwise optimality
condition is equivalently expressed as
\begin{equation}\label{eq:concave-stationary-pointwise}
   \mu'(b_0)|\nabla u_0|^2\in\partial g(b_0)
   \qquad\text{a.e. in }\Omega.
\end{equation}
When $g$ is differentiable, this reduces to
\[
   g'(b_0)=\mu'(b_0)|\nabla u_0|^2.
\]
\iffalse
Thus the proposition certifies a supplied stationary pair;
it does not establish the existence of a conductivity recovered
from an arbitrary minimizing pressure.
The inverse-moment constraint need not be strictly inactive:
the conclusion also applies when its defining integral equals $K$.
On the other hand, stationary states involving a nonzero
constraint multiplier are not covered by the pointwise
calibration condition.
\fi

We now specialize to power mobilities and obtain an explicit
$p$-Laplacian reduction. Let
\begin{equation}\label{eq:concave-power-mobility}
   \mu(s)=\frac{s^\sigma}{\sigma},\qquad
   g(s)=\frac{s^\lambda}{\lambda},
   \qquad 0<\sigma\leq1,\quad\lambda>1.
\end{equation}
A direct maximization gives
\begin{equation}\label{eq:concave-power-conjugate}
   g_\mu^*(s)=\frac{\lambda-\sigma}{\sigma\lambda}
   s^{\lambda/(\lambda-\sigma)},\qquad
   p:=\frac{2\lambda}{\lambda-\sigma}>2.
\end{equation}
The pressure Hamiltonian is therefore
\[
   H_\mu(\varphi)
   =\frac{\lambda-\sigma}{2\sigma\lambda}
      \int_\Omega|\nabla\varphi|^p\,\d x
      -\int_\Omega f\varphi\,\d x,
   \qquad \varphi\in W_0^{1,p}(\Omega).
\]
Its unique minimizer $u_0\in W_0^{1,p}(\Omega)$
solves
\begin{equation}\label{eq:p-laplace}
   \begin{cases}
   -\displaystyle\frac{1}{\sigma}\operatorname{div}
       \bigl(|\nabla u_0|^{p-2}\nabla u_0\bigr)=f
       &\text{in }\Omega,\\
   u_0=0&\text{on }\partial\Omega.
   \end{cases}
\end{equation}
The pointwise maximizing coefficient is
\begin{equation}\label{eq:concave-power-candidate}
   b_0=|\nabla u_0|^{2/(\lambda-\sigma)},\qquad
   a_0:=\mu(b_0)=\frac1\sigma|\nabla u_0|^{p-2}.
\end{equation}

Suppose that
\begin{equation}\label{eq:concave-power-admissibility}
   b_0\in L^2(\Omega),\qquad
   \sigma^{d/2}\int_\Omega
       |\nabla u_0|^{(2-p)d/2}\,\d x\leq K.
\end{equation}
The first condition ensures $a_0\in L^2(\Omega)$, while
the second gives $a_0>0$ almost everywhere and the
inverse-moment bound. Together they imply
$b_0\in\mathcal B_{\mu,K}$.
The $L^2$ condition is automatic for $\lambda\geq2$, because
\[
   \int_\Omega b_0^\lambda\,\d x
   =\int_\Omega|\nabla u_0|^p\,\d x<+\infty.
\]
%It can require additional regularity when $1<\lambda<2$.

We verify that $u_0$ is the selected smooth-core pressure
for the coefficient $a_0$. Indeed,
\[
   a_0\in L^{p/(p-2)}(\Omega),
   \qquad a_0|\nabla u_0|^2=\frac1\sigma|\nabla u_0|^p
   \in L^1(\Omega).
\]
For every sequence $\varphi_n\in C_0^\infty(\Omega)$
converging to $u_0$ in $W_0^{1,p}(\Omega)$,
H\"older's inequality gives
\[
   \int_\Omega a_0|\nabla(\varphi_n-u_0)|^2\,\d x
   \leq\|a_0\|_{L^{p/(p-2)}(\Omega)}
       \|\nabla(\varphi_n-u_0)\|_{L^p(\Omega)}^2
   \longrightarrow0.
\]
Hence $u_0\in\Ha{a_0}$.
The weak formulation of \eqref{eq:p-laplace}, initially tested
against smooth functions, extends to $\Ha{a_0}$
by the weighted Cauchy--Schwarz inequality and the
weak-cutoff embedding. Thus $u_0=u[a_0]$.
Finally, \eqref{eq:concave-power-candidate} implies that
$b_0$ attains the supremum defining $g_\mu^*$ at
$|\nabla u_0|^2$ almost everywhere.
Proposition~\ref{prop:concave-stationary-certificate}
therefore shows that $b_0$ is the unique global minimizer of $\mathcal E_\mu^+$.
% whether the inverse-moment constraint is inactive or active.

By Proposition \ref{prop:H-convex-lsc}, $\mathcal H \geq 0$ is convex and lower semicontinuous.
Since $\Gamma$ is lso nonnegative, convex and lower semicontinuous,
$\mathcal E_\mu^+$ is proper, convex and lower semicontinuous on $L^2(\Omega)$.
It therefore generates a unique $L^2$-subgradient flow  for every initial datum in
$\overline{\operatorname{dom}\mathcal E_\mu^+}^{L^2}$, which contains $b_0$.
Global minimality gives $0\in \partial \mathcal E_\mu^+(b_0)$.
trict convexity of the power cost therefore makes $b_0$
the unique equilibrium of this abstract flow.

%%%%%%%%%%%%%%%%%%%%%%%%%%%%%%%%%%%%%%%%%%%%%%
\section{Reciprocal-concave mobility}
\label{sec:ReciprocalConcave}

The sufficient condition \eqref{eq:concavity-condition}, obtained from the
formal second-variation calculation, identifies concavity of $1/\mu$ as a
regime in which the pumping-energy functional has the opposite curvature from
that studied in Section~\ref{sec:ConcaveMobility}. We now place this
observation in the degenerate variational setting. The reciprocal coefficient
must possess sufficient integrability to define the weighted elliptic problem,
while the possibly singular Nemytskii map $b\mapsto\mu(b)$ must take admissible
coefficients into $L^2(\Omega)$. Assumption \textup{(R1)} below encodes the
curvature condition, and \textup{(R2)--(R3)} provide the required integrability.

The flux representation will yield concavity of the pumping energy and hence
convexity of the sign-reversed functional. In contrast to the preceding
sections, we do not develop an unrestricted $L^2$-gradient-flow theory in this
regime: continuity of the coefficient map does not imply continuity of the
selected weighted pressure or of its pumping energy. Instead, we establish a
direct criterion for global minimality of stationary states.
%requiring no continuity, subdifferential sum rule, or chain rule.

\subsection{Structural assumptions and admissible coefficients}

We begin with the curvature assumption
\begin{itemize}
\item[(R1)] $\mu:(0,+\infty)\to(0,+\infty)$ is of class $C^2$, and
   $h:=1/\mu$ is concave on $(0,+\infty)$.
\end{itemize}
Because $h$ is positive and concave on an unbounded interval, it cannot have a
negative derivative: otherwise concavity would force it to become negative at
large arguments. Consequently,
\begin{equation}\label{eq:reciprocal-monotonicity}
   h'=-\frac{\mu'}{\mu^2}\geq0,
   \qquad
   h''=\frac{2(\mu')^2-\mu\mu''}{\mu^3}\leq0.
\end{equation}
Thus $\mu'\leq0$ and $\mu''\geq0$. The endpoint limits
\begin{equation*}
   m_0:=\lim_{s\to 0^+}\mu(s)\in(0,+\infty],
   \qquad
   m_\infty:=\lim_{s\to+\infty}\mu(s)\in[0,m_0]
\end{equation*}
exist. Since $\mu$ is non-increasing, the equality $m_0=m_\infty$ occurs
if and only if $\mu$ is constant.

Assumption \textup{(R1)} controls the curvature and monotonicity of the
mobility, but it does not by itself guarantee that the elliptic problem is
weakly coercive. Since $h(b)=\mu(b)^{-1}$, the theory of
Section~\ref{sec:WeakCutoff} requires $h(b)\in L^{d/2}(\Omega)$. Every
positive concave function has an affine upper bound, which is sufficient for
this purpose when $d\leq4$. To cover all dimensions with a single hypothesis,
set
\begin{equation} \label{def:alpha}
   \alpha:=\min\left\{1,\frac4d\right\}
\end{equation}
and assume
\begin{itemize}
\item[(R2)] there are $C_0,C_1\geq0$ such that
\begin{equation}\label{eq:reciprocal-growth}
   h(s)=\frac1{\mu(s)}\leq C_0+C_1s^\alpha
   \qquad \mbox{for all } s>0.
\end{equation}
\end{itemize}
For $d\leq4$, this follows from (R1), with $\alpha=1$. Extend $h$ continuously
to zero by $h(0):=1/m_0$, with the convention $1/(+\infty)=0$. In every
dimension, \eqref{eq:reciprocal-growth} yields
\begin{equation}\label{eq:inverse-weight-bound}
   \int_\Omega h(b)^{d/2}\,\d x
   \leq C\bigl(1+\|b\|_{L^2(\Omega)}^2\bigr)
   \qquad \mbox{for } b\in L^2_+(\Omega),
\end{equation}
where the constant also depends on $|\Omega|$.

If $m_0<+\infty$, we extend $\mu$ continuously to $s=0$ and set
$\mathcal A:=L^2_+(\Omega)$. If $m_0=+\infty$, we additionally assume
\begin{itemize}
\item[(R3)] for some $E_0,E_1\geq0$ and $0<\beta\leq\alpha$,
\begin{equation*}
   \mu(s)\leq E_0+E_1s^{-\beta}\qquad \mbox{for all } s>0.
\end{equation*}
\end{itemize}
Choose $\gamma>2\beta$ and $\eta>0$, and define
\begin{equation}\label{def:reciprocal-admissible}
   \mathcal A:= %\mathcal A_{\eta,\gamma} :=
   \left\{b\in L^2_+(\Omega):
      \int_\Omega b^{-\gamma}\,\d x\leq\eta\right\}.
\end{equation}
Here and below we set $0^{-\gamma}:=+\infty$. The set $\mathcal A$ is closed and
convex in $L^2(\Omega)$ by the same integral-sublevel argument as in
Lemma~\ref{lem:KKcc}. The condition $\gamma>2\beta$ is chosen so
that the singular growth allowed in \textup{(R3)} yields uniform
integrability of $\mu(b_n)^2$ along convergent sequences in $\mathcal A$.

Thus \textup{(R2)} controls the reciprocal coefficient needed for weak
coercivity, whereas \textup{(R3)} controls the coefficient itself near the
singular endpoint $b=0$. The following proposition establishes the resulting
continuity properties.

\begin{proposition}\label{prop:reciprocal-nemytskii}
Under assumptions \textup{(R1)--(R3)}, with \textup{(R3)} required only when
$m_0=+\infty$, the maps
\begin{equation*}
   b\mapsto\mu(b):\mathcal A\to L^2_+(\Omega),
   \qquad
   b\mapsto\frac1{\mu(b)}:\mathcal A\to L^{d/2}(\Omega)
\end{equation*}
are continuous in the indicated norm topologies.
\end{proposition}

\begin{proof}
The second assertion is the standard continuity theorem for Nemytskii
operators, using \eqref{eq:reciprocal-growth} and
$\alpha d/2\leq2$. If $m_0<+\infty$, the non-increasing function $\mu$ is
bounded and continuous on $[0,+\infty)$, so the first assertion follows from
convergence in measure and dominated convergence.

If $m_0=+\infty$, let $b_n\to b$ in $L^2(\Omega)$ with
$b_n,b\in\mathcal A$. Every subsequence has a further subsequence converging
almost everywhere. Moreover,
\begin{equation*}
   \mu(b_n)^2\leq C\bigl(1+b_n^{-2\beta}\bigr),
   \qquad
   \sup_n\int_\Omega
      \bigl(b_n^{-2\beta}\bigr)^{\gamma/(2\beta)}\,\d x
   \leq\eta.
\end{equation*}
Since $\gamma/(2\beta)>1$, the de la Vall\'ee Poussin criterion gives uniform
integrability. Vitali's theorem then yields
$\mu(b_n)\to\mu(b)$ in $L^2(\Omega)$ along that further subsequence. The
subsequence principle gives convergence of the whole sequence.
\end{proof}

In what follows, Proposition~\ref{prop:reciprocal-nemytskii} together with
\eqref{eq:inverse-weight-bound} ensures that $a=\mu(b)\in L^2_+(\Omega)$ and
$a^{-1/2}\in L^d(\Omega)$, so that the weighted energy solution is well
defined by Section~\ref{sec:WeakCutoff}.

\subsection{Concavity of the pumping-energy functional}

For $b\in\mathcal A$, let $u[\mu(b)]$ be the energy solution of
\eqref{eq:Poisson_mub} and set
\begin{equation}\label{def:I-mu-reciprocal}
   I_\mu(b):=\G(\mu(b))
   =\int_\Omega\mu(b)|\nabla u[\mu(b)]|^2\,\d x.
\end{equation}
Thus $I_\mu=2J_\mu$ with the normalization in \eqref{eq:J-def}.
The functional $I_\mu$ is finite throughout $\mathcal A$ by the
weak-cutoff energy estimate. Its concavity follows directly from the
flux representation, Proposition~\ref{prop:flux-representation}.

\begin{proposition}\label{prop:I-mu-cont-concave}
Under \textup{(R1)--(R3)}, with \textup{(R3)} required only when
$m_0=+\infty$, the functional $I_\mu:\mathcal A\to\R$ is concave.
\end{proposition}

\begin{proof}
By Proposition~\ref{prop:flux-representation}, the functional
$I_\mu$ admits the representation
\begin{equation}\label{eq:flux-representation}
   I_\mu(b)=\inf_{j\in\mathcal Q_f}
      \int_\Omega\frac{|j|^2}{\mu(b)}\,\d x,
\end{equation}
where $\mathcal Q_f$ is defined in
Proposition~\ref{prop:flux-representation}.

Now let $b_t:=tb_1+(1-t)b_0$, where $b_0,b_1\in\mathcal A$ and
$t\in[0,1]$. For every $j\in\mathcal Q_f$, concavity of $h=1/\mu$
gives
\begin{align*}
   \int_\Omega |j|^2h(b_t)\,\d x
   &\geq t\int_\Omega |j|^2h(b_1)\,\d x
      +(1-t)\int_\Omega |j|^2h(b_0)\,\d x\\
   &\geq tI_\mu(b_1)+(1-t)I_\mu(b_0).
\end{align*}
Taking the infimum over $j\in\mathcal Q_f$ proves the claim.
\end{proof}

\begin{remark}
The continuity of the Nemytskii maps in
Proposition~\ref{prop:reciprocal-nemytskii} does not imply continuity of
$I_\mu$. For unbounded weights, a weak limit of selected weighted pressure
solutions may have finite weighted Dirichlet energy without belonging to
the smooth-core energy space for the limiting weight; compare
the weighted-space phenomena discussed in~\cite{Zhikov}. This distinction
can lead to a jump in $I_\mu$ even under strong convergence of the
coefficients. Accordingly, neither continuity of $I_\mu$ nor lower
semicontinuity of $-I_\mu$ is asserted here.
\end{remark}

\subsection{The sign-reversed convex energy and stationary minimizers}

The natural convex functional associated with concave $1/\mu$ is obtained
by changing the sign of the pumping term.
Let $g:[0,+\infty)\to[0,+\infty]$ be proper, convex and lower
semicontinuous, and define
\begin{equation}\label{def:reciprocal-energy}
   \mathcal E_\mu^-(b):=
   \begin{cases}
   \displaystyle
      -\frac12 I_\mu(b)+\int_\Omega g(b)\,\d x,
      &b\in\mathcal A,\\[2mm]
      +\infty,&b\notin\mathcal A.
   \end{cases}
\end{equation}
The difference in sign compared with the concave-mobility regime
reflects the curvature of $I_\mu$; the factor $1/2$ is consistent
with \eqref{eq:J-def}.

\begin{proposition}\label{prop:reciprocal-energy-convex}
Under assumptions \textup{(R1)--(R3)}, with \textup{(R3)} required only
when $m_0=+\infty$, the functional $\mathcal E_\mu^-$ defined in
\eqref{def:reciprocal-energy} is convex on $L^2(\Omega)$.
%In general it need not be lower semicontinuous.
\end{proposition}

\begin{proof}
By Proposition~\ref{prop:I-mu-cont-concave}, $-I_\mu$ is convex on
the convex set $\mathcal A$. The integral functional associated
with $g$ is convex, and extension by $+\infty$ outside $\mathcal A$ preserves convexity.
%Lower semicontinuity does not follow: although the integral functional is lower semicontinuous,
%the term $-I_\mu$ need not be, as discussed above.
\end{proof}

Without lower semicontinuity, convexity alone does not justify the
maximal-monotone construction of a global gradient flow. However,
global minimizers can be verified directly through the flux
representation. The following sufficient condition applies to arbitrary
convex metabolic costs and does not require the constraint to be
strictly inactive.

\begin{proposition}[Stationary flux certificate]
\label{prop:reciprocal-stationary-certificate}
Assume \textup{(R1)--(R3)}, with \textup{(R3)} required only when
$m_0=+\infty$, and let $g:[0,+\infty)\to[0,+\infty]$ be proper,
convex and lower semicontinuous. Suppose that $b_*\in\mathcal A$
is positive almost everywhere, that
\[\int_\Omega g(b_*)\,\d x<+\infty,\]
and that $u_*:=u[\mu(b_*)]$ is its selected smooth-core energy
solution. If
\begin{equation}\label{eq:reciprocal-stationary-certificate}
   -\frac12\mu'(b_*)|\nabla u_*|^2
   \in\partial g(b_*)\qquad\text{almost everywhere in }\Omega,
\end{equation}
%where $\partial g$ is the subdifferential of the scalar convex
%function $g$ (extended by $+\infty$ to negative arguments),
then $b_*$ is a global minimizer of $\mathcal E_\mu^-$ on
$L^2(\Omega)$. If $g$ is strictly convex, this minimizer is unique.
%No continuity of $I_\mu$, $L^2$-integrability of the density in
%\eqref{eq:reciprocal-stationary-certificate}, or subdifferential sum rule is required.
\end{proposition}

\begin{proof}
Put $a_*:=\mu(b_*)$, $h:=1/\mu$, and
$j_*:=a_*\nabla u_*$. The flux $j_*$ belongs to $\mathcal Q_f$
and is optimal for $b_*$ in \eqref{eq:flux-representation}; thus
\begin{equation}\label{eq:reciprocal-reference-flux}
   I_\mu(b_*)
   =\int_\Omega |j_*|^2h(b_*)\,\d x
   =\int_\Omega a_*|\nabla u_*|^2\,\d x<+\infty.
\end{equation}
Since $h'=-\mu'/\mu^2$, condition
\eqref{eq:reciprocal-stationary-certificate} reads
\[
   \frac12 h'(b_*)|j_*|^2\in\partial g(b_*)
   \qquad\text{a.e. in }\Omega.
\]
By convexity of $g$ and concavity of $h$, for almost every $x\in\Omega$
and every $v\geq0$ in the effective domain of $g$,
\begin{align*}
   g(v)&\geq g(b_*)
      +\frac12h'(b_*)|j_*|^2(v-b_*),\\
   h(v)&\leq h(b_*)+h'(b_*)(v-b_*).
\end{align*}
Consequently,
\begin{equation}\label{eq:reciprocal-pointwise-calibration}
   g(v)-\frac12|j_*|^2h(v)
   \geq g(b_*)-\frac12|j_*|^2h(b_*).
\end{equation}

Take $b\in\mathcal A$ with $\int_\Omega g(b)\,\d x<+\infty$.
Applying \eqref{eq:reciprocal-pointwise-calibration} with $v=b(x)$
gives
\[
   0\leq |j_*|^2h(b)
   \leq 2g(b)-2g(b_*)+|j_*|^2h(b_*).
\]
The right-hand side is integrable by the assumptions and
\eqref{eq:reciprocal-reference-flux}. Hence $j_*$ has finite
flux energy also for the coefficient $\mu(b)$, and the flux
representation yields
\begin{align*}
   \mathcal E_\mu^-(b)
   &=\int_\Omega g(b)\,\d x-\frac12I_\mu(b)\\
   &\geq\int_\Omega
      \left(g(b)-\frac12|j_*|^2h(b)\right)\,\d x\\
   &\geq\int_\Omega
      \left(g(b_*)-\frac12|j_*|^2h(b_*)\right)\,\d x
    =\mathcal E_\mu^-(b_*).
\end{align*}
The inequality is trivial for competitors of infinite metabolic
energy or outside $\mathcal A$. Thus $b_*$ is a global minimizer.
If $g$ is strictly convex, the metabolic integral is strictly
convex on its effective domain. Together with convexity of
$-I_\mu$, this proves uniqueness.
\end{proof}

\iffalse
\begin{remark}
When $g$ is differentiable at $b_*>0$, condition
\eqref{eq:reciprocal-stationary-certificate} becomes
\[
   g'(b_*)=-\frac12\mu'(b_*)|\nabla u_*|^2.
\]
The certificate is sufficient even if the inverse-moment
constraint is active, provided the pointwise relation holds;
it corresponds to a stationary pair with zero constraint
multiplier. It does not characterize minimizers with a
nonzero multiplier, nor does it assert that an admissible
stationary pair exists. In particular, no global gradient-flow
conclusion follows from this proposition.
\end{remark}
\fi

\subsection{Inverse-power mobilities and the subquadratic regime}

The inverse-power family provides the subquadratic counterpart of the
positive-power example in Section~\ref{sec:ConcaveMobility}. Let
\begin{equation}\label{eq:reciprocal-power-mobility}
   \mu(s)=\frac1\kappa s^{-\kappa},
   \qquad g(s)=\frac1\lambda s^\lambda,
   \qquad 0<\kappa\leq\alpha,\quad
          \lambda\geq1,\quad\lambda>\kappa,
\end{equation}
with $\alpha\leq1$ given by \eqref{def:alpha}.
Then $h(s)=\kappa s^\kappa$ is concave and assumptions
\textup{(R1)--(R3)} hold with $\beta=\kappa$ and with
the admissible set $\mathcal A$ from
\eqref{def:reciprocal-admissible} for any $\gamma>2\kappa$.

Suppose that $b_*\in\mathcal A$ and
$u_*:=u[\mu(b_*)]$ satisfy the pointwise stationary relation
\begin{equation}\label{eq:reciprocal-power-stationarity}
   b_*^{\lambda+\kappa}=\frac12|\nabla u_*|^2
   \qquad\text{a.e. in }\Omega.
\end{equation}
This is precisely the differentiable form of
\eqref{eq:reciprocal-stationary-certificate}, since
$\mu'(s)=-s^{-\kappa-1}$ and $g'(s)=s^{\lambda-1}$.
Furthermore,
\[
   \mu(b_*)|\nabla u_*|^2=\frac2\kappa b_*^\lambda,
\]
so the weighted energy identity implies
\begin{equation}\label{eq:reciprocal-power-metabolic}
   \int_\Omega b_*^\lambda\,\d x
   =\frac\kappa2 I_\mu(b_*)<+\infty.
\end{equation}
Thus the finite-metabolic-energy hypothesis in
Proposition~\ref{prop:reciprocal-stationary-certificate} holds
automatically. That proposition shows that $b_*$ is a global
minimizer of $\mathcal E_\mu^-$, unique when $\lambda>1$.
No strict inactivity of the inverse-moment constraint is needed.

The stationary relation gives
\begin{equation}\label{eq:reciprocal-power-candidate}
   b_*=2^{-1/(\lambda+\kappa)}
      |\nabla u_*|^{2/(\lambda+\kappa)},
   \qquad
   \mu(b_*)=c_{\kappa,\lambda}|\nabla u_*|^{p-2},
\end{equation}
where
\[
   c_{\kappa,\lambda}:=\frac{2^{\kappa/(\lambda+\kappa)}}{\kappa},
   \qquad p:=\frac{2\lambda}{\lambda+\kappa}.
\]
Since $u_*$ is the selected energy solution, it follows that
\begin{equation}\label{eq:reciprocal-power-p-laplace}
   -c_{\kappa,\lambda}\operatorname{div}
   \bigl(|\nabla u_*|^{p-2}\nabla u_*\bigr)=f
   \qquad\text{in }\mathcal D'(\Omega),
\end{equation}
with homogeneous Dirichlet boundary condition in the trace sense.
Here
\[
   1<p<2 \quad\text{for }d\in\{2,3,4\},
   \qquad \frac{2d}{d+4}\leq p<2
   \quad\text{for }d\geq5.
\]
Moreover, \eqref{eq:reciprocal-power-metabolic} shows that
$\nabla u_*\in L^p(\Omega)$. The weak-cutoff zero-trace property
then implies $u_*\in W_0^{1,p}(\Omega)$.
For exponents $p<2d/(d+2)$, a general source
$f\in L^2(\Omega)$ need not define an element of
$W^{-1,p'}(\Omega)$; nevertheless, for the stationary pairs
considered here, equation \eqref{eq:reciprocal-power-p-laplace}
itself supplies the required dual-space compatibility.

Conversely, solving the subquadratic $p$-Laplacian does not by
itself produce an admissible stationary pair. Starting from a
solution $u_*\in W_0^{1,p}(\Omega)$ of
\eqref{eq:reciprocal-power-p-laplace}, the candidate
\eqref{eq:reciprocal-power-candidate} must satisfy
\begin{equation}\label{eq:reciprocal-power-admissibility}
   b_*\in L^2(\Omega),\qquad
   \int_\Omega b_*^{-\gamma}\,\d x\leq\eta,
   \qquad \gamma>2\kappa,
\end{equation}
and $u_*$ must belong to the smooth-core weighted space
$\Ha{\mu(b_*)}$. These requirements are not automatic from
the $p$-energy. If they hold, the distributional equation
extends to the weighted energy space, so $u_*=u[\mu(b_*)]$;
Proposition~\ref{prop:reciprocal-stationary-certificate} then
gives global minimality, and uniqueness when $\lambda>1$.
A sufficient condition for the core-membership requirement is
$u_*\in W_0^{1,4}(\Omega)$: since $\mu(b_*)\in L^2(\Omega)$,
smooth approximation in $W_0^{1,4}(\Omega)$ converges also in
the weighted energy norm.

%The reduction to \eqref{eq:reciprocal-power-p-laplace} therefore identifies the stationary equation but does not establish
%existence of an admissible pair for arbitrary sources and parameters. In particular, no existence or convergence theorem
%for an unrestricted reciprocal-mobility gradient flow is claimed.

\begin{remark}[Gradient flows under uniform ellipticity]
Although the functional $\mathcal E_\mu^-$ need not be
lower semicontinuous on the unrestricted admissible set
$\mathcal A$, a rigorous $L^2$-gradient-flow theory can
be recovered by imposing uniform bounds on the coefficient.
Indeed, for fixed $0<\varepsilon<M<+\infty$, consider
the closed convex set
\[
   \mathcal A_{\varepsilon,M}
   :=\{b\in L^2(\Omega):
       \varepsilon\leq b\leq M\text{ a.e. in }\Omega\}.
\]
Since $\mu$ is positive and continuous on
$[\varepsilon,M]$, the coefficients $\mu(b)$ are
uniformly bounded and uniformly elliptic.
Consequently, if $b_n\to b$ in $L^2(\Omega)$ with
$b_n,b\in\mathcal A_{\varepsilon,M}$, standard
elliptic stability yields
\[
   u[\mu(b_n)]\to u[\mu(b)]
   \quad\text{strongly in }H_0^1(\Omega),
\]
and hence $I_\mu(b_n)\to I_\mu(b)$.

Thus, if $g$ is finite and convex on
$[\varepsilon,M]$, the restricted functional
\[
   \mathcal E_{\mu,\varepsilon,M}^-(b)
   :=-\frac12 I_\mu(b)+\Gamma(b)
     +\iota_{\mathcal A_{\varepsilon,M}}(b)
\]
is proper, convex and lower semicontinuous on
$L^2(\Omega)$, and therefore generates a unique
$L^2$-subgradient flow. The box constraint contributes
its normal cone to the evolution, leading to
projections at the lower and upper bounds whenever
a pointwise identification is justified.

Related uniformly elliptic auxiliary-variable
gradient flows, formulated in an adapted metric,
are studied in~\cite{ZampiniBoffiDovletovMarkowich2026}.
The box restriction avoids the weighted-space
instability discussed above, but excludes genuinely
degenerate conductivities. We therefore do not
pursue this restricted flow here.
\end{remark}
\vspace{1cm}

\section*{Declaration of Generative AI Use}

During the preparation of this manuscript, the authors used Claude Opus 5.5 and Fable 5.1, and ChatGPT Astra 6.0 and Sol 6.1 exclusively for assistance with language editing, improving the clarity and presentation of the manuscript, LaTeX transcription, and checking the consistency and correctness of mathematical arguments. All original ideas, mathematical results, proofs, and computations are the authors' own. All AI-assisted suggestions were critically reviewed and verified by the authors, who take full responsibility for the content of the manuscript.
\vspace{1cm}

%%%%%%%%%%%%%%%%%%%%%%%%%%%%%%%%%%%%%%%%%%%%%%

\end{document}